\documentclass[a4paper,10pt]{amsart}
\usepackage[utf8]{inputenc}
\usepackage{amssymb,amsmath,dsfont,tikz-cd}
\usepackage[all]{xy}
\usepackage{upgreek}
\usepackage{colonequals}

\renewcommand{\theta}{\uptheta}
\renewcommand{\alpha}{\upalpha}
\renewcommand{\beta}{\upbeta}
\renewcommand{\gamma}{\upgamma}
\renewcommand{\delta}{\updelta}
\renewcommand{\zeta}{\upzeta}
\renewcommand{\pi}{\uppi\hspace{0.05em}}
\renewcommand{\rho}{\uprho}
\renewcommand{\xi}{\upxi}
\renewcommand{\chi}{\upchi}
\renewcommand{\sigma}{\upsigma}
\renewcommand{\Lambda}{\Uplambda}
\renewcommand{\Gamma}{\Upgamma}
\renewcommand{\phi}{\upphi}
\renewcommand{\psi}{\uppsi}
\renewcommand{\nu}{\upnu}
\renewcommand{\tau}{\uptau}
\renewcommand{\mu}{\upmu}
\renewcommand{\eta}{\upeta}

\newtheorem{theorem}{Theorem}[section]
\newtheorem{thmx}{Theorem}

\newtheorem{definition}[theorem]{Definition}
\newtheorem{assumption}[theorem]{Assumption}
\newtheorem{example}[theorem]{Example}
\newtheorem{warning}[theorem]{Warning}
\newtheorem{proposition}[theorem]{Proposition}
\newtheorem{lemma}[theorem]{Lemma}
\newtheorem{remark}[theorem]{Remark}

\newtheorem{corollary}[theorem]{Corollary}

\newcommand{\ab}{\mathtt{abs}}
\renewcommand{\AA}{\mathbb{A}}

\DeclareMathOperator{\BC}{B\mathbb{C}^*}
\newcommand{\Boxtimes}{\mbox{\larger[3]{$\boxtimes$}}}

\newcommand{\DD}{\mathbb{D}}

\newcommand{\eu}{\mathfrak{eu}}
\DeclareMathOperator{\EXP}{EXP}
\DeclareMathOperator{\Pf}{\mathfrak{P}}

\newcommand{\JJ}{\mathbb{J}}
\newcommand{\ff}{\mathbf{f}}
\newcommand{\GG}{\mathbb{G}}
\DeclareMathOperator{\HN}{HN}

\newcommand{\IS}{\mathcal{P}}

\newcommand{\LL}{\mathbb{L}}
\newcommand{\NN}{\mathbb{N}}

\newcommand{\QQ}{\mathbb{Q}}
\newcommand{\SP}{\mathcal{S}}
\newcommand{\Mu}{\Xi}

\newcommand{\TS}{\mathtt{TS}}

\newcommand{\W}{\mathrm{Tr}(W)}
\newcommand{\WW}{\mathcal{T}r(W)}
\newcommand{\WWW}{\mathfrak{Tr}(W)}

\newcommand{\ZZ}{\mathbb{Z}}
\newcommand{\Mst}{\mathfrak{M}}

\newcommand{\Msp}{\mathcal{M}}

\newcommand{\ICS}{\mathcal{IC}}
\newcommand{\ICSn}{\tilde{\mathcal{IC}}}
\newcommand{\ICSt}{\mathfrak{IC}}
\newcommand{\DTS}{\mathcal{BPS}}
\DeclareMathOperator{\obj}{ob}

\newcommand{\phim}[1]{\phi^{\mon}_{#1}}
\newcommand{\GrW}[1]{\Gr^W_{#1}}
\DeclareMathOperator{\sst}{-ss}
\DeclareMathOperator{\st}{-st}

\DeclareMathOperator{\tw}{tw}
\DeclareMathOperator{\reg}{reg}
\DeclareMathOperator{\crit}{crit}
\DeclareMathOperator{\vect}{vect}
\DeclareMathOperator{\Free}{Free}
\DeclareMathOperator{\FreeComm}{Sym}
\DeclareMathOperator{\Hom}{Hom}

\DeclareMathOperator{\mon}{mon}

\DeclareMathOperator{\MMHM}{MMHM}
\DeclareMathOperator{\MMHS}{MMHS}
\DeclareMathOperator{\Dim}{\mathrm{Dim}}
\DeclareMathOperator{\rank}{rank}

\DeclareMathOperator{\fd}{f.d.}
\DeclareMathOperator{\Fun}{Fun}
\DeclareMathOperator{\Gr}{Gr}
\DeclareMathOperator{\Vect}{Vect}

\DeclareMathOperator{\Ka}{K}

\DeclareMathOperator{\rat}{rat}
\DeclareMathOperator{\ms}{\tilde{\ast}}
\DeclareMathOperator{\mos}{\tilde{\cdot}}

\DeclareMathOperator{\cc}{\mathbf{c}}
\DeclareMathOperator{\dd}{\mathbf{d}}
\DeclareMathOperator{\ee}{\mathbf{e}}
\DeclareMathOperator{\perv}{Perv}
\DeclareMathOperator{\con}{con}
\DeclareMathOperator{\Aut}{Aut}
\DeclareMathOperator{\prop}{Sp}
\DeclareMathOperator{\codim}{codim}
\DeclareMathOperator{\Perv}{Perv}
\DeclareMathOperator{\MHM}{MHM}
\DeclareMathOperator{\IC}{IC}
\DeclareMathOperator{\DT}{BPS}
\DeclareMathOperator{\Sym}{Sym}
\DeclareMathOperator{\SSym}{\mathfrak{S}}

\DeclareMathOperator{\red}{red}

\newcommand{\Sch}[1]{\sch_{#1}}
\DeclareMathOperator{\sch}{Sch}

\DeclareMathOperator{\Spec}{Spec}
\DeclareMathOperator{\Gl}{GL}
\DeclareMathOperator{\id}{id}
\DeclareMathOperator{\Jac}{Jac}

\DeclareMathOperator{\Tr}{Tr}

\DeclareMathOperator{\pt}{pt}

\DeclareMathOperator{\tot}{t}

\DeclareMathOperator{\cone}{cone}

\DeclareMathOperator{\forg}{forg}
\newcommand{\form}[1]{\forg^{\mon}_{#1}}
\DeclareMathOperator{\vir}{vir}
\DeclareMathOperator{\Ob}{Ob}
\DeclareMathOperator{\Ho}{\mathcal{H}}
\DeclareMathOperator{\HO}{H}

\DeclareMathOperator{\pr}{pr}

\DeclareMathOperator{\Coha}{\mathcal{A}}
\DeclareMathOperator{\FCoha}{\mathcal{F}}

\newcommand{\Du}{\mathcal{D}^{\geq}}
\newcommand{\Dl}{\mathcal{D}^{\leq}}
\newcommand{\Dulf}{\mathcal{D}^{\geq, \mathrm{lf}}}

\newcommand{\Db}{\mathcal{D}^{b}}
\newcommand{\Dub}{\mathcal{D}}
\newcommand{\Dp}{{}^{\mathfrak{p}}\mathcal{D}}

\title[CoHAs and quantum enveloping algebras]{Cohomological Donaldson--Thomas theory of a quiver with potential and quantum enveloping algebras}
\author{Ben Davison \and Sven Meinhardt}

\begin{document}

\begin{abstract}
This paper concerns the cohomological aspects of Donaldson--Thomas theory for Jacobi algebras and the associated cohomological Hall algebra, introduced by Kontsevich and Soibelman.  We prove the Hodge-theoretic categorification of the integrality conjecture and the wall crossing formula, and furthermore realise the isomorphism in both of these theorems as Poincar\'e-Birkhoff-Witt isomorphisms for the associated cohomological Hall algebra.  

We do this by defining a perverse filtration on the cohomological Hall algebra, a result of the ``hidden properness'' of the semisimplification map from the moduli stack of semistable representations of the Jacobi algebra to the coarse moduli space of polystable representations.  

This enables us to construct a degeneration of the cohomological Hall algebra, for generic stability condition and fixed slope, to a free supercommutative algebra generated by a mixed Hodge structure categorifying the BPS invariants.  As a corollary of this construction we furthermore obtain a Lie algebra structure on this mixed Hodge structure --- the Lie algebra of BPS invariants --- for which the entire cohomological Hall algebra can be seen as the positive part of a Yangian-type quantum group.
\end{abstract}

\maketitle

\tableofcontents

\section{Introduction}
\subsection{Background}
The critical cohomological Hall algebra $\HO(\mathcal{A}_{B,W})$ associated to a noncommutative smooth algebra $B$ and potential $W\in B/[B,B]_{\vect}$ was introduced by Kontsevich and Soibelman in \cite{KS2} as a way of categorifying the theory of Donaldson--Thomas invariants.  The meaning of ``critical'' here is that the Hall algebra is built out of the hypercohomology of the sheaf of \textit{vanishing cycles} of the superpotential $\Tr(W)$ on the stack of finite-dimensional $B$-modules.  This sheaf is supported on the critical locus of $\Tr(W)$, which is identified with the stack of representations of the Jacobi algebra associated to $B$ and $W$.  The critical cohomological Hall algebra can thus be thought of as the Hall algebra categorifying the Donaldson--Thomas theory of representations of this Jacobi algebra.
\smallbreak
Donaldson--Thomas invariants, first introduced in \cite{Thomas1}, already have an extensive literature.  We refer to the sequence of papers by Dominic Joyce \cite{JoyceI,JoyceII,JoyceIII,JoyceIV,JoyceMF} for a comprehensive account, or the book \cite{JoyceDT} by Joyce and Song, and also \cite{KS1,KS3} for the more general and abstract account by Kontsevich and Soibelman using motivic vanishing cycles.  In all of these treatments of Donaldson--Thomas theory, a key role is played by the integrality conjecture, as proven in \cite{KS2} in the case of quivers with potential.  In this paper we describe and prove the natural categorification of the integrality conjecture in the context of the cohomological Hall algebra for a quiver with potential, and explain how this gives rise to a mathematical definition of the Lie algebra of BPS states, for which the cohomological Hall algebra can be thought of as the positive half of an associated quantum group.  See \cite{HM98} for a Physics-oriented construction of Hall algebras of BPS states.
\smallbreak
From now on we assume that $B=\mathbb{C} Q$ is the free path algebra of a quiver, that we are given a Bridgeland stability condition $\zeta$ on the category of $\mathbb{C} Q$-modules, and that $\mu\in(-\infty,\infty]$ is a slope.  We denote by $\Lambda_{\mu}^{\zeta}\subset\mathbb{N}^{Q_0}$ the union of the zero vector with the subset of dimension vectors of slope $\mu$ with respect to the stability condition $\zeta$.  Let $\Mst^{\zeta\sst}_{\mu}$ denote the stack of $\zeta$-semistable $\mathbb{C}Q$-modules with dimension vector in $\Lambda_{\mu}^{\zeta}$.  For full generality, we consider a locally closed substack $\Mst^{\zeta\sst,\SP}_{\mu}\subset \Mst^{\zeta\sst}_{\mu}$, the closed points of which correspond to $\zeta$-semistable $\mathbb{C}Q$-modules in a Serre subcategory $\mathcal{S}$ of the category of $\mathbb{C}Q$-modules.  When $\mathcal{S}$ is the entire full subcategory of finite dimensional $\mathbb{C}Q$-modules, we will omit all superscripts $\SP$.  
\smallbreak
Via natural correspondences there is an algebra structure on the $\Lambda_{\mu}^{\zeta}$-graded Borel--Moore homology of the normalised vanishing cycle complex on $\Mst^{\zeta\sst}_{\mu}$, restricted to $\Mst^{\zeta\sst,\SP}_{\mu}$, that we denote 
\begin{equation}
\label{ua}
\HO\left(\Coha_{W,\mu}^{\zeta,\SP}\right)\colonequals  \bigoplus_{\dd\in\Lambda_{\mu}^{\zeta}}\HO_{c}\left(\Mst^{\zeta\sst,\SP}_{\dd},\phi_{\WWW^{\zeta}_{\dd}}\QQ_{\Mst^{\zeta\sst}_{\dd}}\otimes\LL^{-\dim(\Mst^{\zeta\sst}_{\dd})/2}\right)^{\vee}
\end{equation}
as in \cite{KS2}.  These notions are properly defined and explained in Section \ref{Qrepsection}.  This algebra categorifies Donaldson--Thomas theory in the sense that Donaldson--Thomas invariants for the category of $\zeta$-semistable modules in $\mathcal{S}$ are obtained by considering the class in the Grothendieck group of the $\Lambda_{\mu}^{\zeta}$-graded mixed Hodge structures (equipped with a monodromy action) of the underlying mixed Hodge structure of (\ref{ua}).
\smallbreak 
The main conclusion of this paper is that the introduction of the cohomological Hall algebra structure associated to a quiver, a potential and a Serre subcategory $\mathcal{S}$ should be seen as a first step towards forging a connection between the theory of refined Donaldson--Thomas invariants and the theory of quantum enveloping algebras, and that the proof of the integrality conjecture is the motivic shadow of a Poincar\'e--Birkhoff--Witt theorem for the associated quantum enveloping algebra.  In fact we reach this conclusion by providing all of the remaining steps, proving this PBW type theorem, and defining the associated Lie algebra of BPS cohomology.  This programme was already articulated in \cite{KS2} but several components had yet to be found.  
\smallbreak
Firstly, a coproduct turning the algebra $\HO(\Coha_{W,\mu}^{\zeta,\SP})$ into a Hopf algebra had to be defined.  A localised coproduct was found in the first author's paper \cite{Da13}, reinforcing the hope that $\HO(\Coha_{W,\mu}^{\zeta,\SP})$ could be turned into a quantum enveloping algebra.  Meanwhile the integrality conjecture in the case of free path algebras of quivers with zero potential was reproven by the second author and Markus Reineke in \cite{Meinhardt14}.  We recall that, loosely speaking, the integrality theorem states that the class of $\HO(\Coha_{0,\mu}^{\zeta,\SP})$ in the Grothendieck group of graded monodromic mixed Hodge structures is equal to that of the free supercommutative algebra generated by $\Sym\left(\HO(\BC)_{\vir}\otimes\Omega\right)$ for \textit{some} element $\Omega$ for which the graded pieces are finite-dimensional.  This is of course much weaker than the claim that we have an isomorphism between these two objects, which in turn is weaker than the existence of a PBW isomorphism.  
\smallbreak
Nevertheless, the proof of \cite{Meinhardt14} has the attractive feature of explicitly constructing $\Omega$, in terms of intersection complexes.  This at least gives us a natural candidate for a choice of $\Omega$ with which to try to upgrade equality in the Grothendieck group to a bona fide isomorphism.  In addition, it is proven in \cite{DaMe4} that due to formal properties of the vanishing cycle functor, the methods of \cite{Meinhardt14} are enough to prove a very general version of the integrality conjecture, even in the presence of a nonzero potential.  Similarly, the proof in \cite{DaMe4} gives explicit formulae for the class of $\Omega$, and the construction lifts to cohomology, or mixed Hodge modules, which is the lift that we develop in this paper.  The claim that there is an isomorphism in the category of $\Lambda_{\mu}^{\zeta}$-graded monodromic mixed Hodge structures between $\HO(\Coha_{W,\mu}^{\zeta,\SP})$ and $\Sym\left(\HO(\BC)_{\vir}\otimes\Omega\right)$ for some $\Omega$ depending on $W,\mu,\zeta,\SP$ is what we call the ``cohomological integrality conjecture''.  The proof of the cohomological integrality conjecture, and the existence of a coproduct, are the main ingredients of our proof of the PBW theorem.
\subsection{Absolute versus relative Donaldson--Thomas theory}
\label{abvsrel}
For the Donaldson--Thomas theory and cohomological Hall algebras associated to quivers, a vital role is always played by the grading by $\mathbb{N}^{Q_0}$, the semigroup of dimension vectors, which should be seen as a weight space decomposition in the language of Lie algebras and quantum groups.  As an illustration, consider the case in which we forget about potentials and stability conditions, set $\mathcal{S}$ to be the category of all finite-dimensional $\mathbb{C}Q$-modules, and consider the $\mathbb{N}^{Q_0}$-graded vector space
\begin{equation}
\label{baby}
\bigoplus_{\dd\in\mathbb{N}^{Q_0}}\HO(\Mst_{\dd},\mathbb{Q}[\dim(\Mst_{\dd})]).
\end{equation}
Considering the resulting cohomology as a $\mathbb{N}^{Q_0}\times\mathbb{Z}$-graded vector space, where the extra $\mathbb{Z}$ factor keeps track of cohomological degree, each graded piece is finite-dimensional.  The initial role of the $\NN^{Q_0}$ grading is to break the above infinite-dimensional vector space into infinitely many manageable, finite-dimensional pieces.  Moreover, the cohomological Hall algebra structure respects the $\mathbb{N}^{Q_0}\times\ZZ$-grading, so that this grading provides a route to understanding this algebra.
\smallbreak
Choosing a bijection between the vertices $Q_0$ and the numbers $\{1,\ldots,n\}$, motivic DT partition functions are expressed as formal power series in variables $q^{1/2},x_1,\ldots,x_{n}$.  For $\dd=(\dd_1,\ldots,\dd_n)$  the coefficient of 
\[
x^{\dd}\colonequals \prod_{i\in Q_0}x_i^{\dd_i}
\]
is a formal Laurent power series in $q^{1/2}$, given by the characteristic function of the $\dd$th graded piece of the right hand side of (\ref{baby}) --- for now $q^{1/2}$ records the cohomological degree.  These power series may have finite order poles at $q^{1/2}=0$, due to the shifts appearing in (\ref{baby}).  In the general case this formal Laurent power series is replaced by a formal Laurent power series recording the finite dimensions of the weight filtration of an infinite-dimensional $\mathbb{N}^{Q_0}$-graded (monodromic) mixed Hodge structure associated to the quiver $Q$, a potential $W$ and a stability condition $\zeta$ --- in the final version of the theory, $q^{1/2}$ records weights.  

\smallbreak
In order to make the presentation in this paper conceptually uniform, we slightly change the perspective on the $\mathbb{N}^{Q_0}$-grading.  We consider $\mathbb{N}^{Q_0}$-graded vector spaces (respectively, mixed Hodge structures) as sheaves of vector spaces (respectively, mixed Hodge modules) on the monoid $\mathbb{N}^{Q_0}$, considered as a scheme with infinitely many components, each isomorphic to $\Spec(\mathbb{C})$.  Then our main object of study (\ref{ua}) can be thought of as the direct image of the restriction to the stack of modules in $\SP$ of
\begin{equation}
\label{ICSTintro}
\ICSt_{W,\mu}^{\zeta}\colonequals \bigoplus_{\dd\in\Lambda_{\mu}^{\zeta}}\LL^{-\dim(\Mst_{\dd}^{\zeta\sst})/2}\otimes\phi_{\WWW^{\zeta}_{\mu}}\QQ_{\Mst^{\zeta\sst,\SP}_{\dd}}
\end{equation}
along the map
\[
\Dim\colon\Mst^{\zeta\sst}_{\mu}\rightarrow \Lambda_{\mu}^{\zeta}
\]
taking each component of the moduli stack to the dimension vector of the modules it parametrises.
Instead of studying this direct image directly, we study the direct image
\begin{equation}
\label{intermCo}
\Ho(\mathcal{A}_{W,\mu}^{\zeta})\colonequals \Ho\left((\Mst^{\zeta\sst}_{\mu}\xrightarrow{p^{\zeta}_{\mu}}\Msp^{\zeta\sst}_{\mu})_*\ICSt_{W,\mu}^{\zeta}\right)
\end{equation}
where the map $p^{\zeta}_{\mu}$ takes a $\mathbb{C}Q$-module to its semisimplification in the category of $\zeta$-semistable $\mathbb{C}Q$-modules, i.e. the associated polystable object.  Since $\Dim$ factors through $p^{\zeta}_{\mu}$, understanding the direct image along $p^{\zeta}_{\mu}$ turns out to be the key to understanding the direct image along $\Dim$.
\smallbreak
We flag at the outset the curious role of the symbol $\Ho$ on the right hand side of \eqref{intermCo}.  Firstly, its purpose is to make our lives simpler: we will not actually give a definition of the pushforward of $\ICSt_{W,\mu}^{\zeta}$, nor indeed of $\ICSt_{W,\mu}^{\zeta}$ itself, and we \textit{only} define the result of applying the total cohomology functor to this pushforward, i.e. we define the right hand side of \eqref{intermCo} without defining its constituent pieces.  This apparent bug turns out to be a feature; because of this extra $\Ho$, taking the dual compactly supported hypercohomology of the right hand side of \eqref{intermCo} does not produce something that is a priori isomorphic to the right hand side of \eqref{ua}, and the comparison between them turns out to be a central aspect of the paper.  In brief: the compactly supported hypercohomology of the right hand side of \eqref{intermCo} is the associated graded of the perverse filtration of $\HO\left(\Coha_{W,\mu}^{\zeta,\SP}\right)$.  This perverse filtration is one of the central structures that we introduce in order to understand the cohomological Hall algebra.

Here and elsewhere, if $\mathcal{L}\in\mathcal{D}$ is an element of a triangulated category with a t-structure (for example the bounded derived category of an Abelian category $\mathcal{A}$, such as the category of monodromic mixed Hodge modules on $\Msp_{\mu}^{\zeta\sst}$) then 
\[
\Ho(\mathcal{L})\colonequals \bigoplus_{n\in\mathbb{Z}}\mathcal{H}^i(\mathcal{L})[-i],
\]
considered as an element in $\mathcal{D}$ with zero differential.  If moreover $\mathcal{D}$ has infinite coproducts (for example, the unbounded derived category of $\mathcal{A}$ as above) then the natural functor from $\mathbb{Z}$-graded objects in $\mathcal{A}$ to the essential image of the functor $\Ho$ is an embedding of categories; it is bijective on isomorphism classes of objects and faithful, but not full in general.

Note that $\Dim^{\zeta}_{\mu}=\dim^{\zeta}_{\mu}p^{\zeta}_{\mu}$, where $\dim^{\zeta}_{\mu}\colon\Msp_{\mu}^{\zeta}\rightarrow\Lambda_{\mu}^{\zeta}$ is the natural map.  We furthermore prove that there is an isomorphism of cohomologically graded monodromic mixed Hodge modules
\begin{align}
\label{fullCo}
\Ho\left(\Dim^{\zeta}_{\mu,*}\ICSt_{W,\mu}^{\zeta}\right)\cong &\Ho\left(\dim_{\mu,*}^{\zeta}\Ho\left((\Mst^{\zeta\sst}_{\mu}\xrightarrow{p^{\zeta}_{\mu}}\Msp^{\zeta\sst}_{\mu})_*\ICSt_{W,\mu}^{\zeta}\right)\right)\\
=&\Ho\left(\dim_{\mu,*}^{\zeta}\Ho\left(\mathcal{A}_{W,\mu}^{\zeta}\right)\right)\nonumber
\end{align}
in Section \ref{CoDTSec}, so that an understanding of (\ref{intermCo}) can indeed be used to gain an understanding of the left hand side of (\ref{fullCo}).  As noted above, because of the second $\Ho$ on the right hand side of \eqref{fullCo}, this isomorphism would not follow directly from the usual six functor formalism for monodromic mixed Hodge modules on stacks, but instead relies essentially on the ``hidden properness'' of the morphism $p^{\zeta}_{\mu}$.    The key observation is that $p^{\zeta}_{\mu}$, while not proper or representable, is in some sense approximated by proper maps of schemes (see Section \ref{cohaDT}).  
\smallbreak
Working over the base $\Msp^{\zeta\sst}_{\mu}$ instead of $\Lambda_{\mu}^{\zeta}$ is what we mean by the relative theory.  The payoff for working in the relative setting is that we obtain more refined results, which are nevertheless easier to prove, since in this relative setting, because of hidden properness, we are able to make extensive use of Morihiko Saito's version \cite{Sai88,Saito1,Saito89} of the famous decomposition theorem of Beilinson, Bernstein, Deligne and Gabber \cite{BBD}, as well as the rest of Saito's theory of mixed Hodge modules.
\smallbreak
The idea of proving local results on $\Msp_{\mu}^{\zeta\sst}$ to deduce global results about Donaldson--Thomas invariants was already present in the work of Joyce and Song \cite{JoyceDT}, and later the proof of the integrality conjecture for path algebras in the absence of a potential by the second author and Markus Reineke \cite{Meinhardt14}, which utilised Efimov's theorem \cite{Efimov} on free supercommutativity for cohomological Hall algebras without potentials to understand the local behaviour of the morphism $p^{\zeta}_{\mu}$. 
\subsection{Cohomological integrality}
Now that we have introduced the relative version of Donaldson--Thomas theory, we can state the first main result of this paper, a cohomological refinement of the integrality conjecture.
\begin{thmx}[Cohomological integrality theorem]
\label{ThmA}
Let $Q,W$ be a quiver with potential, let $\mu\in (-\infty,\infty]$ be a slope, and let $\zeta$ be a $\mu$-generic stability condition.  Let $\Mst^{\zeta\sst,\SP}_{\mu}$ be a locally closed substack of the stack of $\zeta$-semistable $\mathbb{C}Q$-representations, the closed points of which correspond to the representations in a Serre subcategory $\SP$ of the category of $\mathbb{C}Q$-representations.  For $\dd\in\mathbb{N}^{Q_0}\setminus\{0\}$ a nonzero dimension vector of slope $\mu$, define the following element of $\MMHM(\Msp^{\zeta\sst}_{\dd})$, the category of monodromic\footnote{See Section \ref{MMHMs} for an introduction to this category.  It is an enlargement of the ordinary category of mixed Hodge modules, taking account of the monodromy operation on the vanishing cycles functor.} mixed Hodge modules on the coarse moduli space of $\zeta$-semistable $\mathbb{C}Q$-representations:
\[
\DTS_{W,\dd}^{\zeta}\colonequals \begin{cases}\phim{\WW^\zeta_{\dd}}\ICS_{\Msp^{\zeta\sst}_{\dd}}(\mathbb{Q})&\textrm{if  }\Msp^{\zeta\st}_{\dd}\neq\emptyset \\0&\textrm{otherwise}\end{cases}
\]
and let $\DTS^{\zeta}_{W,\mu}$ be the direct sum of all the $\DTS_{W,\dd}^{\zeta}$ for $\dd$ of slope $\mu$ with respect to the stability condition $\zeta$.  Define 
\[
\DT_{W,\dd}^{\zeta,\SP}=\Ho\left(\dim_!(\Msp^{\zeta\sst,\SP}_{\dd}\rightarrow \Msp^{\zeta\sst}_{\dd})^*\DTS_{W,\dd}^{\zeta}\right)^{\vee}\in\Dub(\MMHM(\Lambda_{\mu}^{\zeta})),
\]
where the superscript denotes the dual in the category of monodromic mixed Hodge modules, so in particular,
\[
\DT_{W,\dd}^{\zeta}\cong \Ho\left(\dim_*\DTS_{W,\dd}^{\zeta}\right)
\]
by Verdier self-duality of $\DTS_{W,\dd}^{\zeta}$.  Define $\DT_{W,\mu}^{\zeta,\SP}$ to be the direct sum of all the $\DT_{W,\dd}^{\zeta,\SP}$ for $\dd$ of slope $\mu$.  Then
\begin{align*}
\FreeComm_{\boxtimes_+}\left(\HO(\BC)_{\vir}\otimes \DT_{W,\mu}^{\zeta,\SP}\right)\cong \HO(\Coha^{\zeta,\SP}_{W,\mu}) &&\emph{(absolute integrality theorem)}
\end{align*}
in $\Dub(\MMHM(\Lambda_{\mu}^{\zeta}))$.  Furthermore, working over the coarse moduli space of $\zeta$-semistable representations $\Msp_{\mu}^{\zeta\sst}$, there is an isomorphism 
\begin{align*}
\Sym_{\boxtimes_{\oplus}}\left(\HO(\BC)_{\vir}\otimes\DTS_{W,\mu}^{\zeta}\right)\cong\Ho(\mathcal{A}_{W,\mu}^{\zeta})&&\emph{(relative integrality theorem)}
\end{align*}
in $\Dub(\MMHM(\Msp^{\zeta\sst}_{\mu}))$.
\end{thmx}
The definition of $\HO(\BC)_{\vir}$ is given in Section \ref{MMHMs}.  Requisite notions from the theory of representations of quivers, including the functors $\Sym_{\boxtimes_{\oplus}}$ and $\Sym_{\boxtimes_{+}}$, and the definition of $\Msp^{\zeta\sst}_{\mu}$, are recalled in Section \ref{Qrepsection}.  The functor $\Sym_{\boxtimes_{+}}$, via the correspondence between monodromic mixed Hodge modules on $\mathbb{N}^{Q_0}$ and $\mathbb{N}^{Q_0}$-graded monodromic mixed Hodge structures, is the functor that sends an $\mathbb{N}^{Q_0}$-graded monodromic mixed Hodge structure $\mathcal{F}$ satisfying $\mathcal{F}_0=0$ to the $\mathbb{N}^{Q_0}$-graded monodromic mixed Hodge structure
\[
\Sym_{\boxtimes_+}(\mathcal{F})\colonequals \bigoplus_{i\geq 0}\Sym^i(\mathcal{F})
\]
given by taking the free $\mathbb{N}^{Q_0}$-graded supercommutative algebra generated by $\mathcal{F}$ and then forgetting the algebra structure.
\smallbreak
We briefly explain the relation to earlier, less ``categorical'', incarnations of the integrality conjecture in terms of generating functions of weight polynomials, as that is the language in which the integrality conjecture is perhaps most familiar.  In the literature, this is most commonly referred to as ``refined Donaldson--Thomas theory''.  Given an element $\mathcal{L}\in\Db(\MMHM(\mathbb{N}^{Q_0}))$, we define the formal power series
\[
\mathcal{Z}(\mathcal{L},q^{1/2},x_1,\ldots,x_n)\colonequals \sum_{\dd\in\mathbb{N}^{Q_0}}\chi_q(\mathcal{L}_{\dd},q^{1/2})x^{\dd},
\]
where we have used the weight polynomial
\[
\chi_q(\mathcal{L}_{\dd},q^{1/2})\colonequals \sum_{i,j\in\mathbb{Z}} (-1)^i\dim(\GrW{j}(\HO^i(\mathcal{L}_{\dd})))q^{j/2}.
\]
As in Theorem \ref{ThmA} we fix a quiver $Q$, potential $W$, stability condition $\zeta$, and slope $\mu$.  Although the total cohomology $\HO(\mathcal{A}_{W,\mu}^{\zeta,\SP})$ (defined as in (\ref{ua})) is not in the bounded derived category $\Db(\MMHM(\Lambda_{\mu}^{\zeta}))\subset \Db(\MMHM(\mathbb{N}^{Q_0}))$, since each $\HO(\mathcal{A}_{W,\dd}^{\zeta,\SP})$ will typically be nonzero in arbitrarily high cohomological degree, the partition function
\[
\mathcal{Z}\left(\HO(\mathcal{A}_{W,\mu}^{\zeta,\SP}),q^{1/2},x_1,\ldots,x_n\right)
\]
still makes sense as a formal power series in $x_1,\ldots,x_n$, with coefficients given by Laurent series in $q^{1/2}$.  This is because for each $\dd\in\Lambda_{\mu}^{\zeta}$, and for each $j\in\mathbb{Z}$, the element $\GrW{j}(\HO^i(\mathcal{A}_{W,\dd}^{\zeta,\SP}))$ is nonzero for only finitely many values of $i$; we denote by $\Dulf(\MMHM(\mathbb{N}^{Q_0}))\subset \Dub(\MMHM(\mathbb{N}^{Q_0}))$ the full subcategory of monodromic mixed Hodge modules satisfying this condition and also the condition that for fixed $\dd$ the element $\GrW{j}(\HO(\mathcal{L}_{\dd}))$ vanishes for $j\ll 0$.  Then one way to define the plethystic exponential
\[
\EXP\colon\mathbb{Z}((q^{1/2}))[[x_1,\ldots,x_n]]_+\rightarrow\mathbb{Z}((q^{1/2}))[[x_1,\ldots,x_n]]
\]
where the subscript ``$+$'' on the left hand side signifies that we take the subgroup of elements $\sum_{\dd\in\mathbb{N}^{Q_0}}a_{\dd}(q^{1/2})x^{\dd}$ such that $a_0(q^{1/2})=0$, is via the formula for $\mathcal{L}\in\Dulf(\MMHM(\mathbb{N}^{Q_0}\setminus\{0\}))$:
\begin{equation}
\label{EXPdef}
\EXP\left(\mathcal{Z}(\mathcal{L},q^{1/2},x_1,\ldots,x_n)\right)=\mathcal{Z}\left(\Sym_{\boxtimes_+}(\mathcal{L}),q^{1/2},x_1,\ldots,x_n\right).
\end{equation}

With the above conventions in place, we may finally state the integrality conjecture in refined DT theory.  It is the statement that there exist Laurent polynomials $\Omega^{\zeta,\SP}_{W,\dd}(q^{1/2})\in\mathbb{Z}[q^{\pm 1/2}]$ such that 
\begin{equation}
\label{ics}
\mathcal{Z}(\HO(\mathcal{A}_{W,\mu}^{\zeta,\SP}),q^{1/2},x_1,\ldots,x_n)=\EXP\left(\sum_{\dd\in\Lambda_{\mu}^{\zeta}\setminus\{0\}} \Omega^{\zeta,\SP}_{W,\dd}(q^{-1/2})x^{\dd}/(q^{1/2}-q^{-1/2})\right).
\end{equation}
The formal power series $\Omega_{W,\dd}^{\zeta,\SP}(q^{1/2})$ are, by definition\footnote{The substitution $q^{1/2}\mapsto q^{-1/2}$ arises from the fact that DT invariants are typically defined in terms of compactly supported cohomology, which is the dual to the cohomology theory we consider.}, the refined Donaldson--Thomas/Bogomol'nyi--Prasad--Sommerfield\footnote{Some authors refer to the coefficients in the partition function on the left hand side of (\ref{ics}) as (refined) DT invariants, while some have reserved this term to refer to the $\Omega_{W,\dd}^{\zeta}(q^{1/2})$ appearing on the right hand side.  In order to establish some distinction between the two types of invariants, we will refer to the $\Omega_{W,\dd}^{\zeta}(q^{1/2})$ exclusively as BPS invariants, and to the left hand side of (\ref{ics}) as the DT partition function.} invariants for the slope $\mu$, stability condition $\zeta$, and Serre subcategory $\mathcal{S}$.  Note that 
\[
\chi_q(\HO(\BC,\mathbb{Q})_{\vir},q^{1/2})=-q^{1/2}-q^{3/2}-\ldots=(q^{1/2}-q^{-1/2})^{-1}
\]
and so we deduce from (\ref{EXPdef}), (\ref{ics}) and Theorem \ref{ThmA} the equality
\[
\Omega^{\zeta,\SP}_{W,\dd}(q^{-1/2})=\chi_q(\DT_{W,\dd}^{\zeta,\SP},q^{1/2})
\]
where $\DT_{W,\dd}^{\zeta,\SP}$ are the cohomological invariants defined in Theorem \ref{ThmA}.  Note also that since $\EXP$ is injective, one can always find formal Laurent power series $\Omega_{W,\dd}^{\zeta,\SP}(q^{1/2})$ satisfying (\ref{ics}).  The content of the integrality conjecture, as proved in \cite{KS2}, is that these formal power series are in fact polynomials.

We can deduce the refined integrality conjecture from Theorem \ref{ThmA}; since $\DT_{W,\dd}^{\zeta,\SP}$ is the hypercohomology of a bounded complex of monodromic mixed Hodge modules, its weight polynomial is a genuine Laurent polynomial instead of a formal power series.  However, the cohomological lift provided by Theorem \ref{ThmA} is much stronger than the integrality conjecture.  It upgrades an equality in the Grothendieck (lambda-)ring of $\Dulf(\MMHM(\mathbb{N}^{Q_0}))$ to an isomorphism in that category.  This is a surprising result even before introducing the extra structure of the cohomological Hall algebra, since it shows that there is an entire theory of ``cohomologically refined'' BPS invariants waiting to be explored.  For early applications of this theory, we refer the reader to the first author's reproof of the Kac positivity conjecture \cite{Da13b} (originally proved by Hausel, Letellier and Villegas \cite{HLRV13}) and proof of the quantum cluster positivity conjecture \cite{Da16a}.
\subsection{Cohomological wall-crossing}
Our second main result is a (relative) cohomological lift of the following incarnation of the wall crossing formula, with $\int$ denoting the Kontsevich--Soibelman integration map:
\begin{equation}
\label{WCI}
\int[\Mst^{\SP},\phi_{\Tr(W)}]=\prod_{\infty\xrightarrow{\mu}-\infty}\int[\Mst^{\zeta\sst,\SP}_{\mu},\phi_{\Tr(W)}].
\end{equation}
This is an identity in the completed motivic quantum torus of finite-dimensional modules of the Jacobi algebra for the pair $(Q,W)$.  Instead of defining any of the terms in (\ref{WCI}) we refer the reader to \cite{JoyceMF} and \cite{KS1} for more details on motivic Hall algebras and integration maps in the general theory of motivic Donaldson--Thomas invariants, or \cite{DaMe4} for the specific case of the motivic Hall algebra of representations of a Jacobi algebra.  
  
  \smallbreak
  We define
\begin{align}
\label{relCA}
\Ho(\mathcal{A}_W)\colonequals &\Ho\left(\bigoplus_{\dd\in\mathbb{N}^{Q_0}}(\Mst_{\dd}\rightarrow\Msp_{\dd})_*\phi_{\WWW_{\dd}}\QQ_{\Mst_{\dd}}\otimes\LL^{-\dim(\Mst_{\dd})/2}\right)
\end{align}
where $\Mst_{\dd}$ is the stack of $\dd$-dimensional $\mathbb{C}Q$-representations, and $\Msp_{\dd}$ is the coarse moduli space of $\dd$-dimensional $\mathbb{C}Q$-representations, the closed points of which are in bijection with $\dd$-dimensional semisimple $\mathbb{C}Q$-representations.  
\begin{thmx}[Cohomological wall crossing theorem]
\label{CWCT}
There is an isomorphism
\begin{equation}
\label{gthmd}
\Ho(\mathcal{A}_W)\cong\Boxtimes_{\oplus, \infty\xrightarrow{\mu} -\infty}^{\tw} \Ho(\mathcal{A}_{W,\mu}^{\zeta}).
\end{equation}
inside $\Dulf(\MMHM(\Msp))$.  Here the product $\boxtimes_{\oplus,\infty \xrightarrow{\mu}-\infty}^{\tw}$ is an ordered product, taken over descending slopes.  Similarly, there is an isomorphism
\begin{equation}\label{gthmad}
\HO(\mathcal{A}^{\SP}_W)\cong \Boxtimes_{+,\infty \xrightarrow{\mu}-\infty}^{\tw} \HO(\mathcal{A}^{\zeta,\SP}_{W,\mu})
\end{equation}
inside $\Dulf(\MMHM(\mathbb{N}^{Q_0}))$.
\end{thmx}
Note that in general, the individual terms in the tensor products on the right hand sides of (\ref{gthmd}) and (\ref{gthmad}) will depend on the stability condition, while the terms on the left hand side do not.  The fact that very different sheaves can give rise to the same tensor product is explained by the fact that in contrast with the product $\boxtimes_{\oplus}$, the product $\boxtimes^{\tw}_{\oplus}$ is not symmetric; this is the categorification of the fact that the quantum torus utilized in \cite{KS1} as the target of the integration map from the motivic Hall algebra is generally not commutative.


\subsection{The perverse filtration}
An advantage of considering constructions relative to the base $\Msp_{\mu}^{\zeta\sst}$ instead of $\Lambda_{\mu}^{\zeta}$ is that it leads naturally to the introduction of the perverse filtration on $\HO(\mathcal{A}_{W,\mu}^{\zeta,\SP})$.  

Let $\ff\in\mathbb{N}^{Q_0}$ be a framing vector, and let $\pi^{\zeta}_{\ff,\dd}\colon \Msp^{\zeta}_{\ff,\dd}\rightarrow\Msp^{\zeta\sst}_{\dd}$, be the forgetful map to the coarse moduli space of $\zeta$-semistable $Q$-representations from the fine moduli space of $\ff$-framed $\zeta$-semistable representations, and let $\WW_{\ff,\dd}^{\zeta}\colon \Msp^{\zeta}_{\ff,\dd}\rightarrow \mathbb{C}$ be the induced function.  We show that for fixed $n$, and sufficiently large framing vector $\ff$, there is a natural isomorphism
\begin{equation}
\label{basapp}
\HO^n(\mathcal{A}_{W,\dd}^{\zeta,\SP})\cong\LL^{\ff\cdot\dd-(\dd,\dd)/2}\otimes \HO_c^n(\Msp^{\zeta\sst,\SP}_{\ff,\dd},\phim{\WW_{\ff,\dd}^{\zeta}}\mathbb{Q}_{\Msp^{\zeta}_{\ff,\dd}})^{\vee},
\end{equation}
in words, we can approximate the restricted vanishing cycle cohomology on the \textit{stack} $\Mst(Q)_{\dd}^{\zeta\sst}$ by considering restricted vanishing cycle compactly supported cohomology on the moduli \textit{scheme} $\Msp^{\zeta}_{\ff,\dd}$.  This may remind the reader of Totaro's construction \cite{To99} for calculating equivariant intersection theory, and indeed, in a sense that will become precise in Section \ref{cohaDT}, our setup is a relative compactification of Totaro's construction, over the base $\Msp^{\zeta\sst}_{\dd}$.  In particular, since the maps $\pi^{\zeta}_{\ff,\dd}$ are proper, we can use the decomposition theorem along with fundamental properties of the vanishing cycles functor to deduce that 
\[
\pi^{\zeta}_{\ff,\dd,*}\phim{\WW_{\ff,\dd}^{\zeta}}\mathbb{Q}_{\Msp^{\zeta}_{\ff,\dd}}\cong\Ho\left(\pi^{\zeta}_{\ff,\dd,*}\phim{\WW_{\ff,\dd}^{\zeta}}\mathbb{Q}_{\Msp^{\zeta}_{\ff,\dd}}\right)
\]
which in turn implies that the hypercohomology of the right hand side (and hence the left hand side) of (\ref{basapp}) carries a well-behaved perverse filtration.
\subsection{Main results for cohomological Hall algebras}
Assume that the stability condition $\zeta$ is generic.  The object $\HO(\mathcal{A}^{\zeta,\SP}_{W,\mu})$ is a unital algebra in the category of $\Lambda_{\mu}^{\zeta}$-graded monodromic mixed Hodge structures, which corresponds to a monoid in the category $\Dub(\MMHM(\Lambda_{\mu}^{\zeta}))$.  The monoidal product in this category is defined by 
\[
\mathcal{F}\boxtimes_+\mathcal{G}\colonequals +_*(\mathcal{F}\boxtimes \mathcal{G}),
\]
where $\mathcal{F}\boxtimes\mathcal{G}$ is the usual external tensor product, and $+\colon\mathbb{N}^{Q_0}\times\mathbb{N}^{Q_0}\rightarrow \mathbb{N}^{Q_0}$ is the addition map.  A unital graded algebra structure is determined by a map
\[
\mathcal{F}\boxtimes_+\mathcal{F}\rightarrow \mathcal{F}
\]
satisfying the obvious associativity condition, along with a unit map $\mathbb{Q}_{\{0\}}\rightarrow\mathcal{F}$ from the constant sheaf supported at the monoidal unit $0\in \mathbb{N}^{Q_0}$.  Although we prove theorems regarding this monoid, we do so via theorems regarding the monoid
\begin{equation}
\label{intrRel}
\Ho(\mathcal{A}_{W,\mu}^{\zeta})\in\Dulf(\MMHM(\Msp^{\zeta\sst}_{\mu}))
\end{equation}
in the  category of monodromic mixed Hodge modules on the coarse moduli space of $\zeta$-semistable $\mathbb{C}Q$-modules of slope $\mu$.  After restricting to $\Msp^{\zeta\sst,\SP}_W$ and applying the monoidal functor obtained by taking the direct image along the map 
\[
\dim\colon \Msp_{\mu}^{\zeta\sst}\rightarrow\mathbb{N}^{Q_0},
\]
the monodromic mixed Hodge module (\ref{intrRel}) becomes isomorphic to the perverse associated graded version of the monoid (\ref{ua}).  In symbols, there is an isomorphism of algebras
\begin{equation}
\label{PAG}
\Ho\left(\dim_*\left(\Ho\left((\Msp^{\zeta\sst,\SP}_{\mu}\hookrightarrow \Msp^{\zeta\sst}_{\mu})^!\mathcal{A}_{W,\mu}^{\zeta}\right)\right)\right)\cong \Gr_{\Pf}\left(\HO(\mathcal{A}_{W,\mu}^{\zeta,\SP})\right).
\end{equation}
All of our main results regarding $\HO(\mathcal{A}_{W,\mu}^{\zeta,\SP})$ and $\Ho(\mathcal{A}_{W,\mu}^{\zeta})$ stem from the observation that all product and coproduct structures that we consider on $\HO(\mathcal{A}^{\zeta,\SP}_{W,\mu})$ and $\Ho(\mathcal{A}^{\zeta,\SP}_{W,\mu})$ respect the perverse filtration, giving rise to the algebra structures on the right hand side of (\ref{PAG}) that are much easier to describe --- precisely, we prove that (in an appropriate sense) the algebras appearing in (\ref{PAG}) are supercommutative degenerations of the CoHA $\HO(\mathcal{A}^{\zeta,\SP}_{W,\mu})$.
\smallbreak
To define this ``appropriate sense'', we must take some care over the signs in the symmetrizing morphism 
\[
\textbf{sw}:\mathcal{F}\boxtimes_+ \mathcal{F}\rightarrow \mathcal{F}\boxtimes_+\mathcal{F}.
\]
It was noted already in \cite{KS2} that the cohomological Hall algebra on the graded vector space (\ref{baby}), while not commutative, becomes commutative after twisting by a simple system of signs.  There are two related ways of expressing this idea, which we develop in tandem in this paper.  

Firstly, the algebra \textit{is} commutative, after replacing the usual symmetrizing morphism for the category in which we consider it (involving the Koszul sign rule) by a new one, which depends on $\mathbb{N}^{Q_0}$-degree.  We consider an adapted symmetrizing morphism ${}^{\tau}\textbf{sw}$ for the categories of monodromic mixed Hodge modules on $\Lambda_{\mu}^{\zeta}$ and $\Msp^{\zeta\sst}_{\mu}$, with respect to which Theorems \ref{qea} and \ref{strongPBW} are stated --- see Section \ref{lc_signs} for the provenance of these signs.  At the level of elements, the rule is easy to state.  Consider the form $\tau$ induced by 
\begin{equation}
\label{half_tau}
\tau\colon \dd'\otimes \dd''\mapsto \chi(\dd',\dd'')+\chi(\dd',\dd')\chi(\dd'',\dd'')
\end{equation}
on the $\mathbb{Z}/2\mathbb{Z}$-vector space $V=(\mathbb{Z}/2\mathbb{Z})^{Q_0}$.  Then
\[
{}^{\tau}\textbf{sw}(a\otimes b)=(-1)^{\tau(\dd',\dd'')+\lvert a\lvert \cdot\lvert b\lvert}b\otimes a
\]
where $\dd'$ and $\dd''$ are the $\mathbb{N}^{Q_0}$-degrees of $a$ and $b$ respectively, and $\lvert a\lvert$ and $\lvert b\lvert$ are their cohomological degrees.  In other words, we modify the usual Koszul sign rule by $(-1)^{\tau(\dd',\dd'')}$.
\smallbreak
Now we describe the second approach to signs.  As in \cite[Sec.2.6]{KS2} and \cite{Efimov} it is sometimes more convenient to modify the definition of the product in the CoHA $\HO(\Coha_{W,\mu}^{\zeta,\SP})$, so that the algebra (or in the context of this paper, its perverse associated graded algebra) becomes supercommutative with respect to the usual Koszul sign rule.  The form $\tau$ vanishes on pairs $\dd\otimes \dd$, and so there exists a bilinear form $\psi$ on $V$ such that
\begin{equation}
\label{pcocdef}
\psi(\dd',\dd'')+\psi(\dd'',\dd')=\tau(\dd',\dd'').
\end{equation}
We define the underlying graded monodromic mixed Hodge structure of $\HO(^\psi\!\!\Coha_{W,\mu}^{\zeta,\SP})$ to be the same as that of $\HO(\Coha_{W,\mu}^{\zeta,\SP})$.  We define the multiplication on $\HO(^{\psi}\!\!\Coha_{W,\mu}^{\zeta,\SP})$ via 
\begin{equation}
\label{psitwdef}
a\otimes b\mapsto (-1)^{\psi(\dd',\dd'')}\HO(\ast_{W,\mu}^{\zeta,\SP})(a\otimes b)
\end{equation}
where $a\in \HO(^{\psi}\!\!\Coha_{W,\dd'}^{\zeta,\SP})$ and $b\in \HO(^{\psi}\!\!\Coha_{W,\dd''}^{\zeta,\SP})$.  Associativity of the twisted CoHA follows from associativity for the untwisted CoHA and bilinearity of $\psi$.


\smallbreak
Our first main result concerning CoHAs is an upgrade of Theorem \ref{ThmA}.  It states that the isomorphisms of Theorem \ref{ThmA} can be realised explicitly as PBW-type isomorphisms defined using the structure of the relative, or absolute CoHA, respectively. 
\begin{thmx}\label{qea}
Let $\zeta$ be a $\mu$-generic Bridgeland stability condition.  The localised bialgebra structure on $\HO(\mathcal{A}^{\zeta,\SP}_{W,\mu})$ induces a (non-localised) Hopf algebra structure on $\Gr_{\Pf}(\HO(\mathcal{A}_{W,\mu}^{\zeta,\SP}))$, which is a free commutative algebra with respect to the symmetrizing morphism ${}^{\tau}\!\textbf{sw}$.  There is a canonical PBW isomorphism in $\Dulf(\MMHM(\Msp_{\mu}^{\zeta\sst}))$:
\begin{equation}
\label{PBWiso1}
{}^{\tau}\!\FreeComm_{\boxtimes_{\oplus}}\left(\HO(\BC)_{\vir}\otimes\DTS_{W,\mu}^{\zeta}\right)\rightarrow \Ho(\Coha_{W,\mu}^{\zeta})
\end{equation}
that is realized via a canonical embedding\footnote{Strictly speaking, these maps are embeddings in the category of $\mathbb{Z}$-graded monodromic mixed Hodge-modules, considered as maps in the derived category of monodromic mixed Hodge modules via the natural embedding of categories.} $\HO(\BC)_{\vir}\otimes\DTS_{W,\mu}^{\zeta}\subset \Ho(\Coha_{W,\mu}^{\zeta})$ and the relative cohomological Hall algebra multiplication.  There is a similar canonical PBW isomorphism over the base $\Lambda^{\zeta}_{\mu}$:
\begin{equation}
\label{PBWiso2}
{}^{\tau}\!\Sym_{\boxtimes_+}\left(\HO(\BC)_{\vir}\otimes \DT_{W,\mu}^{\zeta,\SP}\right)\rightarrow\HO(\mathcal{A}_{W,\mu}^{\zeta,\SP}).
\end{equation}
and the map is defined via the cohomological Hall algebra product on $\HO(\Coha_{W,\mu}^{\zeta,\SP})$.  
If $\psi$ is a bilinear form as in (\ref{pcocdef}), then $\Gr_{\Pf}(\HO(^{\psi}\!\!\mathcal{A}_{W,\mu}^{\zeta,\SP}))$ is a free supercommutative algebra, and there are canonically defined PBW isomorphisms
\begin{align}
\label{PBWiso3}
\FreeComm_{\boxtimes_{\oplus}}\left(\HO(\BC)_{\vir}\otimes \DTS_{W,\mu}^{\zeta}\right)\rightarrow &\Ho(^{\psi}\!\!\Coha_{W,\mu}^{\zeta})\\
\label{PBWiso4}
\Sym_{\boxtimes_+}\left(\HO(\BC)_{\vir}\otimes\DT_{W,\mu}^{\zeta,\SP}\right)\rightarrow&\HO(^{\psi}\!\!\mathcal{A}_{W,\mu}^{\zeta,\SP})
\end{align}
to the relative and absolute CoHAs with the $\psi$-twisted multiplication.
\end{thmx}
Note that we do not claim that in the ``absolute'' cases \eqref{PBWiso1} and \eqref{PBWiso3} the canonical isomorphisms of Theorem \ref{qea} are isomorphisms of algebras --- in many interesting cases they will not be, hence the moniker ``PBW theorem''.  We expect that the case of symmetric quivers with zero potential, for which it is proved by Efimov in \cite{Efimov} that the $\psi$-twisted cohomological Hall algebra is free supercommutative, is atypical.
\smallbreak
Our final theorem is a more general PBW type statement for the whole CoHA.  Just as Theorem \ref{qea} states that the integrality isomorphisms of Theorem \ref{ThmA} can be realised in terms of PBW-type isomorphisms in the relative CoHA, Theorem \ref{strongPBW} states that the same is true of the combination of the integrality isomorphisms (for every slope) and the wall crossing isomorphism, for a generic stability condition.
\begin{thmx}\label{strongPBW}
Let $\zeta$ be a generic stability condition for $Q$.  Then there are embeddings\footnote{See previous footnote.} $(\Msp^{\zeta\sst}_{\mu}\rightarrow\Msp)_* \DTS^{\zeta}_{W,\mu}\subset \Ho(\mathcal{A}_{W,\mu})$ and $\HO(\Msp_{\mu}^{\zeta\sst,\SP},\DTS^{\zeta}_{W,\mu})\subset \HO(\mathcal{A}^{\SP}_W)$ such that the resulting morphisms
\begin{align*}
\Boxtimes_{\oplus, \infty\xrightarrow{\mu} -\infty}^{\tw} {}^{\tau}\!\Sym_{\boxtimes_{\oplus}}\left(\HO(\BC)_{\vir}\otimes (\Msp^{\zeta\sst}_{\mu}\rightarrow\Msp)_* \DTS^{\zeta}_{W,\mu}\right)\rightarrow &\Ho(\mathcal{A}_W)\\
\Boxtimes_{+, \infty\xrightarrow{\mu} -\infty}^{\tw}{}^{\tau}\!\Sym_{\boxtimes_+}\left( \HO(\BC)_{\vir}\otimes \HO(\Msp_{\mu}^{\zeta\sst,\SP},\DTS^{\zeta}_{W,\mu})\right)\rightarrow &\HO(\mathcal{A}_W^{\SP})
\end{align*}
induced by the relative and absolute cohomological Hall algebra multiplications, respectively, are isomorphisms.  Similarly, if $\psi$ is a twist defined as in (\ref{pcocdef}), then the morphisms
\begin{align*}
\Boxtimes_{\oplus, \infty\xrightarrow{\mu} -\infty}^{\tw} \Sym_{\boxtimes_{\oplus}}\left(\HO(\BC)_{\vir}\otimes (\Msp^{\zeta\sst}_{\mu}\rightarrow\Msp)_* \DTS^{\zeta}_{W,\mu}\right)\rightarrow &\Ho(\mathcal{A}_W)\\
\Boxtimes_{+, \infty\xrightarrow{\mu} -\infty}^{\tw}\Sym_{\boxtimes_+}\left( \HO(\BC)_{\vir}\otimes \HO(\Msp_{\mu}^{\zeta\sst,\SP},\DTS^{\zeta}_{W,\mu})\right)\rightarrow &\HO(\mathcal{A}_W^{\SP})
\end{align*}
induced by the same embeddings as above and the $\psi$-twisted relative and absolute cohomological Hall algebra multiplications are isomorphisms.
\end{thmx}
Theorem \ref{qea} and \ref{strongPBW} together give a very general structural result on the relative and absolute CoHAs, namely that they are obtained by infinite extensions of algebras (indexed by descending slope) that are themselves noncommutative deformations of free supercommutative algebras, with characteristic functions determined by the cohomological BPS invariants.

\subsection{The BPS Lie algebra}
\label{BPSLA}
One of the main contributions of the current paper, and a direct corollary of the above structural theorems, is the definition of the BPS Lie algebra.  This is a natural Lie algebra, for any choice of $Q,W,\mu,\SP,\psi$ and $\mu$-generic stability condition $\zeta$, the characteristic function of which recovers the (refined) BPS invariants for this data.  Part of Theorem \ref{qea} states that $\Gr_{\Pf}(\HO(^{\psi}\!\!\mathcal{A}_{W,\mu}^{\zeta,\SP}))$ is supercommutative.  On the other hand, the perverse filtration on $\HO(^{\psi}\!\!\Coha_{W,\mu}^{\zeta,\SP})$ begins in degree one, and in degree one is given precisely by a half Tate-twist of the cohomological BPS invariants.  It follows that there is an inclusion of Lie algebras
\begin{equation}
\label{BPSdomain}
\LL^{1/2}\otimes\DT_{W,\mu}^{\zeta,\SP}\hookrightarrow \HO(^{\psi}\!\!\Coha_{W,\mu}^{\zeta,\SP}),
\end{equation}
i.e. that the first perverse piece of the target is closed under the commutator Lie bracket (since this bracket becomes zero in the perverse associated graded algebra).  We call the domain of (\ref{BPSdomain}) the \textit{BPS Lie algebra} associated to the data $(Q,W,\zeta,\mu,\SP,\psi)$.  In a forthcoming paper the first author shows that for a specific choice $\tilde{Q},\tilde{W}$ of quiver with potential depending on a given quiver $Q$,
\[
\HO^0(\LL^{1/2}\otimes \DTS_{\tilde{W}})\cong\mathfrak{g}_{Q_{\mathrm{re}}}^+,
\]
where $\mathfrak{g}_{Q_{\mathrm{re}}}^+$ is the positive half of the Kac--Moody Lie algebra for $Q_{\mathrm{re}}$, the quiver obtained from $Q$ by deleting all vertices for which there are loops, along with all arrows to or from these vertices.  In particular, the BPS Lie bracket can be highly nontrivial.
\subsection{Structure of the paper}
The paper is broken roughly into three parts: there is the introduction above, a second part which goes up to Section \ref{cwcfSec}, and then a third part.  The purpose of the second part is to prove the cohomological integrality theorem (Theorem \ref{ThmA}) and the cohomological wall crossing theorem  (Theorem \ref{CWCT}).  For many applications of cohomological Donaldson--Thomas theory (especially positivity results) these two theorems are all one needs.  In addition, there are a lot of added technicalities (coproducts, perverse associated graded algebras, modified symmetrizer isomorphisms etc.) that go into the proofs, and indeed the statements of Theorems \ref{qea} and \ref{strongPBW}.  Therefore the reader who is less interested in the algebraic structure underlying the integrality isomorphism, or the task of making the integrality isomorphism canonical, can safely treat this paper as ending with Section \ref{cwcfSec}.  We hope that this reduces the technical burden of the current paper for the ``working mathematician''.
\smallbreak
The third part of the paper is dedicated to proving that these isomorphisms can be realised as PBW isomorphisms in the cohomological Hall algebra.  Theorem \ref{qea} can be construed as a second proof of the integrality theorem (Theorem \ref{ThmA}) which is dependent on the first one.  The conceptual payoff for reproving the theorem is that the second time we prove the theorem, we are able to make the isomorphism underlying the integrality theorem \textit{canonical} --- this is a key step towards proving the integrality theorem for the category of coherent sheaves on a 3-Calabi--Yau variety, for example.  That this category locally looks like the category of representations for a quiver with potential is an exercise in deformation theory if one works formally locally, and is proved rigorously by Toda in the analytic topology \cite{Toda17-2}.  The key step in proving the integrality conjecture for coherent sheaves on a 3-Calabi--Yau variety, then, is to prove it locally in a sufficiently canonical manner that the proof glues naturally.  This application will be the subject of a follow-up paper by the authors.

\subsection{Acknowledgements}
During the writing of this paper, Ben Davison was a postdoctoral researcher at EPFL, supported by the Advanced Grant ``Arithmetic and physics of Higgs
moduli spaces'' No. 320593 of the European Research Council.  During the extensive redrafting of the paper he was supported by the University of Glasgow and the University of Edinburgh, partly supported by a Royal Society University Research Fellowship, and by the starter grant ``Categorified Donaldson-Thomas theory'' No. 759967 of the European Research Council.  He would like to sincerely thank Davesh Maulik for the many stimulating conversations that helped this paper along the way. Sven Meinhardt wants to thank Tam\'{a}s Hausel for giving him the opportunity to visit  Lausanne and Markus Reineke for providing a wonderful atmosphere in Wuppertal to complete this research.  We would like to sincerely thank the anonymous referee for their helpful and enlightening suggestions and corrections.
\section{Hodge theory and equivariant vanishing cycles}
\subsection{Monodromic mixed Hodge modules}
\label{MMHMs}
Let $X$ be a complex variety.  Then we define as in \cite{Saito89,Saito90} the category $\MHM(X)$ of mixed Hodge modules on $X$.  See \cite{Saito1} for an overview of the theory.  There is an exact functor $\rat_X\colon \Dub(\MHM(X))\rightarrow\Dub(\perv(X))$ which takes a complex of mixed Hodge modules $\mathcal{F}$ to its underlying complex of perverse sheaves, and commutes with $f_*,f_!,f^*,f^!,\DD_X$ and tensor product.  If no remark is made to the contrary, and a non-derived target or domain category is not specified, all these functors, and indeed all functors for the entirety of this paper, will be considered as derived functors, even if their non-derived versions are well-defined.  
\smallbreak
Set $\mathcal{B}$ to be the full subcategory of mixed Hodge modules on $\mathbb{A}^1$ containing those $\mathcal{F}\in\Ob(\MHM(\mathbb{A}^1))$ such that $(\mathbb{G}_m\rightarrow \mathbb{A}^1)^*\mathcal{F}$ is an admissible variation of mixed Hodge structure.  Set $\mathcal{C}$ to be the full subcategory of $\mathcal{B}$ for which the objects are the admissible variations of mixed Hodge structure on $\mathbb{A}^1$.  Then we define the category of \textit{monodromic mixed Hodge structures} $\MMHS$ to be the Serre quotient $\mathcal{B}/\mathcal{C}$.

\smallbreak
Fix $X$ and let $\mathcal{B}_X$ be the full subcategory containing those $\mathcal{F}\in\MHM(X\times\mathbb{A}^1)$ such that for each $x\in X$, the total cohomology of $(\{x\}\times\GG_m\rightarrow X\times\AA^1)^*\mathcal{F}$ is an admissible variation of mixed Hodge structure on $\GG_m$.  Via Saito's description of $\MHM(X\times\mathbb{A}^1)$, the category $\mathcal{B}_X\subset \MHM(X\times\mathbb{A}^1)$ is the full subcategory, closed under extensions, and containing the following types of objects
\begin{enumerate}
\item
Mixed Hodge modules $\ICSn_{Y\times \mathbb{A}^1}(\mathcal{L})[\dim(Y)+1]$, where $Y\subset X$ is an irreducible closed subvariety, and $\mathcal{L}$ is a pure weight $n$ variation of Hodge structure on $Y'\times\mathbb{G}_m$, for $Y'$ an open dense subvariety of $Y_{\reg}$, and $n\in\mathbb{Z}$.
\item
Mixed Hodge modules $\ICSn_{Y\times \{0\}}(\mathcal{L})[\dim(Y)]$, for $\mathcal{L}$ a pure weight $n$ variation of Hodge structure on $Y'\subset Y_{\reg}\subset X$ as above.
\end{enumerate}
Here $\ICSn_{Y\times \mathbb{A}^1}(\mathcal{L})[\dim(Y)+1]$ and $\ICSn_{Y\times \{0\}}(\mathcal{L})[\dim(Y)]$ are the lifts of intersection complexes to mixed Hodge modules defined in \cite[Thm.3.21]{Saito90}.  By semisimplicity, any object $\mathcal{P}$ of type (1) splits into a direct sum $\mathcal{P}\cong \mathcal{P}_1\oplus\mathcal{P}_{\neq 1}$, where
\[
\mathcal{P}_i=\ICSn_{Y\times \mathbb{A}^1}(\mathcal{L}_i)[\dim(Y)+1]
\]
for $\mathcal{L}_1\subset\mathcal{L}$ the maximal pure variation of Hodge structure to have trivial monodromy around $Y\times\{0\}$, and $\mathcal{L}_{\neq 1}$ its complement.

Inside $\mathcal{B}_X$ there is a Serre subcategory $\mathcal{C}_X$ given by the full subcategory of $\MHM(X\times\AA^1)$ containing those objects $\mathcal{F}$ such that the total cohomology of each $(\{x\}\times\AA^1\rightarrow X\times \AA^1)^*\mathcal{F}$ is an admissible variation of mixed Hodge structure on $\mathbb{A}^1$.    It is generated under extensions by simple objects of the form $\mathcal{P}_1$.  

By Schmid's rigidity theorem, the total cohomology of the restriction to the fibres $\{x\}\times\mathbb{A}^1$ of any such mixed Hodge module is constant, and $\mathcal{C}_X$ is the essential image of the functor $\pi^*[1]\colon \MHM(X)\rightarrow\MHM(X\times\mathbb{A}^1)$, where 
\[
\pi\colon X\times\mathbb{A}^1\rightarrow X
\]
is the projection.  We define following \cite[Sec.7]{KS2} the category of \textit{monodromic mixed Hodge modules} $\MMHM(X)=\mathcal{B}_X/\mathcal{C}_X$.  Since the objects of $\MMHM(X)$ are the objects of $\mathcal{B}_X$, we deduce that all isomorphism classes of objects in $\MMHM(X)$ are obtained from iterated extensions of simple objects of type (1) satisfying $\mathcal{L}=\mathcal{L}_{\neq 1}$, and of type (2).  We identify $\MMHS$ with $\MMHM(\pt)$.
\smallbreak

In \cite{KS2} the category of monodromic mixed Hodge modules is presented as a subcategory of $\MHM(X\times\mathbb{A}^1)$, whereas for us it is a quotient.  We collect together some category theoretic lemmas that clarify this shift in point of view.
\begin{lemma}
\label{adjLem1}
Let $\Psi\colon \mathcal{B}\rightarrow\mathcal{B}$ be a functor, and let $\nu\colon\id_{\mathcal{B}}\rightarrow \Psi$ be a natural transformation satisfying 
\begin{enumerate}
\item
$\Psi(\nu)=\nu_{\Psi}\colon \Psi\rightarrow \Psi^2$, i.e. $\Psi(\nu_{\mathcal{F}})=\nu_{\Psi(\mathcal{F})}$ for all objects $\mathcal{F}$ in $\mathcal{B}$.
\item
$\Psi(\nu)=\nu_{\Psi}$ is an isomorphism of functors.
\end{enumerate}
Let $\mathcal{D}\subset \mathcal{B}$ be the full subcategory containing those objects $\mathcal{F}$ of $\mathcal{B}$ such that $\nu_{\mathcal{F}}\colon\mathcal{F}\rightarrow \Psi(\mathcal{F})$ is an isomorphism.  Then the inclusion functor $\iota\colon\mathcal{D}\rightarrow \mathcal{B}$ is right adjoint to $\psi$
\begin{equation*}
\begin{tikzcd}
\mathcal{B} \arrow[r,shift left=.5ex,"\Psi"]
&
\mathcal{D}. \arrow[l,shift left=.5ex,"\iota"]
\end{tikzcd}
\end{equation*}
\end{lemma}
Condition (2) implies that $\Psi(\mathcal{F})\in\obj(\mathcal{D})$ for all $\mathcal{F}\in\obj(\mathcal{B})$, and so we obtain a factorization $\Psi\colon\mathcal{B}\rightarrow \mathcal{D}\xrightarrow{\iota}\mathcal{B}$.  In the lemma, and from now on, we abuse notation by denoting the resulting functor $\mathcal{B}\rightarrow\mathcal{D}$ also by $\Psi$.
\begin{proof}
Via the above abuse of notation we consider $\nu$ as a natural transformation $\id_{\mathcal{B}}\rightarrow \iota\circ\Psi$.  We claim that this is the unit of the adjunction, while
\[
\eta\colonequals(\nu|_{\mathcal{D}})^{-1}\colon \Psi|_{\mathcal{D}}\rightarrow \id_{\mathcal{D}}=\Psi\circ\iota
\]
will be the unit.  We need to check that the compositions of morphisms
\begin{align}
\label{cua1}
&\Psi\xrightarrow{\Psi(\nu)}\Psi\iota\Psi\xrightarrow{\eta_{\Psi}}\Psi\\
\label{cua2}
&\iota\xrightarrow{\nu_{\iota}}\iota\Psi\iota\xrightarrow{\iota(\eta)}\iota
\end{align}
are the identity.  Pick $\mathcal{F}\in\obj(\mathcal{B})$.  Then 
\begin{align*}
\eta_{\Psi(\mathcal{F})}\circ\Psi(\nu_{\mathcal{F}})=&\nu_{\Psi(\mathcal{F})}^{-1}\circ\Psi(\nu_{\mathcal{F}})\\=&\id_{\Psi(\mathcal{F})}.
\end{align*}
Pick $\mathcal{F}\in\obj(\mathcal{D})$.  Then
\begin{align*}
\iota(\eta_{\mathcal{F}})\circ \nu_{\iota(\mathcal{F})}=&\nu_{\mathcal{F}}^{-1}\circ \nu_{\mathcal{F}}\\=&\id_{\mathcal{F}}.
\end{align*}
\end{proof}
\begin{lemma}
\label{adjLem2}
With the assumptions of Lemma \ref{adjLem1}, assume also that $\mathcal{B}$ is Abelian, and $\Psi$ is exact.  Let $\mathcal{C}$ be the full subcategory of $\mathcal{B}$ containing those objects $\mathcal{F}$ for which $\Psi(\mathcal{F})=0$.  Then $\mathcal{C}$ is a Serre subcategory.  We denote the Serre quotient $\mathcal{B}/\mathcal{C}$.  The functor $\Psi$ induces a functor $\tilde{\Psi}\colon\mathcal{B}/\mathcal{C}\rightarrow \mathcal{B}$, which is right adjoint to the canonical functor $Q\colon \mathcal{B}\rightarrow\mathcal{B}/\mathcal{C}$.
\end{lemma}
\begin{proof}
The category $\mathcal{C}$ is a Serre subcategory by exactness of $\Psi$.  By the universal property of $\mathcal{B}/\mathcal{C}$, $\Psi$ induces a functor $\tilde{\Psi}\colon\mathcal{B}/\mathcal{C}\rightarrow \mathcal{B}$.  We claim that $\nu\colon \id_{\mathcal{B}}\rightarrow \Psi$ provides a unit for the adjunction.  Let $\mathcal{F}\in\mathcal{B}/\mathcal{C}$, and consider the map $\nu_{\mathcal{F}}\colon\mathcal{F}\rightarrow \Psi(\mathcal{F})$.  Since $\Psi(\nu_{\mathcal{F}})$ is an isomorphism, and $\Psi$ is exact, we deduce that $\Psi(\mathrm{ker}(\nu_{\mathcal{F}}))=\Psi(\mathrm{coker}(\nu_{\mathcal{F}}))=0$, and so $\nu_{\mathcal{F}}$ represents an isomorphism $\overline{\nu_{\mathcal{F}}}$ in $\mathcal{B}/\mathcal{C}$.  Arguing as above, $(\overline{\nu_{\mathcal{F}}})^{-1}$ provides the counit to the adjunction.
\end{proof}
The proof of the following lemma is again constructed from the above adjunctions, and is omitted.
\begin{lemma}
With the assumptions of Lemmas (\ref{adjLem1}) and (\ref{adjLem2}), the induced functor $\tilde{\Psi}\colon \mathcal{B}/\mathcal{C}\rightarrow\mathcal{D}$ is an equivalence of categories, with inverse equivalence the composition of canonical functors $\mathcal{D}\rightarrow \mathcal{B}\rightarrow\mathcal{B}/\mathcal{C}$.
\end{lemma}
Now we apply these abstractions to certain categories of mixed Hodge modules.  We define as in \cite{KS2} the exact functor $\Psi_X\colon\mathcal{B}_X\rightarrow\mathcal{B}_X$:

\[
\Psi_X(\mathcal{F})=(X\times\mathbb{A}^2\xrightarrow{\id\times +} X\times \mathbb{A}^1)_*\left(\mathcal{F}\boxtimes (\mathbb{G}_m\rightarrow\mathbb{A}^1)_!\mathbb{Q}_{\mathbb{G}_m}[1]\right).
\] 
The functor is exact since it sends each of the simple mixed Hodge modules identified at the start of this section to mixed Hodge modules.  Consider the distinguished triangle 
\[
 \mathbb{Q}_{\mathbb{A}^1}\rightarrow (\{0\}\hookrightarrow \mathbb{A}^1)_*\mathbb{Q}_{\{0\}} \xrightarrow{s}(\mathbb{G}_m\xrightarrow{j} \mathbb{A}^1)_!\mathbb{Q}_{\mathbb{G}_m}[1]
\]
We define the natural transfromation $\nu$, for $\mathcal{F}\in\mathcal{B}_X$, via the following composition of mixed Hodge modules on $X\times\mathbb{A}^1$
\begin{align*}
\mathcal{F}\cong&(X\times\mathbb{A}^2\xrightarrow{\id\times +} X\times \mathbb{A}^1)_*\left(\mathcal{F}\boxtimes (\{0\}\hookrightarrow\mathbb{A}^1)_*\mathbb{Q}_{\{0\}}\right)\xrightarrow{s'}\\
&(X\times\mathbb{A}^2\xrightarrow{\id\times +} X\times \mathbb{A}^1)_*\left(\mathcal{F}\boxtimes (\mathbb{G}_m\rightarrow\mathbb{A}^1)_!\mathbb{Q}_{\mathbb{G}_m}[1]\right)\cong\Psi(\mathcal{F})
\end{align*}
where 
\[
s'=(X\times\mathbb{A}^2\xrightarrow{\id\times +} X\times \mathbb{A}^1)_*\left(\id_{\mathcal{F}}\boxtimes s\right).
\]
\begin{lemma}
The pair $(\Psi_X,\nu)$ satisfies the conditions of Lemma \ref{adjLem1}.
\end{lemma}
\begin{proof}
Consider the commutative diagram
\begin{equation}
\label{cdr}
\xymatrix{
X\times\AA^1\times\AA^1\times\AA^1\ar[rr]^-{\id_{X\times\AA^1}\times+}\ar[d]^{\id_X\times+\times\id_{\AA^1}}&&X\times\AA^1\times\AA^1\ar[d]^{\id_X\times +}\\
X\times\AA^1\times\AA^1\ar[rr]^-{\id_X\times +}&&X\times \AA^1.
}
\end{equation}

We have 
\begin{align*}
\Psi_X(\nu_{\mathcal{F}})&=(\id_X\times +)_*(\nu_{\mathcal{F}}\boxtimes\id_{j_!\mathbb{Q}_{\mathbb{G}_m}[1]})
\\&=(\id_X\times +)_*(\id_X\times +\times \id_{\AA^1})_*(\id_{\mathcal{F}}\boxtimes s\boxtimes \id_{j_!\mathbb{Q}_{\mathbb{G}_m}[1]})
\\&=(\id_X\times +)_*(\id_{\mathcal{F}}\boxtimes +_*(s\boxtimes \id_{j_!\mathbb{Q}_{\mathbb{G}_m}[1]}))
\end{align*}
where we have used commutativity of (\ref{cdr}) for the final equality.  On the other hand, we have
\begin{align*}
\nu_{\Psi_X(\mathcal{F})}=&(\id_X\times+)_*(\id_{\Psi_X(\mathcal{F})}\boxtimes s)\\
=&(\id_X\times +)_*(\id_X\times +\times \id_{\AA^1})_*(\id_{\mathcal{F}}\boxtimes\id_{j_!\mathbb{Q}_{\mathbb{G}_m}[1]}\boxtimes s)\\
=&(\id_X\times +)_*(\id_{\mathcal{F}}\boxtimes +_*(\id_{j_!\mathbb{Q}_{\mathbb{G}_m}[1]}\boxtimes s))
\end{align*}
and
\[
+_*(s\boxtimes \id_{j_!\mathbb{Q}_{\mathbb{G}_m}[1]})=+_*(\id_{j_!\mathbb{Q}_{\mathbb{G}_m}[1]}\boxtimes s)
\]
by commutativity of $+$.  We compute 
\begin{align*}
\cone(\nu_{\Psi_X(\mathcal{F})})\cong&(\id_X\times +)_*(\mathcal{F}\boxtimes +_*(j_!\mathbb{Q}_{\mathbb{G}_m}[1]\boxtimes \cone(s)))\\
\cong&(\id_X\times +)_*(\mathcal{F}\boxtimes +_*(j_!\mathbb{Q}_{\mathbb{G}_m}[1]\boxtimes \QQ_{\AA^1}[1]))
\end{align*}
but 
\[
+_*(j_!\mathbb{Q}_{\mathbb{G}_m}[1]\boxtimes \QQ_{\AA^1}[1]))\cong \HO^*(\AA^1,j_!\mathbb{Q}_{\mathbb{G}_m})\otimes\QQ_{\AA^2}[2]=0
\]
proving that $\nu_{\Psi_X(\mathcal{F})}$ is an isomorphism.
\end{proof}
\begin{lemma}
\label{wrapup1}
Setting $\Psi=\Psi_X$, $\mathcal{B}=\mathcal{B}_X$ in Lemma (\ref{adjLem2}), we have $\mathcal{C}=\mathcal{C}_X$.
\end{lemma}
\begin{proof}
We denote by $\pi\colon X\times\mathbb{A}^1\rightarrow X$ the projection.  By definition, $\mathcal{C}$ is the full subcategory containing those objects $\mathcal{F}$ such that $\Psi_X(\mathcal{F})=0$.  This is equivalent to the condition that the natural composition
\[
\pi^*\pi_*\mathcal{F}=(X\times\mathbb{A}^2\xrightarrow{\id\times +} X\times \mathbb{A}^1)_*\left(\mathcal{F}\boxtimes \mathbb{Q}_{\mathbb{A}^1}\right)\rightarrow (X\times\mathbb{A}^2\xrightarrow{\id\times +} X\times \mathbb{A}^1)_*\left(\mathcal{F}\boxtimes \mathbb{Q}_{0}[1]\right)=\mathcal{F}
\]
is an isomorphism.  It follows that the objects of $\mathcal{C}$ are contained in the objects of $\mathcal{C}_X$.  On the other hand, the natural transformation $\pi^*\pi_*\pi^*\rightarrow \pi^*$ is an isomorphism and the reverse inclusion holds.
\end{proof}
Via the same arguments, we deduce the following Lemma.
\begin{lemma}
Under the assumptions of Lemma \ref{wrapup1} and with the terminology of Lemma \ref{adjLem1}, $\mathcal{D}$ is the full subcategory of $\mathcal{B}_X$ containing those objects $\mathcal{F}$ satisfying $\pi_*\mathcal{F}=0$.
\end{lemma}

The (sub)category $\mathcal{D}$ is identified as the category of monodromic mixed Hodge modules by Kontsevich and Soibelman; as we have shown, it is equivalent to our (quotient) category of monodromic mixed Hodge modules.

The natural pushforward $(X\times\GG_m\rightarrow X\times\AA^1)_!\colon \MHM(X\times\GG_m)^{\mon}\rightarrow\MMHM(X)$ is an equivalence of categories, where $\MHM(X\times\GG_m)^{\mon}$ is the full subcategory of $\MHM(X\times\GG_m)$ containing those $\mathcal{F}$ such that the total cohomology of each pullback $(\{x\}\times\GG_m\rightarrow X\times\GG_m)^*\mathcal{F}$ is an admissible variation of mixed Hodge structure.  An explicit inverse equivalence $\Theta_X$ is provided by
\begin{equation}
\label{ThetaDef}
\Theta_X\colon\mathcal{F}\mapsto (X\times\GG_m\rightarrow X\times\AA^1)^*\tilde{\Psi}_X\mathcal{F}.
\end{equation}
where $\tilde{\Psi}_X$ is functor induced by $\Psi_X$, as above.  In terms of the two types of simple mixed Hodge modules identified at the start of the section, we have 
\begin{align*}
\Theta_X\ICSn_{Y\times \mathbb{A}^1}(\mathcal{L})[\dim(Y)+1]\cong &\ICSn_{Y\times \mathbb{G}_m}(\mathcal{L}_{\neq 1})[\dim(Y)+1]\\
\Theta_X\ICSn_{Y\times \{0\}}(\mathcal{L})[\dim(Y)]\cong &\ICSn_{Y\times \mathbb{G}_m}(\mathcal{L}\boxtimes \mathbb{Q}_{\mathbb{G}_m})[\dim(Y)+1].
\end{align*}


From the existence of a right adjoint to the functor $\mathcal{B}_X\rightarrow \MMHM(X)$, one can show (see e.g. \cite[06XM]{Stackproject}) that the natural functor
\[
\Dub(\mathcal{B}_X)/\Dub_{\mathcal{C}_X}(\mathcal{B}_X)\rightarrow \Dub(\mathcal{B}_X/\mathcal{C}_X)=\Dub(\MMHM(X))
\]
is an equivalence of categories, where $\Dub_{\mathcal{C}_X}(\mathcal{B}_X)$ is the full subcategory of $\Dub(\mathcal{B}_X)$ consisting of those objects whose cohomology objects lie in $\mathcal{C}_X$.  The subcategory $\Dub_{\mathcal{C}_X}(\mathcal{B}_X)$ is stable under the Verdier duality functor $\DD_{X\times\AA^1}$, and so we obtain a Verdier duality functor on $\Dub(\MMHM(X))$ which we denote $\DD^{\mon}_X$.  The associated graded object $\GrW{\bullet}(\mathcal{F})$ of an object in $\mathcal{C}_X$ with respect to the weight filtration is also in $\mathcal{C}_X$, and so the weight filtration descends to $\MMHM(X)$.  
\begin{definition}
Given an element $\mathcal{G}\in\MMHM(X)$, we say it is pure of weight $i$ if $\GrW{j}\mathcal{G}$ is zero for all $j\neq i$.  Given an element $\mathcal{F}\in\Dub(\MMHM(X))$, we say that $\mathcal{F}$ is pure of weight $i$ if each $\Ho^j(\mathcal{F})$ is pure of weight $i+j$, or we just say that $\mathcal{F}$ is \textit{pure} if it is pure of weight zero.
\end{definition}
\begin{remark}
The equivalence $(X\times\GG_m\rightarrow X\times\AA^1)_!\colon \MHM(X\times\GG_m)^{\mon}\rightarrow\MMHM(X)$ does not preserve the weight filtration; for instance $\mathbb{Q}_{\mathbb{G}_m}[1]$ is a pure variation of Hodge structure of weight 1, while the monodromic mixed Hodge structure $(\mathbb{G}_m\rightarrow\mathbb{A}^1)_!\mathbb{Q}_{\mathbb{G}_m}[1]\cong(\{0\}\hookrightarrow \AA^1)_*\mathbb{Q}_{\{0\}} $ is pure of weight zero.
\end{remark}
If $f\colon X\rightarrow Y$ is a morphism of varieties, we define the functors $f^!,f^*,f_!,f_*$ between the associated derived categories of monodromic mixed Hodge modules to be those induced by the functors $(f\times\id_{\mathbb{A}^1})^!,(f\times\id_{\mathbb{A}^1})^*,(f\times\id_{\mathbb{A}^1})_!,(f\times\id_{\mathbb{A}^1})_*$ between the associated derived categories of mixed Hodge modules.  
\begin{definition}
Assume $X$ is a variety $X\xrightarrow{\tau}\mathbb{N}^{Q_0}$ over a monoid $\mathbb{N}^{Q_0}$ of dimension vectors for a quiver $Q$. 
We define $\HO(X,\mathcal{F})$ to be the total cohomology of $\tau_*\mathcal{F}$.  So $\HO(X,\mathcal{F})$ may be thought of as a cohomologically graded $\mathbb{N}^{Q_0}$-graded monodromic mixed Hodge module on a point, e.g. the $\mathbb{N}^{Q_0}$-graded monodromic mixed Hodge structure underlying the total hypercohomology of $\mathcal{F}$.  Similarly, we define $\HO_c(X,\mathcal{F})$ to be the total cohomology of $\tau_!\mathcal{F}$.
\end{definition}

The monoidal structures $\otimes$ on $\Du(\MMHM(X))$ and $\Dl(\MMHM(X))$ are defined by 
\begin{equation}
\label{simmon}
\mathcal{F}\otimes\mathcal{G}\colonequals (X\times \AA^1\times \AA^1\xrightarrow{\id \times +}X\times \AA^1)_*(\pr_{1,2}^*\mathcal{F}\otimes\pr_{1,3}^*\mathcal{G})
\end{equation}
where $\pr_{i,j}\colon X\times\AA^1\times\AA^1\rightarrow X\times\AA^1$ is the projection onto the $i$th and the $j$th component.  
\begin{remark}
\label{convProdDef}
More generally, let $X$ be a monoid in $\Sch{Y}$, the category of schemes over $Y$, with finite type monoid map $\oplus\colon X\times_Y X\rightarrow X$.  Then define
\[
\mathcal{F}\boxtimes_{\oplus} \mathcal{G}\colonequals (X\times_Y X\times \mathbb{A}^1\times\mathbb{A}^1\xrightarrow{\oplus\times +} X\times\mathbb{A}^1)_*(\pr_{1,3}^*\mathcal{F}\otimes\pr_{2,4}^*\mathcal{G}).
\]
We recover (\ref{simmon}) in the special case in which $(X\xrightarrow{\id_X} X)$ is considered as a monoid in $\Sch{X}$.  In addition, for $\mathcal{F},\mathcal{G}$ objects of $\Du(\MMHM(X)),\Du(\MMHM(Y))$ or $\Dl(\MMHM(X)),\Dl(\MMHM(Y))$ respectively, we define
\[
\mathcal{F}\boxtimes \mathcal{G}\colonequals (X\times Y\times\mathbb{A}^1\times\mathbb{A}^1\xrightarrow{\id_{X\times Y}\times +} X\times Y\times \mathbb{A}^1)_*(\pr_{1,3}^*\mathcal{F}\otimes\pr_{2,4}^*\mathcal{G})\in\Dub(\MMHM(X\times Y)).
\]
This external tensor product is biexact and preserves weight filtrations by \cite[Sec.4]{KS2}, \cite[Prop.3.2]{Da16a}, or the proof of Proposition \ref{symmPure} below.

\end{remark}
There is a fully faithful embedding $\MHM(X)\rightarrow \MMHM(X)$ given by 
\[
i_*=(X\xrightarrow{x\mapsto (x,0)} X\times\AA^1)_*
\]
which is furthermore a monoidal functor, commuting with Verdier duality, as $i$ is proper.

Define $\Theta_X$ as in (\ref{ThetaDef}).  Let $e\colon X\xrightarrow{x\mapsto(x,1)}X\times\mathbb{G}_m$ be the natural inclusion.  Then $e^*\Theta_X[-1]\colon \MMHM(X)\rightarrow \MHM(X)$ is also faithful (again using rigidity for variations of mixed Hodge structure \cite[Thm.4.20]{StZu85}), and so we can define an exact faithful (non-derived) functor
\begin{equation}
\label{formDef}
\form{X}\colon \MMHM(X)\rightarrow \perv(X)
\end{equation}
by setting $\form{X}\colonequals \rat_Xe^*\Theta_X[-1]$.  
\smallbreak

Let $f$ be a regular function on a smooth algebraic variety $X$.  Define $X_{<0}=f^{-1}(\mathbb{R}_{<0})$ and $X_0=f^{-1}(0)$.  We define the functor $\psi_f\colon \Dub(\perv(X))\rightarrow\Dub(\perv(X))$
\[
\psi_f\colonequals (X_0\rightarrow X)_*(X_0\rightarrow X)^*(X_{<0}\rightarrow X)_*(X_{<0}\rightarrow X)^*
\]
and define $\phi_f\colonequals \cone((X_0\rightarrow X)_*(X_0\rightarrow X)^*\rightarrow\psi_f)$; this cone can indeed be made functorial --- see \cite{KSsheaves} for details, or the proof of Proposition \ref{TSeq} below for a functorial definition of $\phi_f$.  The functors $\psi_f$ and $\phi_f$ have lifts to endofunctors of $\Dub(\MHM(X))$, forming a key part of Saito's theory.  


In the sequel we consider vanishing cycles always as a functor $\Dub(\MHM(X))\rightarrow\Dub(\MMHM(X))$ via the definition of \cite[Def.27]{KS2}
\begin{equation}
\label{phimdef}
\phim{f}\colonequals (X\times\GG_m\rightarrow X\times\mathbb{A}^1)_!\phi_{f/u}(X\times\GG_m\rightarrow X)^*.  
\end{equation}
In (\ref{phimdef}), $u$ denotes the coordinate on $\GG_m$.
\begin{example}
Let $f=0$.  Then $\phi_f(\mathcal{F})\cong\mathcal{F}[1]$, since $\psi_f(\mathcal{F})=0$.  Likewise, $\phim{f}\mathcal{F}=(X\times\GG_m\rightarrow X\times\mathbb{A}^1)_!(X\times\GG_m\rightarrow X)^*\mathcal{F}[1]$.  We deduce that $\phim{f}\mathcal{F}=j_!j^*\mathcal{G}[1]$ where $\mathcal{G}=(X\times\mathbb{A}^1\xrightarrow{\pi} X)^*\mathcal{F}$ and $j\colon X\times\GG_m\rightarrow X\times\AA^1$ is the inclusion.  Let $i\colon X\times \{0\}\hookrightarrow X\times\mathbb{A}^1$ be the complement to $j$.  Since $\mathcal{G}$ is by definition trivial in $\MMHM(X)$, we deduce from the distinguished triangle $j_!j^*\rightarrow \id\rightarrow i_*i^*$ that $\phim{f}\mathcal{F}\cong i_*\mathcal{F}$ in $\MMHM(X)$.  In other words, $\phim{0}\mathcal{F}\cong\mathcal{F}$, where we consider $\mathcal{F}$ on the right hand side of the isomorphism as a monodromic mixed Hodge module via pushforward along $X\xrightarrow{i\colon x\mapsto(x,0)} X\times\AA^1$.  
\end{example}
\begin{remark}
Generalising the above example; in fact it follows from exactness of the functors $\phi_{f/u}[-1]$ (see \cite{BBD}), $\pi^*[1]$ and $j_!$ that $\phim{f}$ is exact in general.
\end{remark}

We collect together some useful properties of $\phim{f}$.
\begin{proposition}
\label{basicfacts}
\begin{enumerate}
\item \label{monPD}(Verdier duality): There is a natural isomorphism 
\[
\DD^{\mon}_X\phim{f}\cong \phim{f} \DD_X\colon \Dub(\MHM(X))\rightarrow\Dub(\MMHM(X)), 
\]
where $\DD_X$ is the usual Verdier duality functor.
\item (Interaction with adjunction):\label{adjInt}
If $p\colon X\rightarrow Y$ is a map of varieties, and $f$ is a regular function on $Y$, then there is a natural transformation $\phim{f}p_*\rightarrow p_*\phim{fp}$, which is moreover a natural isomorphism if $p$ is proper.
\item \label{HomInv}(Homotopy invariance): Let $X'\rightarrow X$ be an affine fibration with $d$-dimensional fibres.  Then the natural map $\phim{f} \mathcal{F}\rightarrow p_*\phim{fp}p^*\mathcal{F}$ is an isomorphism, as is the natural map $p_!\phim{fp}p^*\mathcal{F}\rightarrow \phim{f}\mathcal{F}\otimes\HO_c(\mathbb{A}^d)$.
\item\label{exFact}
(Exactness): The functor $\phim{f}\colon \Dub(\MHM(X))\rightarrow\Dub(\MMHM(X))$ is exact, i.e. it restricts to an exact functor $\MHM(X)\rightarrow\MMHM(X)$.
\item (Vanishing cycle decomposition theorem)\label{vanDec}
Let $\mathcal{F}\in\MMHM(X)$ be pure of weight $m$ for some $m\in\mathbb{Z}$, let $p\colon X\rightarrow Y$ be a proper map, and let $f$ be a regular function on $Y$.  Then there is a non-canonical isomorphism
\[
p_*\phim{fp}\mathcal{F}\cong \Ho(p_*\phim{fp}\mathcal{F}).
\]
\item (Thom--Sebastiani isomorphism): Let $f_j\colon X_j\rightarrow\mathbb{C}$ be regular functions, for $j=1,2$.  Then there is a natural ismorphism of bifunctors $\MHM(X)\times\MHM(X)\rightarrow\MMHM(X)$:
\[
\phim{f_1}(\bullet)\boxtimes\phim{f_2}(\bullet)\rightarrow\phim{f_1\boxplus f_2}(\bullet\boxtimes\bullet)|_{f^{-1}_1(0)\times f^{-1}_2(0)}.
\]
\item (Integral identity): Let $V_+\oplus V_-$ be a $\mathbb{C}^*$-equivariant vector bundle on the space $X$, given a $\mathbb{C}^*$-action, where the weights of the action on $V_+$ are all $1$, and the weights of the action on $V_-$ are all $-1$.  Let $f$ be a $\mathbb{C}^*$-invariant function.  Below, for a vector bundle $V'$ we denote the total space of $V'$ also by $V'$.  Then the natural map
\[
(V_+\rightarrow X)_!(V_+\rightarrow V_+\oplus V_-)^*\phim{f}(\mathbb{Q}_{V_+\oplus V_-}\rightarrow (V_+\rightarrow V_+\oplus V_-)_*\mathbb{Q}_{V_+})
\]
is an isomorphism.
\end{enumerate}

\end{proposition}
The first five statements follow from the corresponding statements for $\phi_f$.  The first of these is proved at the level of perverse sheaves as the main theorem of \cite{Ma09}, and for mixed Hodge modules as the main theorem of \cite{Sai89duality}, see also Sch\"urmann's appendix to \cite{Br12}.  The second statement is given by combining \cite[Thm.2.14, Thm.4.3]{Saito90}.  Homotopy invariance then follows from the homotopy invariance statement for perverse sheaves.  The hard part of the exactness statement is the statement that the shift of the usual vanishing cycle functor, $\phi_f[-1]$, is exact.  This is a result of Gabber, and can be found in \cite{BBD}.  The version of the decomposition theorem quoted here follows from Saito's version of the decomposition theorem \cite{Saito90}, and statements (\ref{exFact}) and (\ref{adjInt}).  The statement regarding the Thom--Sebastiani isomorphism is in fact false after replacing $\phim{f}$ with $\phi_f[-1]$.  The statement involving $\phim{f}$ is due to Saito \cite{Saito10}; again we refer the reader to Sch\"urmann's appendix to \cite{Br12} for the compatibility between this proof and the proof of Massey \cite{Ma01} at the level of complexes of perverse sheaves.  The integral identity is proved in the above form in \cite{KS2}, and is a central background result in the proof of the integrality conjecture\footnote{The similarity between the names of the ``integral identity'' and the ``integrality conjecture'' is merely coincidencidental.} in \cite{KS2}.  Somewhat surprisingly, given its crucial role in the theory in other treatments of Hall algebras in Donaldson--Thomas theory, it is the only one of the seven facts about vanishing cyclesabove  that we will not use.

For ease of exposition we make the following simplification in what follows.
\begin{assumption}
\label{phicheat}
For all functions $f\colon X\rightarrow \mathbb{C}$ for which we wish to take $\phim{f}$, if $X$ is smooth, we assume that $\crit(f)\subset f^{-1}(0)$ as sets.
\end{assumption}
Under the assumption, the Thom--Sebastiani isomorphism for mixed Hodge modules simplifies to
\[
\phim{f_1}(\mathcal{F})\boxtimes\phim{f_2}(\mathcal{G})\rightarrow\phim{f_1\boxplus f_2}(\mathcal{F}\boxtimes\mathcal{G}).
\]
The assumption can be dropped, with a little care.  If it does not hold, one can instead work with the functor $\phi^{\mon,\textrm{fib}}_f\colonequals \bigoplus_{a\in\mathbb{A}_{\mathbb{C}}^1}\phim{f-a}$, or with exponential mixed Hodge structures --- see \cite{DaMe4}, \cite{KS2}, respectively.

Consider the embedding $\Db(\MHM(\pt))\xrightarrow{i_*}\Db(\MMHM(\pt))$.  The former category contains the element
\[
\mathbb{L}\colonequals \HO_c(\mathbb{A}^1,\mathbb{Q}),
\]
which is pure: it is concentrated entirely in cohomological degree 2, and has weight 2.  There is no square root of $\mathbb{L}$ in $\Db(\MHM(\pt))$, i.e. an element $\mathbb{L}^{1/2}$ such that $(\mathbb{L}^{1/2})^{\otimes 2}\cong \mathbb{L}$, but after embedding $\Db(\MHM(\pt))$ in $\Db(\MMHS)$ there is a square root.  We set 
\[
\mathbb{L}^{1/2}\colonequals \HO_c(\mathbb{A}^1,\phim{x^2}\mathbb{Q}_{\mathbb{A}^1}), 
\]
where $x^2\colon \mathbb{A}^1\rightarrow\mathbb{C}$ is considered as a regular function, and $\mathbb{Q}_{\mathbb{A}^1}$ is the constant (cohomologically shifted) mixed Hodge module on $\mathbb{A}^1$.  Using the Thom--Sebastiani isomorphism one may show that this is indeed a tensor square root of $\LL$.  We fix an isomorphism 
\[
l\colon \mathbb{L}^{1/2}\otimes\mathbb{L}^{1/2}\cong\mathbb{L}.
\]
This amounts to fixing an isomorphism $r\colon\mathbb{Q}[-2]\cong\form{\pt}\LL$, since up to sign there is a canonical isomorphism $\HO_c(\mathbb{A}^1,\phi_{x^2}\mathbb{Q}_{\mathbb{A}^1})\cong \mathbb{Q}[-1]$ given by picking one of the two natural basis elements of $\HO_c(\mathbb{A}^1,\psi_{x^2}\mathbb{Q}_{\mathbb{A}^1})\cong\mathbb{Q}\oplus\mathbb{Q}$.  In particular, since $\mathbb{A}^1$ carries a canonical orientation, there is a canonical choice of $r$ and thus of $l$.  We set
\begin{equation}
\label{tensinv}
\mathbb{L}^{-1/2}\colonequals \mathcal{H}om_{\Db(\MMHM(\pt))}(\mathbb{L}^{1/2},\mathbb{Q}),
\end{equation}
so there is a canonical isomorphism
\[
\mathbb{L}^{1/2}\otimes\mathbb{L}^{-1/2}\cong\mathbb{Q}
\]
with target the monoidal unit in $\MMHM(\pt)$.
\smallbreak

Since $\LL$ is pure, we deduce from Remark \ref{convProdDef} that $\LL^{1/2}$ is pure, concentrated in cohomological degree 1.  Note that due to the definition of the category of monodromic mixed Hodge modules on a point (i.e. monodromic mixed Hodge structures), this cohomological degree is with respect to the perverse t-structure on the underlying derived category of constructible sheaves on $\mathbb{A}^1\cong\mathbb{A}^1\times \pt$.

If $\mathcal{F}\in\MMHM(X)$ and $\mathcal{G}\in\MMHM(\pt)$ we use the abbreviation
\[
\mathcal{F}\otimes\mathcal{G}\colonequals \mathcal{F}\otimes 
\tau^*\mathcal{G},
\]
where $\tau\colon X\rightarrow \pt$ is the map to a point.  

For $X$ a variety and $d\in\mathbb{Z}$ we define 
\[
\LL_X^{d/2}=\LL^{d/2}\otimes \mathbb{Q}_X.
\]
If $X$ is a smooth equidimensional variety, we define 
\[
\ICS_{X}(\mathbb{Q})\colonequals \LL_X^{-\dim(X)/2}.  
\]
This complex of monodromic mixed Hodge modules is pure, and an element of $\MMHM(X)$, again by Remark \ref{convProdDef}.  If $X$ is not smooth, but is irreducible, we define the monodromic mixed Hodge module $\ICS_X(\mathbb{Q})\colonequals \LL^{-\dim(X_{\reg})/2}\otimes \ICSn_X(\mathbb{Q}_{X_{\reg}})$, where $\ICSn_X(\mathbb{Q}_{X_{\reg}})$ is the usual intersection cohomology mixed Hodge module complex of $X$ given by intermediate extension of $\mathbb{Q}_{X_{\reg}}$.   Since $\ICSn_X(\mathbb{Q}_{X_{\reg}})$ is concentrated in cohomological degree $\dim(X_{\reg})$, it follows that $\ICS_X(\mathbb{Q})$ is indeed a monodromic mixed Hodge module.  If $X$ is a finite disjoint union of irreducible varieties, we define 
\[
\ICS_X(\mathbb{Q})\colonequals \bigoplus_{Z\in\pi_0(X)}\ICS_Z(\mathbb{Q}).
\]
\begin{remark}
\label{finSD}
Because of the shift in the definition of $\ICS_X(\mathbb{Q})$, there are natural isomorphisms $\DD^{\mon}_X\ICS_X(\mathbb{Q})\cong\ICS_X(\mathbb{Q})$ and $\DD^{\mon}_X\phim{f}\ICS_X(\mathbb{Q})\cong\phim{f}\ICS_X(\mathbb{Q})$ for $X$ a finite disjoint union of irreducible varieties, and $f$ a regular function on $X$.
\end{remark}

For $X$ a smooth equidimensional variety we define 
\begin{align*}
\HO(X,\mathbb{Q})_{\vir}\colonequals &\HO(X,\ICS_X(\mathbb{Q}))\\\cong&\LL^{-\dim(X)/2}\otimes\HO(X,\mathbb{Q})\in\Db(\MMHM(\pt)),
\end{align*}
and since in the stack-theoretic sense the dimension of  $\dim(\BC)$ is $-1$, we extend this notation by defining 
\[
\HO(\textrm{B}\mathbb{C}^*,\mathbb{Q})_{\vir}\colonequals \LL^{1/2}\otimes \HO(\textrm{B}\mathbb{C}^*,\mathbb{Q})\in \Dub(\MMHM(\pt)).
\]

\subsection{Equivariant vanishing cycles}
\label{TotCon}
We do not propose here to write down a full six functor and vanishing cycles functor formalism for equivariant monodromic mixed Hodge modules via a theory of monodromic mixed Hodge modules on stacks.  Instead we mimic the constructions of \cite{BL94} and Totaro's approximation \cite{To99} of the Chow ring of classifying spaces by finite-dimensional algebraic approximations  to produce only the definitions and constructions we will need for the rest of this paper.
\smallbreak
We will want to be able to work with equivariant monodromic mixed Hodge modules, in the following generality.  We assume that we have a $G$-action on a smooth algebraic variety $X$, for $G$ an affine algebraic group, and a regular function $f$ on the stack $X/G$, i.e. a $G$-invariant regular function on $X$.  Furthermore we assume that we are given a map of stacks $p\colon X/G\rightarrow Y$, where $Y$ is a locally finite type scheme.  We will also assume that $G$ is special, i.e. all \'etale locally trivial principal $G$-bundles are Zariski locally trivial --- see \cite[Def.2.1]{JoyceMF} for a concise discussion of this condition in this context, or \cite{Chev58} for the original references.  As noted in \cite{JoyceMF}, all such $G$ are connected.
\smallbreak
The element $\Ho\left(p_*\phim{f}\ICS_X(\mathbb{Q})\right)$ will be an element of $\Du(\MMHM(Y))$, the derived category of bounded below complexes of monodromic mixed Hodge modules on $Y$.  
\smallbreak
We define $\Ho\left(p_*\phim{f}\ICS_X(\mathbb{Q})\right)$ as follows.  Let $V_1\subset V_2\subset\ldots$ be a chain of inclusions of $G$-representations, and let $U_1\subset U_2\subset\ldots$ be an ascending chain of subvarieties of the underlying vector spaces of $V_1,V_2,\ldots$, considered as $G$-equivariant algebraic varieties.  We suppose that $G$ acts scheme-theoretically freely on each $U_i$, that each principal bundle quotient $U_i\rightarrow U_i/G$ exists in the category of schemes, and that $\codim_{V_i}(V_i\setminus U_i)\rightarrow \infty$ as $i\rightarrow\infty$.  The map $X\times U_i\rightarrow X\times_G U_i$ exists as a principal bundle quotient in the category of schemes by \cite[Prop.23]{EdGr98}.  We define $f_i\colon X\times_G U_i\rightarrow \mathbb{C}$ to be the induced map, and $\iota_i\colon X\times_G U_i\rightarrow X\times_G U_{i+1}$ to be the inclusion.

\smallbreak
To fix notation, we fix an embedding $G\subset \Gl_t(\mathbb{C})$ for some $t$.  We set 
\[
V_i\colonequals \Hom(\mathbb{C}^i,\mathbb{C}^t), 
\]
and define $U_i\subset V_i$ to be the subscheme of surjective maps; if $i\geq t$ then $U_i$ does indeed carry a free $G$-action via the $\Gl_t(\mathbb{C})$-action on $\mathbb{C}^t$.  For each $i$, we define $X_i\colonequals X\times_G U_i$.    We obtain an explicit sequence of maps
\[
p_{i+1,*}\phim{f_{i+1}}\mathbb{Q}_{X_{i+1}}\rightarrow p_{i+1,*}\iota_{i,*}\phim{f_i}\mathbb{Q}_{X_i}= p_{i,*}\phim{f_i}\mathbb{Q}_{X_i}
\]
and so a sequence of mixed Hodge modules 
\[
\mathcal{F}_i\colonequals p_{i,*}\phim{f_i}\QQ_{X_i}
\]  
and morphisms $\mathcal{F}_{i'}\rightarrow \mathcal{F}_{i}$ for $i'>i$.
\begin{proposition}
\label{stabProp}
Fix $n\in\mathbb{N}$.  Then for $i\gg 0$ the map
\begin{equation}
\label{evIs}
\Lambda\colon \mathcal{H}^{n}(Y,\mathcal{F}_{i+1})\rightarrow \mathcal{H}^n(Y,\mathcal{F}_i)
\end{equation}
in $\MMHM(Y)$ is an isomorphism.  
\end{proposition}
\begin{proof}
Since the functor $\form{Y}\colon \MMHM(Y)\rightarrow\Perv(Y)$ of equation (\ref{formDef}) is faithful, it suffices to show that (\ref{evIs}) induces an isomorphism at the level of perverse sheaves.   So let $\mathcal{F}_i$ denote the complex of perverse sheaves $p_{i,*}\phi_{f_i}\mathbb{Q}_{X_i}$.  Say we can prove the same proposition, but with the map $\Lambda$ replaced by the map
\[
\Lambda_{\con}\colon \mathcal{H}_{\con}^{n}(Y,\mathcal{F}_{i+1})\rightarrow \mathcal{H}_{\con}^n(Y,\mathcal{F}_i)
\]
of constructible sheaves on $Y$ (here and below we use $\Ho_{\con}$ to denote constructible cohomology sheaves).  Then for sufficiently large $i$, 
\[
\cone(\mathcal{F}_{i+1}\rightarrow\mathcal{F}_i)\in \mathcal{D}_{\con}^{\geq n+1}(Y)\subset \Dp^{\geq n+1}(Y)
\]
and the original proposition follows.  

Consider the space $U_{i,i+1}\subset \Hom(\mathbb{C}^{i+1},\mathbb{C}^t)$ of linear maps which are surjective after precomposing with the map 
\[
\pi_{i,i+1}\colon \mathbb{C}^{i+1}\rightarrow \mathbb{C}^{i+1}
\]
given by $(z_1,\ldots,z_{i+1})\mapsto (z_1,\ldots,z_i,0)$.  This subspace is open and dense for $i\geq t$.  We denote by $j_i$ the inclusion
\[
j_i\colon U_{i,i+1}\hookrightarrow U_{i+1}.  
\]
Note that $G$ acts freely on $U_{i,i+1}$, and there is a $G$-equivariant affine fibration 
\begin{align*}
\tau_{i,i+1}\colon &U_{i,i+1}\rightarrow U_i\\
&(\tau_{i,i+1}g)(z)\colonequals  g(z,0).
\end{align*}

Define $X_{i,i+1}\colonequals X\times_G U_{i,i+1}$, denote by $f_{i,i+1}\colon  X_{i,i+1}\rightarrow \mathbb{C}$ the function induced by $f$, and denote by 
\begin{equation}
\label{taffDef}
\iota_{i,i+1}\colon X_{i,i+1}\rightarrow X_{i+1}
\end{equation}
the open embedding induced by $j_{i}$.  The projection $\tau_{i,i+1}$ induces an affine fibration 
\[
t_{i,i+1}\colon X_{i,i+1}\rightarrow X_{i}
\]
and we define the induced morphism
\begin{equation}
p_{i,i+1}\colon X_{i,i+1}\xrightarrow{t_{i,i+1}}X_i\xrightarrow{p_i} Y.
\end{equation}
We factorise $\Lambda_{\con}$ as the composition
\[
\Ho^n_{\con}(p_{i+1,*}\phi_{f_{i+1}}\mathbb{Q}_{X_{i+1}})\xrightarrow{a} \Ho^n_{\con}(p_{i+1,*}\iota_{i,i+1,*}\phi_{f_{i,i+1}}\mathbb{Q}_{X_{i,i+1}})\rightarrow \Ho^n_{\con}(p_{i,*}\phi_{f_i}\mathbb{Q}_{X_i})
\]
where the second map is an isomorphism by homotopy invariance.  So it is enough to show that the map $a$ is an isomorphism for sufficiently large $i$.  

Consider the function $\overline{f}_{i+1}\colon X\times U_{i+1}\rightarrow \mathbb{C}$ given by the composition $X\times U_{i+1}\xrightarrow{\pi } X\xrightarrow{f}\mathbb{C}$, and define $\overline{f}_{i,i+1}\colon X\times U_{i,i+1}\rightarrow \mathbb{C}$ similarly.  Let $\overline{\iota}_{i,i+1}\colon  X\times U_{i,i+1}\rightarrow X\times U_{i+1}$ be the inclusion.  Then 
\[
\phi_{\overline{f}_{i+1}}\mathbb{Q}_{X\times U_{i+1}}\cong\phi_{f}\mathbb{Q}_{X}\boxtimes \mathbb{Q}_{U_{i+1}}
\]
and 
\[
\left(\phi_{\overline{f}_{i+1}}\mathbb{Q}_{X\times U_{i+1}}\rightarrow\overline{\iota}_{i,i+1,*}\phi_{\overline{f}_{i,i+1}}\mathbb{Q}_{X\times U_{i,i+1}}\right)=\id_{\phi_{f}\mathbb{Q}_{X}}\boxtimes (\mathbb{Q}_{U_{i+1}}\rightarrow j_{i,*}\mathbb{Q}_{U_{i,i+1}}).
\]
For fixed $m$, $\Ho_{\con}^m(\mathbb{Q}_{U_{i+1}})\rightarrow \Ho_{\con}^m(\mathbb{Q}_{U_{i,i+1}})$ is an isomorphism for sufficiently large $i$, since the codimension of $U_{i+1}\setminus U_{i,i+1}$ inside $U_{i+1}$ goes to infinity as $i$ goes to infinity.  Since the external tensor product is exact, we deduce that for fixed $m$ and sufficiently large $i$,
\[
\Ho^m_{\con}(\phi_{\overline{f}_{i+1}}\mathbb{Q}_{X\times U_{i+1}}\rightarrow\overline{\iota}_{i,i+1,*}\phi_{\overline{f}_{i,i+1}}\mathbb{Q}_{X\times U_{i,i+1}})
\]
is an isomorphism.  On the other hand, taking a Zariski open subspace $C\subset X_{i+1}$ such that the principal bundle $\overline{C}=C\times_{X_{i+1}}(X\times U_{i+1})\rightarrow C$ is trivial, we have that
\begin{align*}
&\Ho_{\con}(\phi_{\overline{f}_{i+1}}\mathbb{Q}_{X\times U_{i+1}}\rightarrow\overline{\iota}_{i,i+1,*}\phi_{\overline{f}_{i,i+1}}\mathbb{Q}_{X\times U_{i,i+1}})|_{\overline{C}}\cong\\& \Ho_{\con}(\phi_{f_{i+1}}\mathbb{Q}_{X_{i+1}}\rightarrow\iota_{i,i+1,*}\phi_{f_{i,i+1}}\mathbb{Q}_{X\times_{G} U_{i,i+1}}))\boxtimes\id_{\mathbb{Q}_{G}},
\end{align*}
and so using exactness of external tensor product again, we deduce that for fixed $n$ and sufficiently large $i$, the map 
\[
\Ho^n_{\con}(\phi_{f_{i+1}}\mathbb{Q}_{X_{i+1}}\rightarrow\iota_{i,i+1,*}\phi_{f_{i,i+1}}\mathbb{Q}_{X_{i,i+1}}))
\]
is an isomorphism.  The pushforward $p_{i+1,*}$ maps $\mathcal{D}_{\con}^{\geq n+1}(X_{i+1})\rightarrow \mathcal{D}_{\con}^{\geq n+1}(Y)$, and so 
\[
\cone\left(p_{i+1,*}\phi_{f_{i+1}}\mathbb{Q}_{X_{i+1}}\rightarrow p_{i+1,*}\iota_{i,i+1,*}\phi_{f_{i,i+1}}\mathbb{Q}_{X_{i,i+1}}\right)\in\mathcal{D}_{\con}^{\geq n+1}(X)
\]
and the map $a$ is an isomorphism for sufficiently large $i$, as required.
\end{proof}

\begin{proposition}
\label{Hpprop}
The complex
\begin{equation}
\label{Hpdef}
\ldots\xrightarrow{0}\lim_{i\mapsto\infty}\Ho^{n-1}(p_{i,*}\phim{f_i}\mathbb{Q}_{X_i})\xrightarrow{0}\lim_{i\mapsto\infty}\Ho^{n}(p_{i,*}\phim{f_i}\mathbb{Q}_{X_i})\xrightarrow{0}\ldots.
\end{equation}
is well defined up to canonical isomorphism in $\Dub(\MMHM(Y))$.
\end{proposition}
\begin{proof}
The proposition says that if, for $l=1,2$ we have two choices $V^{(l)}_i$, and $U^{(l)}_i$ for the spaces $V_i$ and $U_i$ in the definition of $\Ho^n(p_{i,*}\phim{f_i}\mathbb{Q}_{X_i})$, and we define $X^{(l)}_i$, $p^{(l)}_i$ etc. with reference to them, then there is a canonical isomorphism, for each $n$, and for sufficiently large $i$:
\[
\Ho^n(p^{(1)}_{i,*}\phim{f^{(1)}_i}\mathbb{Q}_{X^{(1}_i})\rightarrow\Ho^n(p^{(2)}_{i,*}\phim{f^{(2)}_i}\mathbb{Q}_{X^{(2)}_i}).
\]
By another application of \cite[Prop.23]{EdGr98}, $X_i\times (U^{(1)}_i\times U^{(2)}_i)$ is a principal $G$-bundle over $X_i^{(12)}\colonequals X_i\times_{G_i}(U^{(1)}_i\times U^{(2)}_i)$.  We denote by $p^{(12)}_i\colon X_i^{(12)}\rightarrow Y$ the projection; note that it factors through both $p^{(1)}_i$ and $p^{(2)}_i$.  Via the same argument as Proposition \ref{stabProp}, each of the morphisms
\begin{align*}
\tau\colon\Ho^n(p^{(1)}_{i,*}\phim{f^{(1)}_i}\mathbb{Q}_{X^{(1)}_i})\rightarrow &\Ho^n(p^{(12)}_{i,*}\phim{f^{(12)}_i}\mathbb{Q}_{X^{(12)}_i})\\
\kappa\colon\Ho^n(p^{(2)}_{i,*}\phim{f^{(2)}_i}\mathbb{Q}_{X^{(2}_i})\rightarrow &\Ho^n(p^{(12)}_{i,*}\phim{f^{(12)}_i}\mathbb{Q}_{X^{(12)}_i})
\end{align*}
is an isomorphism, and $\sigma^{(12)}=\kappa^{-1}\tau$ is the desired isomorphism.  From the obvious commutative diagrams of projections between different principal $G$-bundles, there are equalities $\sigma^{(23)}\sigma^{(12)}=\sigma^{(13)}$ and $\sigma^{(11)}=\id$.
\end{proof}

We denote by $\Ho\left(p_*\phim{f}\mathbb{Q}_X\right)$ the element defined in Proposition \ref{Hpprop}.
\smallbreak
If $\Mst=\coprod_{s\in S} \Mst_s$, with $\Mst_s\cong X_s/G_s$ and each $X_s$ a smooth algebraic variety, and each $G_s$ a special affine algebraic group, and $p\colon \Mst\rightarrow Y=\coprod_{s\in S} Y_s$ is a map to a locally finite type scheme respecting the decompositions of the source and the target, for $c\in\ZZ$ we define 
\begin{align*}
\Ho\left(p_*\phim{f}\bigoplus_{s\in S}\mathbb{Q}_{\Mst_s}\right)\colonequals &\bigoplus_{s\in S} \Ho\left(p_*\phim{f|_{X_s/G_s}}\mathbb{Q}_{X_s/G_s}\right)\\
\Ho\left(p_*\phim{f}\bigoplus_{s\in S}\LL^{c/2}_{\Mst_s}\right)\colonequals &\bigoplus_{s\in S} \Ho\left(p_*\phim{f|_{X_s/G_s}}\LL^{c/2}_{X_s/G_s}\right)
\end{align*}
and
\[
\Ho\left(p_*\phim{f}\bigoplus_{s\in S}\ICS_{\Mst_s}(\mathbb{Q})\right)\colonequals \bigoplus_{s\in S}\Ho\left(p_*\phim{f|_{X_s/G_s}}\mathbb{L}^{(\dim(G_s)-\dim(X_s))/2}_{X_s/G_s}\right).
\]

We define $\Ho\left(p_!\phim{f}\bigoplus_{s\in S}\mathbb{Q}_{\Mst_s}\right)\in \Dub(\MMHM(Y))$ in similar fashion.  We give the definition in the case that $\Mst=X/G$, for $X$ a smooth irreducible variety and $G$ a special affine algebraic group, and extend to the case of a disjoint union of such stacks as above. 
\smallbreak
Consider the isomorphism \[\Ho(\iota_{i,i+1,!}\phim{f_{i,i+1}}\mathbb{Q}_{X_{i,i+1}})\cong \Ho(\phim{f_{i}}\mathbb{Q}_{X_{i}})\otimes\mathbb{L}^{\dim(X_{i+1})-\dim(X_i)}\] where $\iota_{i,i+1}$ is as defined in (\ref{taffDef}).  Arguing as in Proposition \ref{stabProp}, for fixed $n$ the natural map 
\[
\Ho_{\con}^{2\dim(U_{i+1})+n}((X_{i,i+1}\rightarrow X_{i+1})_!\phim{f_{i,i+1}}\mathbb{Q}_{X_{i,i+1}})\rightarrow \Ho_{\con}^{2\dim(U_{i+1})+n}(\phim{f_{i+1}}\mathbb{Q}_{X_{i+1}})
\]
is an isomorphism for sufficiently large $i$, and we can define $\Ho^n\left(p_!\phim{f}\mathbb{Q}_{X/G}\right)$ to be the limit of $\Ho^n\left(p_{i,!}\phim{f_i}\mathbb{L}_{X_i}^{-\dim(U_i)}\right)$ as $i$ tends to infinity, and we set
\begin{align*}
\Ho^n\left(p_!\phim{f}\mathbb{L}^c_{X/G}\right)\colonequals &\LL^c\otimes\Ho^n\left(p_!\phim{f}\QQ_{X/G}\right)\\
\Ho\left(p_!\phim{f}\ICS_{X/G}(\mathbb{Q})\right)\colonequals &\Ho\left(p_!\phim{f}\LL^{(\dim(G)-\dim(X))/2}_{X/G}\right).
\end{align*}

\begin{remark}
Note that we do not offer here a definition of $p_*\phi_{\mon,f}\mathbb{Q}_{X/G}$, or $p_!\phi_{\mon,f}\mathbb{Q}_{X/G}$ as objects of the derived category, but instead limit ourselves to defining the total cohomology of these direct images with respect to the natural t-structure on $\Dub(\MMHM(Y))$.  Of course in case $X/G$ is an actual scheme, $p_*\phi^{\mon}_{f}\mathbb{Q}_{X/G}$ and $p_!\phi^{\mon}_{f}\mathbb{Q}_{X/G}$ are well defined monodromic mixed Hodge modules before passing to cohomology, and our definitions of $\Ho\left(p_*\phim{f}\mathbb{Q}_{X/G}\right)$ and $\Ho\left(p_!\phim{f}\mathbb{Q}_{X/G}\right)$ recover their respective total perverse cohomology.
\end{remark}

\begin{remark}
\label{remAbs}
As a special case of the above construction, if $\tau\colon \rightarrow \mathbb{N}^{Q_0}$ is a stack over a monoid of dimension vectors, with each $\Mst_{\dd}\cong X_{\dd}/G_{\dd}$, we define
\begin{align*}
\HO\left(\Mst,\bigoplus_{\dd\in\mathbb{N}^{Q_0}}\phim{f}\QQ_{\Mst_{\dd}}\right)\colonequals \bigoplus_{\dd\in\mathbb{N}^{Q_0}}\Ho\left(\tau_*\phim{f}\QQ_{X_{\dd}/G_{\dd}}\right)
\end{align*}
and $\HO(\Mst,\bigoplus_{\dd}\phim{f}\ICS_{\Mst_{\dd}}(\mathbb{Q}))$, $\HO_c(\Mst,\bigoplus_{\dd}\phim{f}\ICS_{\Mst_{\dd}}(\mathbb{Q}))$ etc.\! similarly.  This is the definition used in \cite[Sec.7]{KS2} in the definition of the underlying $\mathbb{N}^{Q_0}$-graded monodromic mixed Hodge structure of the critical cohomological Hall algebra.
\end{remark}

\begin{definition}
Let $X/G\rightarrow\mathbb{N}^{Q_0}$ be as in Proposition \ref{Hpprop}, and let $\upsilon\colon X^{\SP}/G\rightarrow X/G$ be an inclusion of stacks.  We denote by 
\[
\upsilon_i\colon X^{\SP}\times_G U_i\rightarrow X\times_G U_i
\]
the induced map, and by $p^{\SP}_i$ the restriction of $p_i$ to $X^{\SP}\times_G U_i$.  We define the restricted absolute critical cohomology as the complex of limits with zero differential
\begin{equation}
\HO\left(X^{\SP}/G,\phim{f}\QQ_{X/G}\right)\colonequals \bigoplus_{n\in\mathbb{Z}}\lim_{i\mapsto\infty}\Ho^{n}\left((Y\rightarrow \pt)_*p^{\SP}_{i,*}\upsilon_i^*\phim{f_i}\mathbb{Q}_{X_i}\right)[-n]
\end{equation}
and likewise we define 
\begin{align*}
\HO_c\left(X^{\SP}/G,\phim{f}\QQ_{X/G}\right)\colonequals \bigoplus_{n\in\mathbb{Z}}\lim_{i\mapsto\infty}\Ho^{n}\left((Y\rightarrow \pt)_!p^{\SP}_{i,!}\upsilon_i^*\phim{f_i}\LL^{-\dim(U_i)}_{X_i}\right)[-n].
\end{align*}

\end{definition}

Let $h\colon X'\rightarrow X$ be a morphism of $G$-equivariant varieties over the $G$-invariant locally finite type scheme $Y$.  Then for each $i$ we obtain maps $h_i\colon X'_i\rightarrow X_i$.  There are natural maps 
\[
\delta_1\colon p_{i,*}\phim{f_i}h_{i,*}\mathbb{Q}_{X'_i}\rightarrow p_{i,*}h_{i,*}\phim{f_i\circ h_i}\mathbb{Q}_{X'_i}
\]
which are isomorphisms if $h$ is proper or an affine fibration (see Proposition \ref{basicfacts}).  We  precompose with the natural map
\[
\delta_2\colon p_{i,*}\phim{f_i}\mathbb{Q}_{X_i}\rightarrow p_{i,*}\phim{f_i}h_{i,*}\mathbb{Q}_{X'_i}
\]
and let $i\mapsto \infty$, defining maps 
\[
\Ho\left(p_*\phim{f}\mathbb{Q}_{X/G}\right)\rightarrow \Ho\left((p h)_*\phim{f\circ h}\mathbb{Q}_{X'/G}\right).  
\]
Now assume that $X'$ and $X$ are smooth, or that $h$ is.  Taking $\Ho(\DD_Y(\delta_1\circ \delta_2))$ and letting $i\mapsto\infty$, we define the map
\[
\Ho\left((p h)_!\phim{f\circ h}\mathbb{Q}_{X'/G}\right)\rightarrow\LL^{\dim(X')-\dim(X)}\otimes \Ho\left(p_!\phim{f}\mathbb{Q}_{X/G}\right).
\]

Finally, let $\upsilon\colon H\rightarrow G$ be an inclusion of groups --- the only examples we will consider are when $\upsilon$ is the inclusion of a parabolic subgroup inside $\Gl_n(\mathbb{C})$, or the inclusion $L\subset P$ of the Levi subgroup of a parabolic subgroup of $\Gl_n(\mathbb{C})$.  Let $G$ act on $X$ as above, with $f$ a $G$-invariant regular function on $X$.  Let $h\colon X/H\rightarrow X/G$ be the associated morphism of stacks.  Then we obtain maps 
\begin{align*}
X\times_H U_i\xrightarrow{h_i} &X\times_G U_i\\
(x,z)\mapsto&(x,z)
\end{align*}
which we use in the same way as above to obtain the map
\[
\Ho\left(p_*\phim{f}\mathbb{Q}_{X/G}\right)\rightarrow\Ho\left((p h)_*\phim{f}\mathbb{Q}_{X/H}\right)
\]
and the map 
\[
\Ho\left((p h)_!\phim{f}\mathbb{Q}_{X/H}\right)\rightarrow \Ho\left(p_!\phim{f}\mathbb{L}^{\dim(G)-\dim(H)}_{X/G}\right).  
\]
\begin{remark}
\label{limKer}
If $H\hookrightarrow G$ is the inclusion of a parabolic subgroup, then the induced maps $h_i$ are proper.  As such, there is a natural isomorphism $\phim{f_i}h_{i,*}\QQ_{X\times_H U_i}\cong h_{i,*}\phim{f_i}\QQ_{X\times_H U_i}$.  From the maps $h_{i,*}\DD_{X\times_H U_i}\mathbb{Q}_{X\times_H U_i}\rightarrow \DD_{X\times_G U_{i}}\mathbb{Q}_{X\times_G U_{i}}$ given by the natural isomorphism $h_{i,*}\cong h_{i,!}$ and Verdier duality, we obtain, in the limit, the map 
\[
\Ho\left((ph)_*\phim{f}\mathbb{Q}_{X/H}\right)\rightarrow \Ho\left(p_*\phim{f}\mathbb{L}^{\dim(G)-\dim(H)}_{X/G}\right).
\]
\end{remark}

\section{Moduli spaces of quiver representations}
\label{Qrepsection}
\subsection{Basic notions}\label{Bassection}We will use the notations and conventions from \cite{DaMe4}, which we briefly recall.  Let $Q=(Q_0,Q_1,s,t)$ denote a quiver, that is, a pair of finite sets $Q_0$ and $Q_1$, and a pair of maps $s\colon Q_1\rightarrow Q_0$ and $t\colon Q_1\rightarrow Q_0$, taking an arrow to its source and target, respectively.  Denote by $\mathbb{C} Q$ the free path category of $Q$ over $\mathbb{C}$.  Alternatively we may think of $\mathbb{C} Q$ as the free path algebra of $Q$, with a distinguished family of mutually orthogonal idempotents $e_i$ in bijection with the vertices $Q_0$, summing to $1_{\mathbb{C} Q}$.
\smallbreak

Let $\Sch{\mathbb{C}}$ be the category of schemes over $\Spec(\mathbb{C})$.  For $S\in\Sch{\mathbb{C}}$ we denote by $\Vect_{S}^{\fd}$ the category of finite rank vector bundles over $S$.  Let $\dd\in\mathbb{N}^{Q_0}$ be a dimension vector.  We denote by $\Mst_{\dd}$ the groupoid valued functor on $\Sch{\mathbb{C}}$ that assigns to $S\in\Sch{\mathbb{C}}$ the groupoid obtained from forgetting the non-invertible morphisms in the full subcategory of functors in $\Fun(\mathbb{C} Q,\Vect_{S}^{\fd})$ such that $i\in Q_0$ is sent to a vector bundle of dimension $\dd_i$.  This prestack is an Artin stack, as it is represented by the following global quotient stack.  First define 
\begin{equation}
\label{Xdef}
X_{\dd}\colonequals \prod_{a\in Q_1}\Hom(\mathbb{C}^{\dd_{s(a)}},\mathbb{C}^{\dd_{t(a)}}).
\end{equation}
This affine space carries the change of basis action of 
\[
G_{\dd}\colonequals \prod_{i\in Q_0}\Aut(\mathbb{C}^{\dd_i}), 
\]
and there is an equivalence of stacks $\Mst_{\dd}\cong X_{\dd}/G_{\dd}$.  We denote by $\Mst$ the union $\coprod_{\dd\in\mathbb{N}^{Q_0}} \Mst_{\dd}$, the stack of finite-dimensional representations of $Q$, which by the equivalences just given is a countable disjoint union of finite type global quotient Artin stacks.
\smallbreak
In numerous instances we wish to take cohomology of sheaves restricted to a substack of $\Mst_{\dd}$.  To that end, we assume that for each $\dd\in\mathbb{N}^{Q_0}$ we are given reduced subschemes 
\begin{equation}
\label{Xspdef}
X^{\SP}_{\dd}\subset X_{\dd}
\end{equation}
which are preserved by the action of $G_{\dd}$, and such that the set $\mathcal{S}$ of $\mathbb{C}Q$-representations that are isomorphic to representations parametrised by points of $X^{\SP}=\coprod X_{\dd}^{\SP}$ are the objects in a Serre subcategory of the category of finite-dimensional $\mathbb{C}Q$ representations.  In other words, if 
\begin{equation}
\label{rhoses}
0\rightarrow \rho'\rightarrow \rho\rightarrow\rho''\rightarrow 0
\end{equation}
is a short exact sequence of $\mathbb{C}Q$-representations, $\rho'$ and $\rho''$ are in $\mathcal{S}$ if and only if $\rho$ is in $\mathcal{S}$.  We set $\Mst^{\SP}\subset \Mst$ to be the reduced substack, the closed points of which correspond to modules in $\mathcal{S}$.  Then $\Mst^{\SP}_{\dd}=X^{\SP}_{\dd}/G_{\dd}$.  We denote by 
\begin{align*}
\omega_{\dd}\colon &X_{\dd}^{\SP}\rightarrow X_{\dd}\\
\overline{\omega}_{\dd}\colon &\Mst_{\dd}^{\SP}\rightarrow\Mst_{\dd}
\end{align*}
the natural inclusions.
\smallbreak
For the rest of the paper, $X_{\dd}$ and $X^{\SP}_{\dd}$ will be as in (\ref{Xdef}) and (\ref{Xspdef}) respectively.  Where we wish to be specific regarding the quiver with respect to which these spaces are defined, we will instead use the notation $X(Q)_{\dd}$ or $X(Q)^{\SP}_{\dd}$.
\smallbreak
For $\dd'+\dd''=\dd\in\mathbb{N}^{Q_0}$ let $\Mst_{\dd',\dd''}(S)$ be the groupoid of triples $(F',F,\iota)$, where $F',F\in\Fun(\mathbb{C}Q,\Vect_S^{\fd})$ and $\iota\colon F'\rightarrow F$ is a natural transformation, such that $\rank(F'(i))=\dd'_i$, $\rank(F(i))=\dd_i$, and $\iota(i)$ is injective, with locally free cokernel, for every $i$.  Again, $\Mst_{\dd',\dd''}$ is a finite type Artin stack, which can be described as follows.  Let $X_{\dd',\dd''}\subset X_{\dd}$ be the subspace of representations such that the flag $\mathbb{C}^{\dd'_i}\subset \mathbb{C}^{\dd_i}$ is preserved for all $i\in Q_0$, and let $G_{\dd',\dd''}\subset G_{\dd}$ be the subgroup preserving these same flags.  Then 
\[
\Mst_{\dd',\dd''}\cong X_{\dd',\dd''}/G_{\dd',\dd''}.
\]
\smallbreak
We likewise define $\Mst_{\dd',\dd''}^{\SP}\subset\Mst_{\dd',\dd''}$ to be the reduced substack of the stack parametrising triples as above, for which the $\dd$-dimensional $\mathbb{C} Q$-representation $F_x$ corresponding to each of the closed points $x$ are elements of $\mathcal{S}$.  For each such $x$ there is a short exact sequence
\[
0\rightarrow F'_x\rightarrow F_x\rightarrow F''_x\rightarrow 0
\]
and our condition is equivalent to asking that $F'_x$ and $F''_x$ are in $\mathcal{S}$, since $\mathcal{S}$ defines a Serre property.  So we have
\[
\Mst^{\SP}_{\dd',\dd''}\cong X^{\SP}_{\dd',\dd''}/G_{\dd',\dd''}
\]
where
\[
X^{\SP}_{\dd',\dd''}\colonequals X_{\dd',\dd''}\cap X^{\SP}_{\dd}.
\]

\smallbreak
A tuple $\zeta=(\zeta_i)_{i\in Q_0}\in \mathbb{H}_+^{Q_0}\colonequals \{r\exp(+\sqrt{-1}\pi\phi)\in \mathbb{C}\mid r>0, 0<\phi\le 1\}^{Q_0}\subset \mathbb{C}^{Q_0}$ provides a \textit{Bridgeland stability condition}, as defined in \cite{Bridgeland02}, with central charge defined on finite-dimensional $\mathbb{C}Q$-modules
\[
Z\colon\rho\mapsto\zeta\cdot \dim (\rho)=\sum_{i\in Q_0}\zeta_i\dim (\rho_i).  
\]
We define the \textit{slope} of a nonzero representation $\rho$ by setting 
\[
\Mu^{\zeta}(\rho)\colonequals\begin{cases}- \Re e (Z(\rho))/ \Im m (Z(\rho))&\textrm{if }\Im m (Z(\rho))\neq 0\\\infty&\textrm{if }\Im m (Z(\rho))=0.\end{cases}
\]
Likewise we define $\Mu^{\zeta}(\dd)$, for $\dd\in\mathbb{N}^{Q_0}\setminus\{0\}$, to be the slope of any representation $\rho$ of dimension $\dd$.  A $\mathbb{C} Q$-representation $\rho$ is called $\zeta$\textit{-semistable} if for all proper submodules $\rho'\subset\rho$ we have $\Mu^{\zeta}(\rho')\leq\Mu^{\zeta}(\rho)$, and is $\zeta$\textit{-stable} if instead we have $\Mu^{\zeta}(\rho')<\Mu^{\zeta}(\rho)$ for every proper submodule.  To simplify notation in what follows, we will assume that all representations satisfy $\Mu^{\zeta}(\rho)<\infty$ in this paper.  In other words, we choose our Bridgeland stability conditions to belong to 
\[
\{r\exp(+\sqrt{-1}\pi\phi)\in \mathbb{C}\mid r>0, 0<\phi< 1\}^{Q_0}\subset\mathbb{H}_+^{Q_0}.
\]
\smallbreak
We define two pairings on $\mathbb{Z}^{Q_0}$:
\begin{equation*}
(\dd,\ee)\colonequals \sum_{i\in Q_0} \dd_i \ee_i-\sum_{a\in Q_1}\dd_{s(a)}\ee_{t(a)}
\end{equation*}
and
\begin{equation*}
\langle \dd,\ee\rangle\colonequals (\dd,\ee)-(\ee,\dd).
\end{equation*}
Note that $(\dd,\dd)=-\dim \Mst_{\dd}$.
\smallbreak
As in the introduction, for every $\mu\in (-\infty,\infty)$ we denote by $\Lambda_{\mu}^{\zeta}\subset\mathbb{N}^{Q_0}$ the submonoid of dimension vectors $\dd$ which are either zero, or have slope $\mu$ with respect to the stability condition $\zeta$.
\begin{definition}
\label{genericityDef}
We say $\zeta$ is $\mu$-generic if $\dd,\ee\in\Lambda_{\mu}^{\zeta}$ implies $\langle \dd,\ee\rangle=0$.  We say $\zeta$ is generic if it is $\mu$-generic for all $\mu$.  
\end{definition}
We say that $\zeta$ is a King stability condition if $\Im m(\zeta_i)=1$ and $\Re e(\zeta_i)\in\mathbb{Q}$ for all $i\in Q_0$.  Given a King stability condition $\zeta$, we can fix $m\in\mathbb{N}$ such that $m\Re e(\zeta_i)\in\mathbb{Z}$ for every $i$.  We linearize the $G_{\dd}$-action on $X_{\dd}$ via the trivial line bundle on $X_{\dd}$ and the character 
\begin{align}
\label{charDef}
\chi_{\dd}\colon &G_{\dd}\rightarrow \mathbb{C}^*\\&(g_i)_{i\in Q_0}\mapsto \prod_{i\in Q_0}\det(g_i)^{m\Re e(\zeta_i)}, \nonumber
\end{align}
and define $X_{\dd}^{\zeta\sst}$ to be the variety of semistable points with respect to this linearization.  By \cite{King}, using the constructions and definitions of \cite{MFK94}, the GIT quotient $X^{\zeta\sst}_{\dd}/\!\!/_{\chi_{\dd}} G_{\dd}$ provides a coarse moduli space of $\zeta$-semistable representations of dimension ${\dd}$, which we denote $\Msp^{\zeta\sst}_{\dd}$.  Similarly we deonote by $\Msp^{\zeta\sst,\SP}_{\dd}$ the scheme $X^{\zeta\sst,\SP}_{\dd}/\!\!/_{\chi_{\dd}} G_{\dd}$, and we denote by 
\begin{equation}
\label{tomdef}
\tilde{\omega}^{\zeta}_{\dd}\colon\Msp^{\zeta\sst,\SP}_{\dd}\rightarrow\Msp^{\zeta\sst}_{\dd}
\end{equation}
the natural inclusion.
\begin{remark}
\label{BtoK}
Fix a dimension vector $\dd\in\mathbb{N}^{Q_0}$ of slope $\mu$, for a $\mu$-generic stability condition $\zeta$.  Then we can always find a $\mu'$-generic King stability condition $\zeta'$ such that a $\dd$-dimensional $\mathbb{C} Q$-module is $\zeta$-stable if and only if it is $\zeta'$-stable, by \cite[Lem.4.21]{DMSS15}, where $\mu'$ is the slope of $\dd$ with respect to $\zeta'$.  By construction, a $\dd$-dimensional $\mathbb{C}Q$-module is $\zeta'$-semistable if and only if it is $\zeta$-semistable.  We deduce that for every $\mu$-generic Bridgeland stability condition $\zeta$ and every dimension vector $\dd$ of slope $\mu$, there is a coarse moduli space $\Msp^{\zeta\sst}_{\dd}$ of $\dd$-dimensional $\zeta$-semistable $\mathbb{C}Q$-modules.  
\end{remark}
\begin{remark}
Pick a rational slope $\mu\in (-\infty,\infty)$.  Then we can define a maximally degenerate stability King condition, for which every $\mathbb{C}Q$-module is automatically semistable, by fixing $\zeta_i=+\sqrt{-1}-\mu$ for all $i\in Q_0$.  In this case we have $\Lambda_{\mu}^{\zeta}=\mathbb{N}^{Q_0}$, and $\Mst=\Mst^{\zeta\sst}_{\mu}$.  As a result, all the results in this paper in which we do not assume that we are working with a generic stability condition apply to the case in which we do not impose any stability condition.  In addition, those results in which we do impose a genericity assumption on $\zeta$ apply to the case in which we impose no stability condition and $Q$ is symmetric, since in this case maximally degenerate stability conditions are still generic in the sense of Definition \ref{genericityDef}.
\end{remark}
We write $\Msp_{\dd}$ for the coarse moduli space, taken with respect to the above degenerate stability condition.  Then explicitly
\[
\Msp_{\dd}=\Spec\left(\Gamma(X(Q)_{\dd})^{G_{\dd}}\right)
\]
and the closed points of $\Msp_{\dd}$ are in bijection with isomorphism classes of $\dd$-dimensional semisimple $\mathbb{C}Q$-representations.
\smallbreak
For a slope $\mu$, we define 
\begin{equation}
\Msp_{\mu}^{\zeta\sst,\SP}=\coprod_{\dd\in\Lambda_{\mu}^{\zeta}}\Msp_{\dd}^{\zeta\sst,\SP}
\end{equation}
and
\begin{equation}
\Mst_{\mu}^{\zeta\sst,\SP}=\coprod_{\dd\in \Lambda_{\mu}^{\zeta}}\Mst_{\dd}^{\zeta\sst,\SP}.
\end{equation}
We denote by 
\begin{equation}
\label{pdef}
p_{\dd}^{\zeta,\SP}\colon \Mst^{\zeta\sst,\SP}_{\dd}\rightarrow \Msp^{\zeta\sst,\SP}_{\dd}
\end{equation}
and
\begin{equation}
p_{\mu}^{\zeta,\SP}\colon \Mst^{\zeta\sst,\SP}_{\mu}\rightarrow \Msp^{\zeta\sst,\SP}_{\mu}
\end{equation}
the maps from the stacks to their respective coarse moduli spaces, and by
\begin{equation}
\label{qdef}
q_{\dd}^{\zeta,\SP}\colon \Msp^{\zeta\sst,\SP}_{\dd}\rightarrow \Msp_{\dd}^{\SP}
\end{equation}
and
\begin{equation}
\label{qmudef}
q_{\mu}^{\zeta,\SP}\colon \Msp^{\zeta\sst,\SP}_{\mu}\rightarrow \Msp_{\mu}^{\SP}
\end{equation}
the maps to the subschemes of points in the union of the affinizations $\Msp_{\mu}$ corresponding to representations in $\mathcal{S}$.  Deleting all superscripts $\SP$ in the above definitions we obtain morphisms $p^{\zeta}_{\dd},q^{\zeta}_{\dd}$ etc.

\smallbreak
If $\dd',\dd''\in\Lambda_{\mu}^{\zeta}$ for some $\mu\in(-\infty,\infty)$, we denote by $\Mst_{\dd',\dd''}^{\zeta\sst,\SP}$ the stack of short exact sequences as in (\ref{rhoses}) such that $\dim(\rho')=\dd'$, $\dim(\rho'')=\dd''$ and $\rho\in\SP$.  There is an isomorphism of stacks
\[
\Mst_{\dd',\dd''}^{\zeta\sst,\SP}\cong \left(X_{\dd',\dd''}\cap X_{\dd'+\dd''}^{\zeta\sst,\SP}\right)/G_{\dd',\dd''}.
\]

\subsection{Monoidal products}
\label{prodss}
In this section we define and introduce the first properties of the categorification of the quantum torus from refined DT theory.
\begin{definition}
\label{lfDef}
Let $\mathcal{F}\in\Dub(\MMHM(X))$ for $X$ a scheme.  We say that $\mathcal{F}$ is \textit{locally finite} if for each $Z\in \pi_0(X)$
\begin{enumerate}
\item
for each $n\in\mathbb{Z}$, the element $\bigoplus_{i\in\mathbb{Z}}\GrW{n}(\Ho^i(\mathcal{F})|_Z)$ belongs to $\Db(\MMHM(Z))$
\item
For $n\ll 0$, we have $\GrW{n}(\Ho^i(\mathcal{F})|_Z)=0$ for all $i\in\mathbb{Z}$.
\end{enumerate}
Denote by $\Dulf(\MMHM(X))\subset\Du(\MMHM(X))$ the full subcategory of locally finite objects.
\end{definition}

The category $\Dulf(\MMHM(\mathbb{N}^{Q_0}))$ is going to play the role of the categorification of the motivic quantum torus of \cite[Sec.6.2]{KS1}.  For now we remark that this quantum torus is a power series ring with coefficients in $\mathbb{Z}((q^{1/2}))$ and variables $x^{\dd}$ for $\dd\in\mathbb{N}^{Q_0}$, and if $\zeta$ is a generic stability condition then the subring of power series in variables $x^{\dd}$ for $\dd\in\Lambda_{\mu}^{\zeta}$ is commutative, while the whole quantum torus in general is not.  The categorification of this picture, then, should be a monoidal category, for which the monoidal product cannot be upgraded to a symmetric monoidal product, except on fixed subcategories indexed by the slope $\mu$, for which we should define a symmetric monoidal product.

\smallbreak

The moduli scheme $\Msp_{\mu}^{\zeta\sst}$ carries a symmetric monoidal structure given by the direct sum, and the product 
\[
\Msp_{\mu}^{\zeta\sst}\times \Msp_{\mu}^{\zeta\sst}\xrightarrow{\oplus}\Msp_{\mu}^{\zeta\sst}
\]
is a finite morphism of schemes (see \cite[Lem.2.1]{Meinhardt14}).  By Remark \ref{convProdDef} the category $\Dulf(\MMHM(\mathcal{M}^{\zeta\sst}))$ carries a monoidal product
\[
\mathcal{F}\boxtimes_{\oplus}\mathcal{G}\colonequals \oplus_*(\pi_1^*\mathcal{F}\otimes\pi_2^*\mathcal{G})  
\]
defined as in (\ref{simmon}).
\begin{proposition}
\label{biExp}
The monoidal product $\boxtimes_{\oplus}$ is biexact and preserves the weight filtration.  
\end{proposition}
\begin{proof}
Let $\mathcal{F}\in\MMHM(X)$ and $\mathcal{G}\in\MMHM(Y)$ for two algebraic varieties $X$ and $Y$.  We define their external tensor product 
\[
\mathcal{F}\boxtimes \mathcal{G}\colonequals (p\colon X\times Y\times\mathbb{A}^1\times\mathbb{A}^1\xrightarrow{\id_{X\times Y}\times +} X\times Y\times \mathbb{A}^1)_*(\pr_{1,3}^*\mathcal{F}\otimes\pr_{2,4}^*\mathcal{G}).
\]
as in Remark \ref{convProdDef}.  Then by \cite[Lem.1]{KS2}, there is an isomorphism in $\Dub(\MMHM(X\times Y))$
\[
p_!\left(\pr_{1,3}^*\mathcal{F}\otimes\pr_{2,4}^*\mathcal{G}\right)\rightarrow p_*\left(\pr_{1,3}^*\mathcal{F}\otimes\pr_{2,4}^*\mathcal{G}\right)
\]
and so since $p_!\colon \Dub(\MHM(X\times \mathbb{A}^1\times Y\times\mathbb{A}^1))\rightarrow \Dub(\MHM(X\times Y\times \mathbb{A}^1))$ is left exact, and decreases weights, while $p_*$ is right exact and increases weights, we deduce that the external tensor product 
\[
\boxtimes\colon \Dub(\MMHM(X))\times\Dub(\MMHM(Y))\rightarrow \Dub(\MMHM(X\times Y))
\]
is exact and preserves weights.  
\smallbreak
Since the map $\oplus\colon \Msp_{\mu}^{\zeta\sst}\times \Msp_{\mu}^{\zeta\sst}\rightarrow \Msp_{\mu}^{\zeta\sst}$ is finite, the direct image functor $\oplus_*$ is exact and preserves weights.  The result follows.
\end{proof}
The monoidal unit in $\Dulf(\MMHM(\mathcal{M}^{\zeta\sst}))$ is $\mathbb{Q}_{{\Msp_0^{\zeta\sst}}}$, the constant pure Hodge module supported on $\Msp_0^{\zeta\sst}\cong \pt$, the unit of the monoid $\Msp_{\mu}^{\zeta\sst}$.  

\smallbreak

We briefly summarize the results of Maxim, Saito and Sch\"urmann \cite{MSS11}, dealing with exterior products and symmetric group actions on mixed Hodge modules.  Let $X_i$, $i\in I$ be a finite set of algebraic varieties, let $\mathcal{L}_i$ be objects of $\Dulf(\MHM(X_i))$ for each $i\in I$, and let $f_i\colon X_i\rightarrow Y_i$ be a set of morphisms of algebraic varieties.  Set $f=\coprod_{i\in I}f_i$.  While \cite{MSS11} treats the bounded derived category, and certainly it makes no sense to apply their results at the generality of the unbounded derived category, their arguments extend without modification to the case of the categories $\Dulf(\MHM(X_i))$.  They show that external tensor product commutes with the direct image, that is, there is a natural isomorphism
\[
\boxtimes_{i\in I}f_{i,*}\mathcal{L}_i\cong f_*\left(\boxtimes_{i\in I}\mathcal{L}_i\right).
\]
Furthermore, if $X_i=X$ for all $i$, and $Y_i=Y$ for all $i$, then there is an action of the symmetric group $\SSym_I$ on both external products, which is respected by the above natural isomorphism.  
\begin{remark}
This is the usual action, incorporating the Koszul sign rule --- when we consider cohomological Hall algebras later, we will have cause to modify this action, to express commutativity results for the cohomological Hall algebra.  However, for the purposes of Theorems \ref{ThmA} and \ref{CWCT} we may equivalently use the default action of \cite{MSS11} as we are only interested in the isomorphism classes of underlying objects (see Proposition \ref{underlying_same}).
\end{remark}
Now let $\mathcal{L}_i=\mathcal{L}$ for all $i$, and let $I=\{1,\ldots,n\}$.  The authors then define $\Sym^n(\mathcal{L})\colonequals (\varpi_*\boxtimes^n\!\mathcal{L})^{\SSym_n}$, where $\varpi\colon X^n\rightarrow \Sym^nX$ is the natural map, and show that there is a canonical isomorphism 
\[
\HO(\Sym^n X,\Sym^n\mathcal{L})\cong \Sym^n(\HO(X,\mathcal{L})).
\]
Furthermore, they show that the Kunneth isomorphism
\[
\HO(X^n,\boxtimes^n\mathcal{L})\cong \bigotimes^n\HO(X,\mathcal{L})
\]
is an isomorphism of $\SSym_n$-representations.  

Now let $\mathcal{L}\in \Dulf(\MMHM(X))$.  Since there is a commutative diagram
\[
\xymatrix{
(X\times \mathbb{A}^1)^n\ar[d]_{\varpi_{X\times\mathbb{A}^1}}\ar[r]^-{\id_X^n\times +}&X^n\times\mathbb{A}^1\ar[d]^{\varpi_X\times \id_{\mathbb{A}^1}}\\
\Sym^n(X\times\mathbb{A}^1)\ar[r]^r&\Sym^nX\times\mathbb{A}^1
}
\]
we obtain a $\SSym_n$-action on $\varpi_{X,*}(\boxtimes^n\mathcal{L})$ by taking $r_*$ of the action defined in \cite{MSS11}, and we can define $\Sym^n\!\mathcal{L}$ to be the $\SSym_n$-invariant summand.  Similarly, we define the functor
\begin{align*}
\Sym^n_{\boxtimes_{\oplus}}\colon&\Dulf(\MMHM(\Msp_{\mu}^{\zeta\sst}))\rightarrow \Dulf(\MMHM(\Msp_{\mu}^{\zeta\sst}))\\
&\mathcal{F}\mapsto (\oplus_*\mathcal{F}^{\boxtimes n})^{\SSym_n}.
\end{align*}
We then define
\begin{align}
\label{symdef}
\FreeComm_{\boxtimes_{\oplus}}\colon&\Du(\MMHM(\Msp^{\zeta\sst}_{\mu}\setminus \Msp_0^{\zeta\sst}))\rightarrow \Du(\MMHM(\Msp^{\zeta\sst}_{\mu}))\\ \nonumber
&\mathcal{F}\mapsto\bigoplus_{n\geq 0}\Sym^n_{\boxtimes_{\oplus}}\mathcal{F}.
\end{align}
This functor is well defined, since 
\[
\pi_0(\oplus)\colon \pi_0(\Msp)\times\pi_0(\Msp)\rightarrow\pi_0(\Msp)
\]
has finite fibers.  The same fact implies the following lemma:
\begin{lemma}
\label{lfRem}
The functor (\ref{symdef}) restricts to a functor
\[
\FreeComm_{\boxtimes_{\oplus}}\colon\Dulf(\MMHM(\Msp^{\zeta\sst}_{\mu}\setminus \Msp^{\zeta\sst}_0))\rightarrow \Dulf(\MMHM(\Msp^{\zeta\sst}_{\mu})).
\]
\end{lemma}
\begin{proposition}
\label{symmPure}
The functor $\FreeComm_{\boxtimes_{\oplus}}$ takes pure objects to pure objects.
\end{proposition}
\begin{proof}
Let $\mathcal{F}\in\MMHM(\Msp_{\mu}^{\zeta\sst})$ be pure.  Let 
\[
\oplus^n\colon \Msp_{\mu}^{\zeta\sst}\times\ldots\times \Msp_{\mu}^{\zeta\sst}\rightarrow \Msp_{\mu}^{\zeta\sst}
\]
be the n-fold monoid map.  Then $\oplus^n$ is finite by \cite[Lem.2.1]{Meinhardt14}, and so 
\[
\mathcal{G}=\oplus^n_*\left(\mathcal{F}\boxtimes\ldots \boxtimes\mathcal{F}\right)\in\MMHM(\Msp_{\mu}^{\zeta\sst})
\]
is pure.  On the other hand, $\FreeComm_{\boxtimes_{\oplus}}^n\mathcal{F}\subset \mathcal{G}$ is the isotrivial direct summand under the $\SSym_n$-action, and is therefore pure.  We deduce that $\FreeComm_{\boxtimes_{\oplus}}\mathcal{F}$ is a direct sum of pure objects, and is therefore pure.
\end{proof}

We define a new monoidal structure on $\Dulf(\MMHM(\Msp^{\zeta\sst}))$ by setting 
\begin{equation}
\mathcal{F}\boxtimes_{\oplus}^{\tw} \mathcal{G}\colonequals \bigoplus_{\dd',\dd''\in\mathbb{N}^{Q_0}} \LL^{\langle \dd'',\dd'\rangle/2}\otimes\left(\oplus_*(\mathcal{F}_{\dd'}\boxtimes\mathcal{G}_{\dd''})\right), 
\end{equation}
where $\mathcal{F}=\bigoplus_{\dd'\in\mathbb{N}^{Q_0}}\mathcal{F}_{\dd'}$ for $\mathcal{F}_{\dd'}\in\Dub(\MMHM(\Msp_{\dd'}^{\zeta\sst}))$ and  $\mathcal{G}=\bigoplus_{\dd''\in\mathbb{N}^{Q_0}}\mathcal{G}_{\dd''}$ for $\mathcal{G}_{\dd''}\in\Dub(\MMHM(\Msp_{\dd''}^{\zeta\sst}))$.  Since we are twisting by a pure twist, this monoidal product again preserves weights.  If $\zeta$ is $\mu$-generic, the restriction of $\boxtimes_{\oplus}^{\tw}$ to $\Du(\MMHM(\Msp^{\zeta\sst}_{\mu}))$ is naturally isomorphic to the untwisted symmetric monoidal product $\boxtimes_{\oplus}$, but in general there is no symmetrizing natural isomorphism of bifunctors making $\boxtimes_{\oplus}^{\tw}$ into a symmetric monoidal product.  

The associator natural isomorphism for the monoidal structure $\boxtimes_{\oplus}$ is the one induced from the usual monoidal structure on the derived categories of constructible sheaves and $\mathcal{D}$-modules.  We have to take some care in defining the signs for the associator for $\boxtimes_{\oplus}^{\tw}$.  Let $\mathcal{F}^{(h)}\in\Dulf(\MMHM(\Msp^{\zeta\sst}_{\dd^{(h)}}))$ for $h=1,2,3$.  Consider the composition of natural transformations
\begin{align*}
\eta_{\mathcal{F}^{(1)},\mathcal{F}^{(2)},\mathcal{F}^{(3)}}\colon &\LL^{\langle \dd^{(3)}+\dd^{(2)},\dd^{(1)}\rangle/2}\otimes\left(\mathcal{F}^{(1)}\boxtimes_{\oplus} \left(\LL^{\langle \dd^{(3)},\dd^{(2)}\rangle/2}\otimes(\mathcal{F}^{(2)}\boxtimes_{\oplus}  \mathcal{F}^{(3)})\right)\right)
\\
\xrightarrow{\cong}&\LL^{\langle \dd^{(3)}+\dd^{(2)},\dd^{(1)}\rangle/2+\langle \dd^{(3)},\dd^{(2)}\rangle}\otimes\left(\mathcal{F}^{(1)}\boxtimes_{\oplus} \left(\mathcal{F}^{(2)}\boxtimes_{\oplus}  \mathcal{F}^{(3)}\right)\right)
\\
\xrightarrow{\cong}&\LL^{\langle \dd^{(3)}+\dd^{(2)},\dd^{(1)}\rangle/2+\langle \dd^{(3)},\dd^{(2)}\rangle/2}\otimes\left(\left(\mathcal{F}^{(1)}\boxtimes_{\oplus} \mathcal{F}^{(2)}\right)\boxtimes_{\oplus}  \mathcal{F}^{(3)}\right)
\\
\xrightarrow{\cong}&\LL^{\langle \dd^{(3)},\dd^{(2)}+\dd^{(1)}\rangle/2}\otimes \left(\LL^{\langle \dd^{(2)},\dd^{(1)}\rangle/2}\otimes\left(\mathcal{F}^{(1)}\boxtimes_{\oplus} \mathcal{F}^{(2)}\right)\boxtimes_{\oplus}  \mathcal{F}^{(3)}\right)
\end{align*}
The resulting natural isomorphism $\mathcal{F}^{(1)}\boxtimes_{\oplus}^{\tw}(\mathcal{F}^{(2)}\boxtimes_{\oplus}^{\tw}\mathcal{F}^{(3)})\rightarrow (\mathcal{F}^{(1)}\boxtimes_{\oplus}^{\tw}\mathcal{F}^{(2)})\boxtimes_{\oplus}^{\tw}\mathcal{F}^{(3)})$ does \textit{not} satisfy the pentagon identity.  We define instead
\begin{equation}
\label{AssocSign}
\eta_{\mathcal{F}^{(1)},\mathcal{F}^{(2)},\mathcal{F}^{(3)}}'\colonequals (-1)^{\langle \dd^{(2)},\dd^{(3)}\rangle \chi(\dd^{(1)},\dd^{(1)})}\eta_{\mathcal{F}^{(1)},\mathcal{F}^{(2)},\mathcal{F}^{(3)}}
\end{equation}
to produce a valid associator natural isomorphism.

In a little more detail, the diagram
\begin{align}
\label{pentagon}
\xymatrix{
& ((\mathcal{F}^{(1)}\boxtimes_{\oplus}^{\tw} \mathcal{F}^{(2)}) \boxtimes_{\oplus}^{\tw}\mathcal{F}^{(3)}) \boxtimes_{\oplus}^{\tw}\mathcal{F}^{(4)}\ar[dl]^{\alpha}\ar[d]_{\delta}\\
(\mathcal{F}^{(1)} \boxtimes_{\oplus}^{\tw}(\mathcal{F}^{(2)} \boxtimes_{\oplus}^{\tw}\mathcal{F}^{(3)})) \boxtimes_{\oplus}^{\tw}\mathcal{F}^{(4)}\ar[d]^{\beta}&(\mathcal{F}^{(1)}\boxtimes_{\oplus}^{\tw} \mathcal{F}^{(2)}) \boxtimes_{\oplus}^{\tw}(\mathcal{F}^{(3)}\boxtimes_{\oplus}^{\tw} \mathcal{F}^{(4)})\ar[d]^{\epsilon}\\
\mathcal{F}^{(1)}\boxtimes_{\oplus}^{\tw} ((\mathcal{F}^{(2)}\boxtimes_{\oplus}^{\tw} \mathcal{F}^{(3)})\boxtimes_{\oplus}^{\tw} \mathcal{F}^{(4)})\ar[r]^{\gamma}&\mathcal{F}^{(1)} \boxtimes_{\oplus}^{\tw}(\mathcal{F}^{(2)}\boxtimes_{\oplus}^{\tw} (\mathcal{F}^{(3)}\boxtimes_{\oplus}^{\tw} \mathcal{F}^{(4)}))
}
\end{align}
fails to commute if the morphisms are all defined using $\eta_{\bullet,\bullet,\bullet}$.  The trouble is in the definition of $\delta$: in swapping $\LL^{\langle \dd^{(4)},\dd^{(3)}\rangle/2}$ with $(\mathcal{F}^{(1)}\boxtimes_{\oplus}^{\tw} \mathcal{F}^{(2)})$ we pick up an extra sign $(-1)^{\langle \dd^{(4)},\dd^{(3)}\rangle\langle \dd^{(2)},\dd^{(1)}\rangle}$ because of the half Tate twist in the definition of $(\mathcal{F}^{(1)}\boxtimes_{\oplus}^{\tw} \mathcal{F}^{(2)})$.  Modifying all of the morphisms in (\ref{pentagon}) via the rule (\ref{AssocSign}) introduces an overall parity change of
\begin{align*}
&\langle \dd^{(2)},\dd^{(3)}\rangle\chi(\dd^{(1)},\dd^{(1)})+\langle \dd^{(2)}+\dd^{(3)},\dd^{(4)}\rangle\chi(\dd^{(1)},\dd^{(1)})
\\
+&\langle \dd^{(3)},\dd^{(4)}\rangle\chi(\dd^{(2)},\dd^{(2)})+\langle \dd^{(3)},\dd^{(4)}\rangle\chi(\dd^{(1)}+\dd^{(2)},\dd^{(1)}+\dd^{(2)})\\
+&\langle \dd^{(2)},\dd^{(3)}+\dd^{(4)}\rangle\chi(\dd^{(1)},\dd^{(1)})=\langle \dd^{(4)},\dd^{(3)}\rangle \langle \dd^{(2)},\dd^{(1)}\rangle &(\textrm{modulo }2)
\end{align*}
as required.  The terms in the sum are the contributions from $\alpha,\beta,\gamma,\delta,\epsilon$ respectively.  We extend $\eta'_{\bullet,\bullet,\bullet}$ via linearity to triples of objects in $\Dulf(\MMHM(\Msp^{\zeta\sst}))$ to define the associator natural isomorphism for the monoidal structure $\boxtimes_{\oplus}^{\tw}$.

We define the monoidal product on $\Dulf(\MMHM(\mathbb{N}^{Q_0}))$ by setting 
\begin{equation}
\label{AssociatorSign}
\mathcal{F}\boxtimes_{+}^{\tw} \mathcal{G}\colonequals \bigoplus_{\dd',\dd''\in\mathbb{N}^{Q_0}}\LL^{\langle \dd'',\dd'\rangle/2}\otimes\left(+_*(\mathcal{F}_{\dd'}\boxtimes\mathcal{G}_{\dd''})\right)
\end{equation}
and constructing the associator natural isomorphism as in (\ref{AssociatorSign}).  The map 
\[
\dim_*\colon \Dulf(\MMHM(\mathcal{M}^{\zeta\sst}))\rightarrow \Dulf(\MMHM(\mathbb{N}^{Q_0}))
\]
is a monoidal functor, and the restriction
\[
\dim_{\mu,*}\colon \Dulf(\MMHM(\mathcal{M}_{\mu}^{\zeta\sst}))\rightarrow \Dulf(\MMHM(\Lambda^{\zeta}_{\mu}))
\]
is a symmetric monoidal functor for $\zeta$-generic $\mu$.
\begin{remark}\label{twDual}
By skew symmetry of $\langle \bullet,\bullet\rangle_Q$ and finiteness of the maps of schemes $\oplus$ and $+$, there are natural isomorphisms of bifunctors
\begin{align*}
&\DD^{\mon}(\mathcal{F}\boxtimes_{\oplus}^{\tw}\mathcal{G})\cong \DD^{\mon}\mathcal{G}\boxtimes_{\oplus}^{\tw}\DD^{\mon}\mathcal{F}\\
&\DD^{\mon}(\mathcal{F}\boxtimes_{+}^{\tw}\mathcal{G})\cong \DD^{\mon}\mathcal{G}\boxtimes_{+}^{\tw}\DD^{\mon}\mathcal{F}.
\end{align*}
Note the swap of arguments.
\end{remark}

\subsection{Framed moduli spaces}
\label{framedSec}
Framed moduli spaces will play a central role in what follows.  Their cohomology provides an approximation to the cohomology of $\Mst^{\zeta\sst}$, in a way which we will make precise in Section \ref{cohaDT}.  

Let $\ff\in\mathbb{N}^{Q_0}$ be a dimension vector (called the framing vector).  We form a new quiver $Q_{\ff}$ by setting $Q_{\ff}=(Q_0\sqcup\{\infty\}, Q_1\sqcup \{\beta_{i,l_i}\colon \infty \to i \mid i\in Q_0, \hbox{ }1\le l_i\le \ff_i \})$.  Given a stability condition $\zeta$ for $Q$, and a slope $\mu\in(-\infty,\infty)$ we extend $\zeta$ to a stability condition $\zeta^{(\mu)}_{\ff}$ for $Q_{\ff}$ by picking $\zeta^{(\mu)}_{\ff,\infty}\in\mathbb{H}_+$ so that 
\[
-\Re e(\zeta^{(\mu)}_{\ff,\infty})/\Im m(\zeta^{(\mu)}_{\ff,\infty})=\mu+\epsilon
\]
for sufficiently small $\epsilon>0$, and picking $|\zeta_{\ff,\infty}^{(\mu)}|\gg 0$.

 A $\mathbb{C} Q_{\ff}$-module $\rho$ with $\dim(\rho)_{\infty}=1$ and $\dim(\rho|_{\mathbb{C}Q})\in \Lambda^{\zeta}_{\mu}$ is $\zeta^{(\mu)}_{\ff}$-semistable if and only if the underlying $\mathbb{C} Q$-module is $\zeta$-semistable, and for every $\mathbb{C}Q_{\ff}$-submodule $\rho'\subset \rho$ such that $\dim(\rho')_{\infty}=1$, the underlying $\mathbb{C} Q$-module of $\rho'$ has slope strictly less than $\mu$.  A $\zeta^{(\mu)}_{\ff}$-semistable $\mathbb{C}Q_{\ff}$-module is automatically $\zeta^{(\mu)}_{\ff}$-stable, and we write $\Msp^{\zeta}_{\ff,\dd}$ for the fine moduli space of $\zeta^{(\mu)}_{\ff}$-semistable $\mathbb{C} Q_{\ff}$-modules of dimension $(1,\dd)\in\mathbb{N}\times\Lambda_{\mu}^{\zeta}$.  The moduli space $\Msp_{\ff,\dd}^{\zeta}$ is smooth. 

\begin{remark}
In fact $G_{\dd}$ acts freely on the smooth variety 
\begin{equation}
\label{Ydef}
Y_{\ff,\dd}^{\zeta}\colonequals X(Q_{\ff})^{\zeta_{\ff}^{(\mu)}\sst}_{(1,\dd)},
\end{equation}
and $\Msp_{\ff,\dd}^{\zeta}$ is the quotient.
\end{remark}

We denote by 
\[
\pi_{\ff,\dd}^{\zeta}\colon\Msp^{\zeta}_{\ff,\dd}\rightarrow \Msp^{\zeta\sst}_{\dd}
\]
the map given by forgetting the framing.  It is a proper map, since the other two maps in the diagram
\[
\xymatrix{
\Msp^{\zeta}_{\ff,\dd}\ar@/^1pc/[rr]\ar[r]_-{\pi^\zeta_{\ff,\dd}}&\Msp_{\dd}^{\zeta\sst}\ar[r]_-{q^\zeta_{\dd}}&\Msp_{\dd}
}
\]
are, as they are GIT quotient maps.
\smallbreak
For $\dd\in\Lambda_{\mu}^{\zeta}$ we may alternatively extend $\dd$ to a dimension vector for $Q_{\ff}$ by setting $\dd_{\infty}=0$.  There is an obvious isomorphism $X(Q_{\ff})_{(0,\dd)}^{\zeta_{\ff}^{(\mu)}\sst}\cong X(Q)_{\dd}^{\zeta\sst}$.
\smallbreak

\subsection{Jacobi algebras and potentials}
Let $W\in\mathbb{C} Q/[\mathbb{C}Q,\mathbb{C}Q]$ be a finite linear combination of equivalence classes of cyclic words in $\mathbb{C} Q$.  Such a $W$ is called a \textit{potential}.  A potential $W$ induces a function $\WWW$ on $\Mst$, defined as follows.  Firstly, assume that $W$ lifts to a single cyclic word $c=a_r\ldots a_0$ in $\mathbb{C} Q$.  Then a $\mathbb{C}Q$-module $F$ determines an endomorphism 
\[
F(a_r)\circ\ldots\circ F(a_0)\colon F(s(a_0))\rightarrow F(s(a_0)), 
\]
and $\Tr(F(a_r)\circ\ldots\circ F(a_0))$ determines a function on $X(Q)_{\dd}$, for each $\dd\in\mathbb{N}^{Q_0}$.  By cyclic invariance of the trace, this function is $G_{\dd}$-invariant, and does not depend on the lift $c$ of $W$.  Extending by linearity, we define for general $W\in\mathbb{C}Q/[\mathbb{C}Q,\mathbb{C}Q]$ the induced function 
\[
\WWW_{\dd}\colon \Mst_{\dd}\rightarrow \mathbb{C}.  
\]

We denote by $\WWW_{\dd}^{\zeta}$ the restriction of this function to $\Mst^{\zeta\sst}_{\dd}$, and by 
\[
\WW^{\zeta}_{\dd}\colon \Msp^{\zeta\sst}_{\dd}\rightarrow\mathbb{C}
\]
the unique function through which $\WWW^{\zeta}_{\dd}$ factors.  Similarly, we define 
\[
\WW_{\ff,\dd}^{\zeta}\colonequals \WW_{\dd}^{\zeta}\circ\pi_{\ff,\dd}^{\zeta}\colon \Msp_{\ff,\dd}^{\zeta}\rightarrow \mathbb{C},
\]
and we define $\WWW_{\dd',\dd''}^{\zeta}$ to be the composition 
\[
\Mst_{\dd',\dd''}^{\zeta\sst}\hookrightarrow\Mst_{\dd'+\dd''}^{\zeta\sst}\xrightarrow{\WWW_{\dd'+\dd''}^{\zeta}}\mathbb{C}.
\]

Associated to the data $(Q,W)$ is the Jacobi algebra 
\[
\Jac(Q,W)\colonequals \mathbb{C} Q/\langle \partial W/\partial a\;|\;a\in Q_1\rangle.  
\]
Here the \textit{noncommutative derivatives} $\partial W/\partial a$ are defined as follows.  First assume that $W$ lifts to a single cyclic word $c\in\mathbb{C}Q$.  Then 
\[
\partial W/\partial a\colonequals \sum_{c=c'ac''}c''c'.
\]
We then extend the definition to general $W$ by linearity.  We define $\Mst_{W}$, the stack of finite-dimensional $\Jac(Q,W)$-modules, in the same way as the stack of finite-dimensional $\mathbb{C} Q$-modules.  In particular there is a natural closed embedding of stacks $\Mst_W\subset \Mst$, and it is easy to show that $\Mst_W=\crit(\WWW)$ as substacks of $\Mst$.  

In order to keep to Assumption \ref{phicheat} we will assume that $W\in \langle \partial W/\partial a|a\in Q_1\rangle$, for then $\crit(\WWW)\subset\WWW^{-1}(0)$.  One very common set of circumstances in which this requirement is met is when there is a grading of the arrows $Q_1$ with integers such that $W$ is homogeneous of nonzero weight.  As mentioned after Assumption \ref{phicheat}, we can drop this requirement at the expense of slightly more complicated definitions.

\begin{proposition}\label{TSeq}
Let $W\in\mathbb{C}Q/[\mathbb{C}Q,\mathbb{C}Q]$ be a potential.  Then 
\[
\phim{\WW_{\mu}^{\zeta}}\colon\Dulf\left(\MHM(\Msp^{\zeta\sst}_{\mu})\right)\rightarrow\Dulf\left(\MMHM(\Msp^{\zeta\sst}_{\mu})\right)
\]
is a symmetric monoidal functor.
\end{proposition}
\begin{proof}
Let $\mathcal{F},\mathcal{G}\in\MHM(\Msp^{\zeta\sst}_{\mu})$.  Let $\pi\colon \Msp^{\zeta\sst}_{\mu}\times\Msp^{\zeta\sst}_{\mu}\rightarrow\Msp^{\zeta\sst}_{\mu}\times\Msp^{\zeta\sst}_{\mu}$ be the map swapping arguments.  The statement follows from exactness of $\phim{\WW_{\mu}^{\zeta}}$ and biexactness of $\boxtimes_{\oplus}$ and the claim that the following diagram is commutative,
\[
\xymatrix{
\phim{\WW_{\mu}^{\zeta}}\left(\mathcal{F}\boxtimes_{\oplus}\mathcal{G}\right)\ar[r]\ar[d]&\phim{\WW_{\mu}^{\zeta}}\left(\mathcal{G}\boxtimes_{\oplus}\mathcal{F}\right)\ar[d]\\
\phim{\WW_{\mu}^{\zeta}}\mathcal{F}\boxtimes_{\oplus}\phim{\WW_{\mu}^{\zeta}}\mathcal{G}\ar[r]&\phim{\WW_{\mu}^{\zeta}}\mathcal{G}\boxtimes_{\oplus}\phim{\WW_{\mu}^{\zeta}}\mathcal{F},
}
\]
where the horizontal maps are induced by the symmetrizing natural isomorphisms of Section \ref{prodss}, and the vertical maps are given by the Thom-Sebastiani natural isomorphism.  By faithfulness of the functor $\form{\Msp_{\mu}^{\zeta\sst}}$ defined in (\ref{formDef}), it is enough to prove commutativity of the diagram at the level of perverse sheaves, or constructible complexes.  Since $\oplus$ is finite, by natural commutativity of vanishing cycle functors with proper maps, it is enough to prove commutativity of the diagram
\begin{equation}
\label{liftSymTS}
\xymatrix{
\phi_{\WW_{\mu}^{\zeta}}\left(\mathcal{F}\boxtimes\mathcal{G}\right)[-1]\ar[r]\ar[d]&\phi_{\WW_{\mu}^{\zeta}}\left(\mathcal{G}\boxtimes\mathcal{F}\right)[-1]\ar[d]\\
\phi_{\WW_{\mu}^{\zeta}}\mathcal{F}[-1]\boxtimes\phi_{\WW_{\mu}^{\zeta}}\mathcal{G}[-1]\ar[r]&\phi_{\WW_{\mu}^{\zeta}}\mathcal{G}[-1]\boxtimes\phi_{\WW_{\mu}^{\zeta}}\mathcal{F}[-1].
}
\end{equation}
Let $\iota\colon\Msp_{\mu}^{\zeta}\rightarrow X$ be an embedding inside a smooth scheme, let $f$ be a function on $X$ extending $\WW_{\mu}^{\zeta}$ and consider $\mathcal{F}$ as a perverse sheaf on $X$ via the direct image.  Then we define the functor on the category of constructible sheaves on $X$
\[
\Gamma_{f^{-1}(\mathbb{R}_{\leq 0})}\mathcal{F}(U)\colonequals \ker\left(\mathcal{F}(U)\rightarrow \mathcal{F}(U\setminus f^{-1}(\mathbb{R}_{\leq 0}))\right)
\]
and we may alternatively define $\phi_f\mathcal{F}[-1]\colonequals (R\Gamma_{f^{-1}(\mathbb{R}_{\leq 0})}\mathcal{F})|_{f^{-1}(0)}$.  By \cite[Prop.A.2]{Br12} \cite{Ma01} the Thom--Sebastiani isomorphism is given by the natural morphism
\[
\alpha\colon R\Gamma_{f^{-1}(\mathbb{R}_{\leq 0})\times f^{-1}(\mathbb{R}_{\leq 0})}\left(\bullet \boxtimes\bullet\right)\rightarrow R\Gamma_{(f\boxplus f)^{-1}(\mathbb{R}_{\leq 0})}\left(\bullet\boxtimes\bullet\right)
\]
induced by the inclusion 
\[
f^{-1}(\mathbb{R}_{\leq 0})\times f^{-1}(\mathbb{R}_{\leq 0})\subset (f\boxplus f)^{-1}(\mathbb{R}_{\leq 0}),
\]
and the horizontal maps in (\ref{liftSymTS}) are induced by the natural morphism
\[
\mathcal{F}\boxtimes \mathcal{G}\xrightarrow{\pi^\#}\pi_*\left(\mathcal{G}\boxtimes\mathcal{F}\right)
\]
on both sides.  Applying the natural isomorphism $\alpha$ to this morphism, we obtain the diagram (\ref{liftSymTS}).
\end{proof}

\section{Cohomological Donaldson--Thomas invariants}
\label{CoDTSec}
\subsection{Approximation of $p^{\zeta}_{\dd}$ by proper maps}
\label{cohaDT}
Recall from (\ref{pdef}) (and the remark following it) the morphism
\[
p^{\zeta}_{\dd}\colon \Mst^{\zeta\sst}_{\dd}\rightarrow\Msp^{\zeta\sst}_{\dd}
\]
taking a $\zeta$-semistable $\dd$-dimensional $\mathbb{C}Q$-representation to the associated polystable representation.  We start this section by showing that there is a module-theoretic compactification of Totaro's construction relative to $\Msp^{\zeta\sst}_{\dd}$, which we may use as in Section \ref{TotCon} to define the direct image along $p^{\zeta}_{\dd}$ of the monodromic mixed Hodge module of vanishing cycles from the stack of $\dd$-dimensional $\zeta$-semistable representations of $\mathbb{C}Q$.  

Recall the definition of $\zeta_{\ff}^{(\mu)}$ from Section \ref{framedSec}.  Recall from (\ref{Ydef}) the notation
\begin{align*}
X_{\ff,\dd}\colonequals &X(Q_{\ff})_{(1,\dd)}\\
Y_{\ff,\dd}^{\zeta}\colonequals &X_{\ff,\dd}^{\zeta_{\ff}^{(\mu)}\sst}.
\end{align*}
We replace $\zeta_{\ff}^{(\mu)}$ by an equivalent King stability condition for $Q_{\ff}$, as we always can, by Remark \ref{BtoK}.  Then $\zeta_{\ff}^{(\mu)}$ defines a linearization of the natural $G_{(1,\dd)}$-action on $X_{\ff,\dd}$ and we have 
\begin{align*}
\Msp_{\ff,\dd}^{\zeta}\colonequals &X_{\ff,\dd}/\!\!/_{\chi}G_{(1,\dd)}
\\\cong &Y_{\ff,\dd}^{\zeta}/G_{\dd}
\end{align*}
where $G_{(1,\dd)}=\Gl_1(\mathbb{C})\times G_{\dd}$.  As a $G_{\dd}$-equivariant variety, $X_{\ff,\dd}$ admits a product decomposition $X_{\ff,\dd}=X_{\dd}\times V_{\ff,\dd}$, where $V_{\ff,\dd}\colonequals \bigoplus_{i\in Q_0}\Hom(\mathbb{C}^{\ff_i},\mathbb{C}^{\dd_i})$, and the extra $\Gl_1(\mathbb{C})$ factor acts by rescaling $V_{\ff,\dd}$.  We define $U_{\ff,\dd}\subset V_{\ff,\dd}$ to be the subspace of $Q_0$-tuples of surjective maps.  The group $G_{\dd}$ acts freely on $U_{\ff,\dd}$, and the quotient is a product of Grassmannians, and is a fibre bundle quotient in the category of schemes.  Therefore $G_{\dd}$ also acts freely on $X_{\dd}^{\zeta\sst}\times U_{\ff,\dd}$, which is a fibre bundle in the category of schemes by \cite[Prop.23]{EdGr98}.  Consider the commutative diagram
\[
\xymatrix{
(X_{\dd}^{\zeta\sst}\times U_{\ff,\dd})/G_{\dd}\ar@/^1.5pc/[rr]^-{i}\ar[drr]_{\kappa_{\ff,\dd}^{\zeta}}\ar@{^(->}[r]^-h& Y_{\ff,\dd}^{\zeta}/G_{\dd}\ar@{^(->}[r]\ar[rd]^{\pi ^{\zeta}_{\ff,\dd}}& (X_{\dd}^{\zeta\sst}\times V_{\ff,\dd})/G_{\dd}\ar[d]
\\&& \Msp_{\dd}^{\zeta\sst}.
}    
\]
All of the objects in the above diagram are algebraic varieties, with the exception of $(X_{\dd}^{\zeta\sst}\times V_{\ff,\dd})/G_{\dd}$, which is strictly an Artin stack.  The diagram obtained by deleting this stack and all arrows incident to it is a relative compactification of the map $\kappa_{\ff,\dd}^{\zeta}$.

In what follows we use the notation $\ff\gg 0$ to mean that $\ff_i\gg 0$ for every $i\in Q_0$.  
\begin{lemma}
\label{appLem}
For fixed $n$ and $\ff\gg 0$ the natural map
\[
\Ho^n\left(\pi ^{\zeta}_{\ff,\dd,*}\phim{\WW^{\zeta}_{\ff,\dd}}\mathbb{Q}_{Y_{\ff,\dd}^{\zeta}/G_{\dd}}\right)\rightarrow\Ho^n\left(\kappa^{\zeta}_{\ff,\dd,*}\phim{\WW^{\zeta}_{\ff,\dd}}\mathbb{Q}_{(X_{\dd}^{\zeta\sst}\times U_{\ff,\dd})/G_{\dd}}\right) 
\]
is an isomorphism.
\end{lemma}
\begin{proof}
As in the proof of Proposition \ref{stabProp} it is enough to show that for fixed $n\in\mathbb{Z}$, $\Ho_{\con}^n\left(\phi_{\overline{\WW}_{\ff,\dd}^{\zeta}}\mathbb{Q}_{Y_{\ff,\dd}^{\zeta}}\rightarrow\overline{h}_*\overline{h}^*\phi_{\overline{\WW}_{\ff,\dd}^{\zeta}}\mathbb{Q}_{Y_{\ff,\dd}^{\zeta}}\right)$ is an isomorphism for $\ff\gg 0$, where
\[
\overline{h}\colon X_{\dd}^{\zeta\sst}\times U_{\ff,\dd} \rightarrow Y_{\ff,\dd}^{\zeta}
\]
is the open inclusion and 
\[
\overline{\WW}_{\ff,\dd}^{\zeta}\colon Y_{\ff,\dd}^{\zeta}\rightarrow \mathbb{C}
\]
is the composition of $\WW_{\ff,\dd}^{\zeta}$ with the quotient map.  On the other hand, the complex of sheaves $\phi_{\overline{\WW}_{\ff,\dd}^{\zeta}}\mathbb{Q}_{Y_{\ff,\dd}^{\zeta}}$ is given by restriction of the pullback $\rho^*\phi_{\overline{\WW}_{\dd}}\mathbb{Q}_{X^{\zeta\sst}_{\dd}}$ of $\phi_{\overline{\WW}_{\dd}}\mathbb{Q}_{X^{\zeta\sst}_{\dd}}$ along the projection $\rho\colon X^{\zeta\sst}_{\dd}\times V_{\ff,\dd}\rightarrow X^{\zeta\sst}_{\dd}$.  It follows that for each $x\in X_{\dd}^{\zeta\sst}$, the complex of constructible sheaves $\phi_{\overline{\WW}_{\ff,\dd}^{\zeta}}\mathbb{Q}_{Y_{\ff,\dd}^{\zeta}}|_{Y_{\ff,\dd}^{\zeta}\cap\rho^{-1}(x)}$ is a constant complex, supported in degrees that are bounded independently of $x$.  The codimension of the complement of $(X_{\dd}^{\zeta\sst}\times U_{\ff,\dd})\cap\rho^{-1}(x)$ inside $Y_{\ff,\dd}^{\zeta}\cap{\rho^{-1}(x)}$ tends to infinity as $\ff\mapsto\infty$, and so for sufficiently large $\ff$
\[
\Ho_{\con}^n\left(\phi_{\overline{\WW}_{\ff,\dd}^{\zeta}}\mathbb{Q}_{Y_{\ff,\dd}^{\zeta}}\right)_{Y_{\ff,\dd}^{\zeta}\cap\rho^{-1}(x)}\rightarrow\Ho_{\con}^n\left(\overline{h}_*\overline{h}^*\phi_{\overline{\WW}_{\ff,\dd}^{\zeta}}\mathbb{Q}_{Y_{\ff,\dd}^{\zeta}}\right)_{Y_{\ff,\dd}^{\zeta}\cap\rho^{-1}(x)}
\]
is an isomorphism for all $x\in X^{\zeta\sst}_{\dd}$, and the result follows.
\end{proof}
We deduce that for fixed $\dd\in\mathbb{N}^{Q_0}$ and $n\in\mathbb{Z}$, and for $\ff\gg 0$ there are isomorphisms
\begin{equation}
\label{PhiSurj}
\Phi_{\ff,\dd,W}\colon \Ho^n\left(\pi _{\ff,\dd,*}^{\zeta}\phim{\WW_{\ff,\dd}^{\zeta}}\mathbb{Q}_{\Msp_{\ff,\dd}^{\zeta}}\right)\rightarrow \Ho^n\left(\kappa^{\zeta}_{\ff,\dd,*}\phim{\WW_{\ff,\dd}^{\zeta}}\mathbb{Q}_{(X_{\dd}^{\zeta\sst}\times U_{\ff,\dd })/G_{\dd}}\right)
\end{equation}
and 
\begin{equation}
\label{PsiSurj}
\Psi_{\ff,\dd ,W}\colon  \Ho^n\left(\kappa^{\zeta}_{\ff,\dd ,!}\phim{\WW_{\ff,\dd }^{\zeta}}\mathbb{L}^{-\ff\cdot \dd}_{(X_{\dd}^{\zeta\sst}\times U_{\ff,\dd })/G_{\dd}}\right)\rightarrow \Ho^n\left(\pi_{\ff,\dd ,!}^{\zeta}\phim{\WW_{\ff,\dd }^{\zeta}}\mathbb{L}^{-\ff\cdot\dd}_{\Msp_{\ff,\dd }^{\zeta}}\right)
\end{equation}
where the argument that $\Psi_{\ff,\dd ,W}$ is an isomorphism is as in Lemma \ref{appLem}.  On the other hand, for $\ff\gg 0$, the right hand side of (\ref{PhiSurj}) is by definition $\Ho^n(p_{\dd,*}^{\zeta}\phim{\WWW_{\dd}^{\zeta}}\mathbb{Q}_{\Mst^{\zeta\sst}_{\dd}})$, while the left hand side of (\ref{PsiSurj}) is $\Ho^n(p_{\dd,!}^{\zeta}\phim{\WWW_{\dd}^{\zeta}}\mathbb{Q}_{\Mst^{\zeta\sst}_{\dd}})$.  
\begin{remark}
Put in words, we can say that the direct image of the vanishing cycle monodromic mixed Hodge module along the non-representable map $p^{\zeta}_{\dd}\colon\Mst_{\dd}^{\zeta\sst}\rightarrow\Msp_{\dd}^{\zeta\sst}$ is approximated by the direct image along \textit{proper} maps of schemes $\pi _{\ff,\dd }^{\zeta}\colon \Msp_{\ff,\dd }^{\zeta}\rightarrow\Msp_{\dd}^{\zeta\sst}$.  
\end{remark}
For $\ff\gg 0$ we obtain isomorphisms
\begin{equation}
\label{PhiBar}
\overline{\Phi}^{\SP}_{\ff,\dd,W }\colon \HO^n(\Msp_{\ff,\dd }^{\zeta,\SP},\phim{\WW_{\ff,\dd }^{\zeta}}\mathbb{Q}_{\Msp_{\ff,\dd }^{\zeta}})\cong\HO^n(\Mst^{\zeta\sst,\SP}_{\dd},\phim{\WWW^{\zeta}_{\dd}}\mathbb{Q}_{\Mst^{\zeta\sst}_{\dd}})
\end{equation}
and
\begin{equation}
\label{PsiBar}
\overline{\Psi}^{\SP}_{\ff,\dd,W }\colon \HO^n_c(\Mst^{\zeta\sst,\SP}_{\dd},\phim{\WWW^{\zeta}_{\dd}}\mathbb{Q}_{\Mst^{\zeta\sst}_{\dd}})\cong \HO^n_c(\Msp_{\ff,\dd }^{\zeta,\SP},\phim{\WW^{\zeta}_{\ff,\dd }}\mathbb{L}^{-\ff\cdot \dd}_{\Msp_{\ff,\dd }^{\zeta}})
\end{equation}
in the same way.  

Since each map $\pi ^{\zeta}_{\ff,\dd }$ is proper, and $\phim{\WW^{\zeta}_{\dd}}$ is exact, we deduce the following proposition.
\begin{proposition}
\label{cvs}
For every $\dd\in\mathbb{N}^{Q_0}$ there are natural isomorphisms
\begin{equation}
\label{nudef}
\nu_{\dd}\colon \Ho(p^{\zeta}_{\dd,*}\phim{\WWW^{\zeta}_{\dd}}\mathbb{Q}_{\Mst_{\dd}^{\zeta\sst}})\cong \phim{\WW^{\zeta}_{\dd}}\Ho(p^{\zeta}_{\dd,*}\mathbb{Q}_{\Mst_{\dd}^{\zeta\sst}}).
\end{equation}
\end{proposition}
\begin{proof}
The isomorphisms $\nu_{\dd}$ are obtained by considering the left hand side of (\ref{PhiSurj}) and using the natural isomorphisms 
\[
 \phim{\WW^{\zeta}_{\dd}}\pi ^{\zeta}_{\ff,\dd ,*}\mathbb{Q}_{\Msp_{\ff,\dd }^{\zeta}}\cong \pi^{\zeta}_{\ff,\dd ,*}\phim{\WW^{\zeta}_{\ff,\dd }}\mathbb{Q}_{\Msp_{\ff,\dd }^{\zeta}}.
\]
The isomorphisms $\nu_{\dd}$ are well-defined by commutativity of the square
\[
\xymatrix{
\pi ^{\zeta}_{\ff',\dd,*}\phim{\WW^{\zeta}_{\ff',\dd}}\mathbb{Q}_{\Msp_{\ff',\dd}^{\zeta}}\ar[r]&\pi ^{\zeta}_{\ff',\dd,*}\iota_{\ff,\ff',*}\phim{\WW^{\zeta}_{\ff,\dd }}\mathbb{Q}_{\Msp_{\ff,\dd }^{\zeta}}\\
\phim{\WW^{\zeta}_{\dd}}\pi ^{\zeta}_{\ff',\dd,*}\mathbb{Q}_{\Msp_{\ff',\dd}^{\zeta}}\ar[u]\ar[r]& \phim{\WW^{\zeta}_{\dd}}\pi ^{\zeta}_{\ff',\dd,*}\iota_{\ff,\ff',*}\mathbb{Q}_{\Msp_{\ff,\dd }^{\zeta}}\ar[u]
}
\]
where $\ff'>\ff$ and $\iota_{\ff,\ff'}\colon\Msp^{\zeta}_{\ff,\dd }\rightarrow \Msp^{\zeta}_{\ff',\dd}$ is the natural inclusion obtained by extending a framing by zero.  The square commutes as it is obtained by applying the natural transformation 
\[
\phim{\WW_{\dd}^{\zeta}}\pi ^{\zeta}_{\ff',\dd,*}\rightarrow \pi ^{\zeta}_{\ff',\dd,*}\phim{\WW^{\zeta}_{\ff',\dd}}
\]
to the restriction map
\[
\mathbb{Q}_{\Msp_{\ff',\dd}^{\zeta}}\rightarrow \iota_{\ff,\ff',*}\mathbb{Q}_{\Msp_{\ff,\dd }^{\zeta}}.
\]
\end{proof}
Similarly, the natural isomorphisms $ \phim{\WW^{\zeta}_{\dd}}\pi ^{\zeta}_{\ff,\dd ,!}\mathbb{Q}_{\Msp_{\ff,\dd }^{\zeta}}\cong \pi ^{\zeta}_{\ff,\dd ,!}\phim{\WW^{\zeta}_{\ff,\dd }}\mathbb{Q}_{\Msp_{\ff,\dd }^{\zeta}}$, along with exactness of $\phim{\WW^{\zeta}_{\dd}}$, induce natural isomorphisms 
\begin{equation}
\nu_{c,\dd}\colon \Ho(p^{\zeta}_{\dd,!}\phim{\WWW^{\zeta}_{\dd}}\mathbb{Q}_{\Mst_{\dd}^{\zeta\sst}})\cong\phim{\WW^{\zeta}_{\dd}}\Ho(p_{\dd,!}^{\zeta}\mathbb{Q}_{\Mst_{\dd}^{\zeta\sst}}).
\end{equation}
\begin{proposition}
\label{bcprop}
With $Q$ as above a finite quiver, $\dd\in\mathbb{N}^{Q_0}$ a dimension vector, and $W$ a potential for $Q$, there are (non natural) isomorphisms 
\begin{align*}
\HO\left(\Mst^{\zeta\sst}_{\dd},\phim{\WWW_{\dd}^{\zeta}}\ICS_{\Mst_{\dd}^{\zeta\sst}}(\mathbb{Q})\right)\cong& \Ho(\dim_*\Ho(p^{\zeta}_{\dd,*}\phim{\WWW_{\dd}^{\zeta}}\ICS_{\Mst_{\dd}^{\zeta\sst}}(\QQ)))\\
\HO_c\left(\Mst^{\zeta\sst}_{\dd},\phim{\WWW_{\dd}^{\zeta}}\ICS_{\Mst_{\dd}^{\zeta\sst}}(\mathbb{Q})\right)\cong& \Ho(\dim_!\Ho(p^{\zeta}_{\dd,!}\phim{\WWW_{\dd}^{\zeta}}\ICS_{\Mst_{\dd}^{\zeta\sst}}(\QQ)))\\
\HO_c\left(\Mst^{\zeta\sst,\SP}_{\dd},\phim{\WWW_{\dd}^{\zeta}}\ICS_{\Mst_{\dd}^{\zeta\sst}}(\mathbb{Q})\right)\cong&\Ho(\dim_!\tilde{\omega}_{\dd}^{\zeta,*}\Ho(p^{\zeta}_{\dd,!}\phim{\WWW_{\dd}^{\zeta}}\ICS_{\Mst_{\dd}^{\zeta\sst}}(\QQ)))
\end{align*}
where $\tilde{\omega}^{\zeta}_{\dd}$ is as in (\ref{tomdef}).
\end{proposition}
\begin{proof}
Fix $n\in\mathbb{Z}$.  We prove the existence of the final isomorphism, the proof of the second follows from the special case $\Mst^{\zeta\sst,\SP}_{\dd}=\Mst_{\dd}^{\zeta\sst}$, and the proof of the first follows from commutativity of Verdier duality with $\phim{\WWW^{\zeta}_{\dd}}$.  For sufficiently large $\ff$ we may write the degree $n$ piece of the left hand side as
\[
\Ho^n\left(\dim_!\pi^{\zeta,\SP}_{\ff,\dd,!}\tilde{\omega}_{\ff,\dd}^{\zeta,*}\phim{\WW^{\zeta}_{\ff,\dd}}\LL_{\Msp^{\zeta}_{\ff,\dd}}^{(\dd,\dd)/2}\right)
\]
while the right hand side is written as
\[
\Ho^n\left(\dim_!\tilde{\omega}^{\zeta,*}_{\dd}\Ho\left(\pi^{\zeta}_{\ff,\dd,!}\phim{\WW^{\zeta}_{\ff,\dd}}\LL_{\Msp^{\zeta}_{\ff,\dd}}^{(\dd,\dd)/2}\right)\right).
\]
The result then follows from proper base change, and the following chain of isomorphisms, coming from properness of $\pi^{\zeta}_{\ff,\dd}$, exactness of vanishing cycles functors, and the decomposition theorem
\begin{align*}
\Ho\left(\pi^{\zeta}_{\dd,\ff,!}\phim{\WW^{\zeta}_{\ff,\dd}}\LL_{\Msp^{\zeta}_{\ff,\dd}}^{(\dd,\dd)/2}\right)\cong&\Ho\left(\phim{\WW^{\zeta}_{\ff,\dd}}\pi^{\zeta}_{\dd,\ff,!}\LL_{\Msp^{\zeta}_{\ff,\dd}}^{(\dd,\dd)/2}\right)\\
\cong&\phim{\WW^{\zeta}_{\ff,\dd}}\Ho\left(\pi^{\zeta}_{\dd,\ff,!}\LL_{\Msp^{\zeta}_{\ff,\dd}}^{(\dd,\dd)/2}\right)\\
\cong&\phim{\WW^{\zeta}_{\ff,\dd}}\pi^{\zeta}_{\dd,\ff,!}\LL_{\Msp^{\zeta}_{\ff,\dd}}^{(\dd,\dd)/2}\\
\cong&\pi^{\zeta}_{\ff,\dd,!}\phim{\WW^{\zeta}_{\ff,\dd}}\LL_{\Msp^{\zeta}_{\ff,\dd}}^{(\dd,\dd)/2}.
\end{align*}
\end{proof}
\begin{remark}
Throughout the paper we work over the complex numbers, and in the algebraic setting.  For greater generality, we could instead have considered analytic mixed Hodge modules --- by Saito's work the same fundamental facts that we have used regarding the vanishing cycle functor remain true in the category of analytic mixed Hodge modules (see \cite[Sec.2]{Saito90}).  In particular, if $W$ is a formal potential such that $\WW^{\zeta}_{\dd}$ defines an analytic function in some neighbourhood $U_{\dd}$ of $0\in\Msp^{\zeta\sst}_{\dd}$, then Proposition \ref{cvs} applies in the neighbourhood $U_{\dd}$.  If, moreover, there is a Serre subcategory $\mathcal{U}$ of the category of $\zeta$-semistable $\mathbb{C}Q$-modules of slope $\mu$ such that for all $\dd\in\Lambda_{\mu}^{\zeta}$ the inclusion $\Msp^{\zeta\sst,\mathcal{U}}_{\dd}\subset \Msp^{\zeta\sst}_{\dd}$ is an analytic open inclusion factoring through an open set $U_{\dd}$ as above, then the relative versions of Theorems \ref{ThmA}, \ref{CWCT}, \ref{qea} and \ref{strongPBW} hold with unmodified proofs.  See \cite{Toda17-2} for results defining such a subcategory $\mathcal{U}$ of the category of coherent sheaves on a Calabi--Yau 3-fold.
\end{remark}

\subsection{The integrality isomorphism --- Theorem \ref{ThmA}}
Assume that $\zeta$ is a $\mu$-generic Bridgeland stability condition.  As in \cite{DaMe4} and the introduction, for $\dd\in\Lambda_{\mu}^{\zeta}\setminus \{0\}$ we define the following elements of $\MMHM(\Msp^{\zeta\sst}_{\dd})$ and $\Db(\MMHM(\pt))$, respectively
\begin{align*}
\DTS_{W,\dd}^{\zeta}\colonequals&\begin{cases} \phim{\WW^\zeta_{\dd}}\ICS_{\Msp^{\zeta \sst}_{\dd}}(\mathbb{Q})&\textrm{if }\Msp^{\zeta \st}_{\dd}\neq\emptyset ,\\0&\textrm{otherwise}\end{cases}\\
\DT^{\zeta,\SP}_{W,\dd}\colonequals &\HO_c\left(\Msp^{\zeta\sst,\SP}_{\dd},\DTS_{W,\dd}^{\zeta}\right)^{\vee}.
\end{align*}

\begin{remark}
Note that from our shift convention on $\ICS_{\Msp^{\zeta\sst}_{\dd}}(\mathbb{Q})$, along with exactness of $\phim{\WW^{\zeta}_{\dd}}$ it follows that $\DTS_{W,\dd}^{\zeta}$ is indeed a genuine (monodromic) mixed Hodge module, instead of merely being an element of $\Db(\MMHM(\Msp^{\zeta\sst}_{\dd}))$.  One can find examples in which it is not pure.  For example consider the three loop quiver $Q$ and the potential $W=x^p+y^q+z^r+axyz$, with $a\neq 0$ and $p^{-1}+q^{-1}+r^{-1}<1$.  Then $\Msp_1=\mathbb{A}^3$, and on $\Msp_1$ there is an identity $\W=W$, and the potential has an isolated singularity at the origin.  The cohomology $\Ho(\phim{W}\mathbb{Q}_{\mathbb{A}^3})$ is not pure (see \cite[Rem.3.5]{DMSS15} \cite[Ex.7.3.5]{Kul98} \cite[Ex.9.1]{Sch85}), and so
\[
\DTS_{W,1}\colonequals \Ho(\phim{\WW_{1}}\mathbb{L}^{-3/2}_{\Msp_1})
\]
is not pure either.
\end{remark}
\begin{remark}
\label{sdrem}
By Proposition \ref{basicfacts} there are isomorphisms
\begin{align*}
\DD^{\mon}_{\Msp_{\dd}^{\zeta\sst}}\phim{\WW^{\zeta}_{\dd}}\ICS_{\Msp^{\zeta\sst}_{\dd}}(\mathbb{Q})\cong& \phim{\WW^{\zeta}_{\dd}}\DD_{\Msp_{\dd}^{\zeta\sst}}\ICS_{\Msp^{\zeta\sst}_{\dd}}(\mathbb{Q})\\
\cong&\phim{\WW^{\zeta}_{\dd}}\ICS_{\Msp^{\zeta\sst}_{\dd}}(\mathbb{Q})
\end{align*}
and so
\begin{equation}
\label{sdual}
\DD^{\mon}_{\Msp^{\zeta\sst}_{\dd}}\DTS_{W,\dd}^{\zeta}\cong\DTS_{W,\dd}^{\zeta}.
\end{equation}
\end{remark}

\begin{remark}
Recall that in the special case $\Mst^{\zeta\sst,\SP}=\Mst^{\zeta\sst}$, we omit the superscript $\SP$ from all notation.  By Remark \ref{sdrem}, the definition of $\DT^{\zeta}_{W,\dd}$ can be simplified via the isomorphism
\[
\DT^{\zeta}_{W,\dd}\colonequals \HO_c\left(\Msp^{\zeta\sst}_{\dd},\DTS_{W,\dd}^{\zeta}\right)^{\vee}\cong \HO\left(\Msp^{\zeta\sst}_{\dd},\DTS_{W,\dd}^{\zeta}\right).
\]
\end{remark}

\begin{definition}
\label{abbrev}
We make the abbreviation of symbols
\[
\ICSt_{W,\dd}^{\zeta}\colonequals \phim{\WWW^{\zeta}_{\dd}}\ICS_{\Mst_{\dd}^{\zeta\sst}}(\mathbb{Q})
\]
and
\[
\ICSt_{W,\mu}^{\zeta}\colonequals \bigoplus_{\dd\in\Lambda_{\mu}^{\zeta}}\phim{\WWW^{\zeta}_{\dd}}\ICS_{\Mst_{\dd}^{\zeta\sst}}(\mathbb{Q}).
\]
\end{definition}

\begin{theorem}[Theorem \ref{ThmA}]
\label{weakPBW}
Assume that $\zeta$ is a $\mu$-generic stability condition on the quiver $Q$.  There is an isomorphism in $\Dulf(\MMHM(\Msp^{\zeta\sst}_{\mu}))$ 
\begin{equation}
\label{weakeq}
\Ho(p^{\zeta}_{\mu,*}\ICSt_{W,\mu}^{\zeta})\cong\FreeComm_{\boxtimes_{\oplus}}\left(\HO(\BC )_{\vir}\otimes\DTS_{W,\mu}^{\zeta}\right)
\end{equation}
and an isomorphism in $\Dulf(\MMHM(\Lambda_{\mu}^{\zeta}))$
\begin{equation}
\label{vweakeq}
\bigoplus_{\dd\in\Lambda_{\mu}^{\zeta}}\HO_c\left(\Mst^{\zeta\sst,\SP}_{\dd},\ICSt_{W,\dd}^{\zeta}\right)^{\vee}\cong\FreeComm_{\boxtimes_{+}}\left(\HO(\BC )_{\vir}\otimes\DT^{\zeta,\SP}_{W,\mu}\right).
\end{equation}

\end{theorem}
\begin{proof}
We start by proving that there is an isomorphism as in (\ref{weakeq}).  By definition, the left hand side of (\ref{weakeq}) is isomorphic to its total cohomology.  On the other hand, $\HO(\BC )_{\vir}$ is also isomorphic to its total cohomology, and $\DTS_{W,\mu}^{\zeta}$ is an object of $\MMHM(\Msp_{\mu}^{\zeta\sst})\subset \Dub(\MMHM(\Msp_{\mu}^{\zeta\sst}))$, and is trivially isomorphic to its total cohomology.  It follows from the exactness of $\boxtimes_{\oplus}$ (Proposition \ref{biExp}) that the right hand side of (\ref{weakeq}) is also isomorphic to its total cohomology, and so it is sufficient to construct the isomorphism (\ref{weakeq}) at each cohomological degree.  
\smallbreak
We first show the result, under the assumption that $W=0$.  That is, we show that 
\[
\Ho(p^{\zeta}_{\mu,*}\ICS_{\Mst^{\zeta\sst}_{\mu}}(\mathbb{Q}))\cong\FreeComm_{\boxtimes_{\oplus}}(\HO(\BC )_{\vir}\otimes\DTS_{W=0,\mu}^{\zeta}).
\]
For this special case, the idea of the proof is to use the purity of all relevant complexes of monodromic mixed Hodge modules to upgrade the results of \cite{Meinhardt14}, from equalities in a Grothendieck ring, to isomorphisms.
\smallbreak
By Lemma \ref{appLem}, for fixed $n\in\mathbb{Z}$ and $\ff\gg 0$ the map 
\[
\Ho^n\left(\pi _{\ff,\dd ,*}^{\zeta}\LL^{(\dd,\dd)/2}_{\Msp_{\ff,\dd }^{\zeta}}\right)\rightarrow\Ho^n\left(p^{\zeta}_{\mu,*}\ICSt_{0,\dd}^{\zeta}\right)
\]
is an isomorphism.  It follows that $\Ho(p^{\zeta}_{\mu,*}\ICSt_{0,\mu}^{\zeta})$ is pure, since $\Ho^n(\pi _{\ff,\dd ,*}^{\zeta}\LL^{(\dd,\dd)/2}_{\Msp_{\ff,\dd }^{\zeta}})$ is a pure mixed Hodge module of weight $n$, as purity is preserved by direct image along proper maps.  It follows that $\Ho(p^{\zeta}_{\mu,*}\ICSt^{\zeta}_{0,\mu})$ is locally finite in the sense of Definition \ref{lfDef}, as is $\FreeComm_{\boxtimes_{\oplus}}(\HO(\BC )_{\vir}\otimes\DTS_{0,\mu}^{\zeta})$, by Lemma \ref{lfRem}, since 
\[
\HO(\BC )_{\vir}\otimes\DTS_{0,\mu}^{\zeta}\in\Dulf(\MMHM(\Msp_{\mu}^{\zeta\sst}\setminus\Msp^{\zeta\sst}_0)).  
\]
The element $\FreeComm_{\boxtimes_{\oplus}}(\HO(\BC )_{\vir}\otimes \DTS_{0,\mu}^{\zeta})$ is also a pure locally finite complex of monodromic mixed Hodge modules, by Proposition \ref{symmPure} and Lemma \ref{lfRem}, and so both $\Ho(p^{\zeta}_{\mu,*}\ICS_{\Mst^{\zeta\sst}_{\mu}}(\mathbb{Q}))$ and $\FreeComm_{\boxtimes_{\oplus}}(\HO(\BC )_{\vir}\otimes\DTS_{0,\mu}^{\zeta})$ are direct sums of simple pure mixed Hodge modules, and are isomorphic if and only if they have the same class in $\Ka_0(\Dulf(\MMHM(\Msp^{\zeta\sst}_{\mu})))$.  

Inside $\Ka_0(\MMHM(\Msp^{\zeta\sst}_{\dd}))$ consider the subgroup $\mathcal{I}_i$ generated by classes of monodromic mixed Hodge modules of weight greater than $i$.  Then we identify each $\Ka_0(\Dulf(\MMHM(\Msp^{\zeta\sst}_{\dd})))$ with the completion of $\Ka_0(\MMHM(\Msp^{\zeta\sst}_{\dd}))$ with respect to $\mathcal{I}_i$, allowing us to make sense of the limits below.  For an object $\mathcal{F}\in \Dulf(\MMHM(\Msp^{\zeta\sst}_{\mu}))$ we let $[\mathcal{F}]_{\Ka_0}$ denote the corresponding class in the Grothendieck group of $\Dulf(\MMHM(\Msp^{\zeta\sst}_{\mu}))$.  From the fact that the map $\Phi^n_{0,\ff,\dd }$ of equation (\ref{PhiBar}) is an isomorphism for fixed $n$ and $\ff\gg 0$ we deduce that
\begin{align*}
\left[\Ho\left(p^{\zeta}_{\mu,*}\ICS_{\Mst^{\zeta\sst}_{\mu}}(\mathbb{Q})\right)\right]_{\Ka_0}=&\lim_{\ff\mapsto \infty}[\bigoplus_{\dd\in\Lambda_{\mu}^{\zeta}}\LL^{\ff\cdot \dd/2}\otimes\pi _{\ff,\dd ,*}^{\zeta}\ICS_{\Msp_{\ff,\dd }^{\zeta\sst}}(\mathbb{Q})]_{\Ka_0}\\=&\lim_{\ff\mapsto\infty}\left[\FreeComm_{\boxtimes_{\oplus}}\left(\HO(\mathbb{P}^{\ff\cdot \dd-1})_{\vir}\otimes\LL^{\ff\cdot \dd/2}\otimes \DTS_{0,\mu}^{\zeta}\right)\right]_{\Ka_0}\\=&\left[\Sym_{\boxtimes_{\oplus}}\left(\HO(\BC)_{\vir}\otimes \DTS^{\zeta}_{0,\mu}\right)\right]_{\Ka_0}
\end{align*}
as required  --- for the second equality we have used the main result of \cite{Meinhardt14}.  
\smallbreak
For the case $W\neq 0$, by Proposition \ref{TSeq} we deduce the existence of isomorphisms
\begin{align*}
\Ho\left(p^{\zeta}_{\mu,*}\phim{\WWW^{\zeta}_{\dd}}\ICS_{\Mst^{\zeta\sst}_{\mu}}(\mathbb{Q})\right)&\cong^{\nu_{\dd}}\phim{\WW^{\zeta}_{\dd}}\Ho\left(p^{\zeta}_{\mu,*}\ICS_{\Mst^{\zeta\sst}_{\mu}}(\mathbb{Q})\right)\\
&\cong \phim{\WW^{\zeta}_{\dd}}\FreeComm_{\boxtimes_{\oplus}}\left(\HO(\BC)_{\vir}\otimes\DTS^{\zeta}_{0,\mu}\right)\\
&\cong \FreeComm_{\boxtimes_{\oplus}}\left(\HO(\BC)_{\vir}\otimes \phim{\WW^{\zeta}_{\mu}}\DTS^{\zeta}_{0,\mu}\right)\\
&= \FreeComm_{\boxtimes_{\oplus}}\left(\HO(\BC)_{\vir}\otimes \DTS^{\zeta}_{W,\mu}\right)
\end{align*}
where for the third isomorphism we have used that $\phim{\WW_{\mu}^{\zeta}}$ is a symmetric monoidal functor (Proposition \ref{TSeq}).  
\smallbreak
For the absolute case, we compose the following chain of isomorphisms, the first of which is Proposition \ref{bcprop}.
\begin{align*}
\bigoplus_{\dd\in\Lambda_{\mu}^{\zeta}}\HO_c\left(\Mst^{\zeta\sst,\SP}_{\mu},\ICSt_{W,\mu}^{\zeta}\right)^{\vee}\cong&\bigoplus_{\dd\in\Lambda_{\mu}^{\zeta}}\Ho\left(\dim_!\tilde{\omega}_{\dd}^{\zeta,*}\Ho\left(p^{\zeta}_{\dd,!}\ICSt_{W,\mu}^{\zeta}\right)\right)^{\vee}\\
\cong&\Ho\left(\dim_*\tilde{\omega}_{\mu}^{\zeta,!}\Ho\left(p^{\zeta}_{\dd,*}\ICSt_{W,\mu}^{\zeta}\right)\right)\\
\cong&\Ho\left(\dim_*\tilde{\omega}_{\mu}^{\zeta,!}\Sym_{\boxtimes_{\oplus}}\left(\HO(\BC)_{\vir}\otimes \DTS_{W,\mu}^{\zeta}\right)\right)\\
\cong&\Ho\left(\Sym_{\boxtimes_+}\left(\HO(\BC)_{\vir}\otimes \dim_*\tilde{\omega}_{\dd}^{\zeta,!}\DTS_{W,\mu}^{\zeta}\right)\right)\\
\cong&\Sym_{\boxtimes_+}\left(\bigoplus_{\dd\in\Lambda_{\mu}^{\zeta}}\left(\HO(\BC)_{\vir}\otimes\HO_c(\Msp_{\dd}^{\zeta\sst\SP},\DTS_{W,\dd}^{\zeta})^{\vee}\right)\right)
\end{align*}
to construct the isomorphism (\ref{vweakeq}).
\end{proof}

\begin{corollary}
\label{inccor}
There is a canonical split inclusion 
\begin{equation}
\label{caninc}
\Upsilon_{W,\mu}^{\zeta}\colon\HO(\BC)_{\vir}\otimes\DTS^{\zeta}_{W,\mu}\rightarrow \Ho\left(p_{\mu,*}^{\zeta}\ICSt_{W,\mu}^{\zeta}\right)
\end{equation}
\end{corollary}
\begin{proof}
First we prove the corollary in the case $W=0$.  In this case, the left hand side of (\ref{caninc}) is precisely the summand of the right hand side with strict supports equal to $\Msp^{\zeta\sst}_{\dd}$ for $\dd\in\Lambda^{\zeta}_{\mu}$ by Theorem \ref{weakPBW}, and the result follows from the decomposition theorem.  For the general case, we apply the functor $\phim{\WW_{\mu}^{\zeta}}$ to the inclusion, and its left inverse, from the case $W=0$.
\end{proof}

\subsection{The cohomological wall crossing isomorphism --- Theorem \ref{CWCT}}
\label{cwcfSec}
In this section we prove Theorem \ref{CWCT}, which is a categorification of the motivic wall crossing formula (\ref{WCI}).  In Section \ref{CWCF} we will prove that this isomorphism can be realised in terms of multiplication in the cohomological Hall algebra.  Even without considering the CoHA we can still prove that there is \textit{some} isomorphism (\ref{gthmd}).  The mere existence of this isomorphism is enough for many applications.

Fix a dimension vector $\dd\in\mathbb{N}^{Q_0}$.  Denote by $\HN_{\dd}$ the set of Harder--Narasimhan types for $\dd$, that is, sequences $\dd^1,\ldots,\dd^s\in\mathbb{N}^{Q_0}\setminus \{0\}$ such that the slopes $\Mu^{\zeta}(\dd^1),\ldots,\Mu^{\zeta}(\dd^s)$ are strictly decreasing and $\sum_{i=1}^s \dd^i=\dd$.
Recall that by \cite[Prop.3.4]{Reineke_HN}, the moduli stack $\Mst_{\dd}$ has a stratification by locally closed substacks 
\begin{equation}
\label{HNstrat}
\Mst_{\dd}=\coprod_{\overline{\dd}\in\HN_{\dd}}\Mst_{\overline{\dd}}^{\zeta}
\end{equation}
where $\Mst_{\overline{\dd}}^{\zeta}$ is the stack of representations which have Harder--Narasimhan type $\overline{\dd}$ with respect to the stability condition $\zeta$.  

This is a stratification in the following weak sense.  We define similarly to \cite[Def.3.6]{Reineke_HN} a partial ordering $\leq'$ on $\HN_{\dd}$.  First, given $\overline{\dd}=(\dd^1,\ldots,\dd^s)\in\HN_{\dd}$ we define $P(\overline{\dd})$ to be the convex hull in $\mathbb{C}$ of the origin and the points $(\sum_{l=1}^k\lvert \dd^l\rvert,\Mu^{\zeta}(\dd^l)\lvert\dd^l\rvert)$ for $1\leq l\leq s$.  As in \cite[Prop.3.7]{Reineke_HN} it is easy to show that for a representation $\rho$ of Harder--Narasimhan type $\overline{\dd}$, the polytope $P(\overline{\dd})$ is the convex hull of $0,\zeta\cdot\dd$ and $\zeta\cdot\dim(\rho')$ for all representations $\rho'\subset \rho$ with slope greater than $\Mu^{\zeta}(\dd)$. We write $\ee\leq'\dd$ if the polygon $P(\dd)$ is contained in $P(\ee)$.  The condition on a representation $\rho$ to have a subrepresentation $\rho'\subset\rho$ of fixed dimension $\dd'$ is closed, since the morphism $\Mst_{\dd',\dd-\dd'}\rightarrow \Mst_{\dd}$ is proper.  It follows that
\[
\overline{\Mst_{\overline{\dd}}^{\zeta}}\subset \bigcup_{\substack{\overline{\ee}\in\HN_{\dd}\\ \overline{\ee}\leq' \overline{\dd}}}\Mst_{\overline{\ee}}^{\zeta}.
\]
Each of the stacks $\Mst^{\zeta}_{\overline{\dd}}$ can be written as a global quotient stack
\[
\Mst^{\zeta}_{\overline{\dd}}\cong X^{\zeta}_{\overline{\dd}}/G_{\overline{\dd}}
\]
where $X^{\zeta}_{\overline{\dd}}\subset X_{\dd}$ is the subspace of representations preserving the flag defined by $\overline{\dd}$, such that each of the associated $\dd^r$-dimensional subquotients is $\zeta$-semistable, and $G_{\overline{\dd}}\subset G_{\dd}$ is the subgroup preserving the same flag.  Each of these stacks comes with a map $p_{\overline{\dd}}\colon \Mst^{\zeta}_{\overline{\dd}}\rightarrow\Msp_{\dd}$ sending a representation to its semisimplification, and an inclusion $i_{\overline{\dd}}\colon \Mst^{\zeta}_{\overline{\dd}}\rightarrow \Mst_{\dd}$ given by forgetting the Harder-Narasimhan filtration.  The diagram
\[
\xymatrix{
\Mst_{\overline{\dd}}^{\zeta}\ar[d]^{p_{\overline{\dd}}}\ar[r]^{i_{\overline{\dd}}}&\Mst_{\dd}\ar[dl]^{p_{\dd}}\\
\Msp_{\dd}
}
\]
commutes.  In the following theorem, the map $q^{\zeta}_{\mu}$ is as defined in (\ref{qmudef}).
\begin{theorem}
\label{gcDT}
Let $\zeta$ be a Bridgeland stability condition on $Q$, not necessarily generic.  Then there are isomorphisms
\begin{equation}
\label{ngwc}
\Ho\left(p_!\phim{\WWW}\ICS_{\Mst}(\mathbb{Q})\right)\cong\Boxtimes_{\oplus, -\infty\xrightarrow{\mu} \infty}^{\tw} \Ho\left(q_{\mu,!}^{\zeta}p_{\mu,!}^{\zeta}\phim{\WWW_{\mu}^{\zeta}}\ICS_{\Mst_{\mu}^{\zeta\sst}}(\mathbb{Q})\right)
\end{equation}
and
\begin{equation}
\label{ngwcnc}
\Ho\left(p_*\phim{\WWW}\ICS_{\Mst}(\mathbb{Q})\right)\cong\Boxtimes_{\oplus, \infty\xrightarrow{\mu} -\infty}^{\tw} \Ho\left(q_{\mu,*}^{\zeta}p_{\mu,*}^{\zeta}\phim{\WWW_{\mu}^{\zeta}}\ICS_{\Mst_{\mu}^{\zeta\sst}}(\mathbb{Q})\right).
\end{equation}

Assume in addition that $\zeta$ is generic.  Then there are isomorphisms
\begin{equation}
\label{gwPBW}
\Ho\left(p_!\phim{\WWW}\ICS_{\Mst}(\mathbb{Q})\right)\cong \Boxtimes_{\oplus,-\infty \xrightarrow{\mu}\infty}^{\tw} \Sym_{\boxtimes_{\oplus}}\left(q_{\mu,!}^{\zeta}\DTS_{W,\mu}^{\zeta}\otimes\HO(\BC )^{\vee}_{\vir}\right)
\end{equation}
and
\begin{equation}
\label{gwPBWnc}
\Ho\left(p_*\phim{\WWW}\ICS_{\Mst}(\mathbb{Q})\right)\cong \Boxtimes_{\oplus, \infty\xrightarrow{\mu} -\infty }^{\tw} \Sym_{\boxtimes_{\oplus}}\left(q_{\mu,*}^{\zeta}\DTS_{W,\mu}^{\zeta}\otimes\HO(\BC )_{\vir}\right).
\end{equation}
\end{theorem}
\begin{proof}
As in the proof of Theorem \ref{weakPBW} we only need to prove the case for which $W=0$ and then we can deduce the general case from the fact that monodromic vanishing cycle functors are exact, and commute with proper maps and the relevant monoidal structures.  We first prove that the isomorphism (\ref{ngwc}) exists.  Fix a $\dd\in\mathbb{N}^{Q_0}$.  Then if we complete $\geq'$ to a total ordering $\geq$ of $\HN_{\dd}$ and define, for each $\overline{\dd}\in\HN_{\dd}$,
\begin{align*}
\Mst_{\leq \overline{\dd}}^{\zeta}\colonequals &\bigcup_{\overline{\ee}\leq \overline{\dd}}\Mst_{\overline{\ee}}^{\zeta}\\
\Mst_{< \overline{\dd}}^{\zeta}\colonequals &\bigcup_{\overline{\ee}< \overline{\dd}}\Mst_{\overline{\ee}}^{\zeta}
\end{align*}
then $\Mst_{\leq \overline{\dd}}^{\zeta}\subset \Mst_{\dd}$ is a closed embedding, and $\Mst_{\overline{\dd}}^{\zeta}\subset\Mst_{\leq \overline{\dd}}^{\zeta}$ is an open embedding with complement $\Mst_{<\overline{\dd}}^{\zeta}\subset\Mst_{\leq \overline{\dd}}^{\zeta}$.  We denote by 
\begin{align}
i_{\overline{\dd}}\colon &\Mst_{\overline{\dd}}^{\zeta}\hookrightarrow\Mst_{\dd} \label{id}\\
i_{<\overline{\dd}}\colon &\Mst_{<\overline{\dd}}^{\zeta}\hookrightarrow\Mst_{\dd}\label{idd}\\
i_{\leq\overline{\dd}}\colon &\Mst_{\leq \overline{\dd}}^{\zeta}\hookrightarrow\Mst_{\dd}\label{iddd}
\end{align}
the natural inclusions.  We will show that all of the terms in the following distinguished triangle are pure, and that the connecting maps are zero, and furthermore that the triangle is split:
\begin{equation}
\label{HNtri}
\Ho\left(p_{\dd,!}i_{\overline{\dd},!}\QQ_{\Mst^{\zeta}_{\overline{\dd}}}\right)\rightarrow \Ho\left(p_{\dd,!}i_{\leq \overline{\dd},!}\QQ_{\Mst^{\zeta}_{\leq \overline{\dd}}}\right)\rightarrow \Ho\left(p_{\dd,!}i_{< \overline{\dd},!}\QQ_{\Mst^{\zeta}_{< \overline{\dd}}}\right)\rightarrow.
\end{equation}
Fix $\overline{\dd}\in\HN_{\dd}$.  Consider the following commutative diagram
\[
\xymatrix{
&\Mst_{\overline{\dd}}^{\zeta}\ar[d]^{q_{\overline{\dd}}}\ar[r]^-{i_{\overline{\dd}}}&\Mst_{\dd}\ar[ddd]^{p_{\dd}}\\
X^{\zeta}_{\overline{\dd}}/(G_{\dd^1}\times\ldots \times G_{\dd^s})\ar[r]^{q_{2,\overline{\dd}}}\ar[ur]^{q_{1,\overline{\dd}}}&\Mst_{\dd^1}^{\zeta\sst}\times\ldots\times\Mst_{\dd^s}^{\zeta\sst}\ar[d]^{p^{\zeta}_{\dd^1}\times\ldots\times p^{\zeta}_{\dd^s}}\\
&\Msp^{\zeta\sst}_{\dd^1}\times\ldots\times\Msp^{\zeta\sst}_{\dd^s}\ar[d]^{q^{\zeta}_{\dd^1}\times\ldots\times q^{\zeta}_{\dd^s}}\\
&\Msp_{\dd^1}\times\ldots\times\Msp_{\dd^s}\ar[r]^-{\oplus}&\Msp_{\dd}
}
\]
where $q_{1,\overline{\dd}}$ is an affine fibration of relative dimension 
\[
f_1(\overline{\dd})\colonequals \sum_{1\leq r<r'\leq s}\dd^r\cdot \dd^{r'}
\]
and $q_{2,\overline{\dd}}$ is an affine fibration of relative dimension 
\[
f_2(\overline{\dd})\colonequals \sum_{1\leq r<r'\leq s}\sum_{a\in Q_1}\dd^{r'}_{s(a)}\dd^{r}_{t(a)}.  
\]
We define 
\begin{equation}
\label{dddotdef}
(\overline{\dd},\overline{\dd})\colonequals f_1(\overline{\dd})-f_2(\overline{\dd}).  
\end{equation}
For $\ff\in\mathbb{N}^{Q_0}$ we define $\ff\cdot\overline{\dd}=\sum_{r\leq s}\ff\cdot\dd^r$.
Fix $n$, and let $\ff\gg 0$.  Then since there are isomorphisms
\begin{align*}
\Ho^n\left(p_{\dd,!}i_{\overline{\dd},!}\QQ_{\Mst_{\overline{\dd}}}\right)
\cong& \Ho^n\left(\oplus_!(q^{\zeta}_{\dd^1}\times\ldots\times q^{\zeta}_{\dd^s})_!(p^{\zeta}_{\dd^1}\times\ldots\times p^{\zeta}_{\dd^s})_!\LL^{-(\overline{\dd},\overline{\dd})}_{\Mst_{\dd^1}^{\zeta\sst}\times\ldots\times\Mst_{\dd^s}^{\zeta\sst}}\right)\\
\cong& \Ho^n\left(\oplus_!(q^{\zeta}_{\dd^1}\times\ldots\times q^{\zeta}_{\dd^s})_!(\pi^{\zeta}_{\ff,\dd ^1}\times\ldots\times \pi^{\zeta}_{\ff,\dd ^s})_!\LL^{-(\overline{\dd},\overline{\dd})-\ff\cdot\overline{\dd}}_{\Msp_{\ff,\dd ^1}^{\zeta}\times\ldots\times\Msp^{\zeta}_{\ff,\dd ^s}}\right)\\
\end{align*}
the complex $\Ho^n\left(p_{\dd,!}i_{\overline{\dd},!}\QQ_{\Mst_{\overline{\dd}}}\right)$ is pure, by properness of $q^{\zeta}_{\dd^r}$, $\pi^{\zeta}_{\ff,\dd^r}$ and $\oplus$.  The claim regarding the distinguished triangles (\ref{HNtri}) then follows by induction, and semisimplicity of the category of pure mixed Hodge modules.  Taking care of the twists, we calculate
\begin{align}
\label{int1}
\Ho\left(p_!\ICS_{\Mst}(\QQ)\right)=&\Ho\left(p_!\LL^{(\dd,\dd)/2}_{\Mst}\right)\\
\nonumber\cong&\bigoplus_{\overline{\dd}\in\HN}\Ho\left(p_{\dd,!}i_{\overline{\dd},!}\LL^{(\dd,\dd)/2}_{\Mst_{\overline{\dd}}}\right)\\
\nonumber\cong&\bigoplus_{(\dd^1,\ldots,\dd^s)\in\HN}q^{\zeta}_{\dd^1,!}\Ho\left(p^{\zeta}_{\dd^1,!}\QQ_{\Mst_{\dd^1}^{\zeta\sst}}\right)\boxtimes_{\oplus}\ldots\boxtimes_{\oplus} q^{\zeta}_{\dd^s,!}\Ho\left(p^{\zeta}_{\dd^s,!}\QQ_{\Mst_{\dd^s}^{\zeta\sst}}\right)\otimes\LL^{-(\overline{\dd},\overline{\dd})+(\dd,\dd)/2}\\
\nonumber\cong&\bigoplus_{(\dd^1,\ldots,\dd^s)\in\HN}q^{\zeta}_{\dd^s,!}\Ho\left(p^{\zeta}_{\dd^s,!}\ICS_{\Mst_{\dd^s}^{\zeta\sst}}(\QQ)\right)\boxtimes^{\tw}_{\oplus}\ldots\boxtimes^{\tw}_{\oplus} q^{\zeta}_{\dd^1,!}\Ho\left(p^{\zeta}_{\dd^1,!}\ICS_{\Mst_{\dd^1}^{\zeta\sst}}(\QQ)\right)\\
\nonumber \cong&\Boxtimes_{\oplus, -\infty\xrightarrow{\mu} \infty}^{\tw} q_{\mu,!}^{\zeta}\Ho\left(p_{\mu,!}^{\zeta}\ICS_{\Mst_{\mu}^{\zeta\sst}}(\mathbb{Q})\right).
\end{align}
This completes the proof of the first part of the theorem: since $p^{\zeta}_{\mu}$ is approximated by proper maps, and $q^{\zeta}_{\mu}$ is proper, there are isomorphisms
\[
q^{\zeta}_{\mu,!} \Ho(p^{\zeta}_{\mu,!}\ICS_{\Mst_{\mu}^{\zeta\sst}}(\mathbb{Q}))\cong \Ho(q^{\zeta}_{\mu,!}p^{\zeta}_{\mu,!}\ICS_{\Mst_{\mu}^{\zeta\sst}}(\mathbb{Q}))
\]
as required.

Since $q^{\zeta}_{\mu}\colon \Msp^{\zeta\sst}_{\mu}\rightarrow\Msp_{\mu}$ is a morphism of monoids, $q^{\zeta}_{\mu,!}$ is a symmetric monoidal functor.  If $\zeta$ is generic, we may then apply Theorem \ref{weakPBW} at each slope $\mu\in(-\infty,\infty)$, to deduce (\ref{gwPBW}).  Finally we prove that there are isomorphisms (\ref{ngwcnc}) and (\ref{gwPBWnc}).  Note that it is enough to construct the restrictions of these isomorphisms to each space $\Msp_{\dd}$.  Restricting to this component, there are finitely many tensor products of sheaves on the left hand side and the right hand side of (\ref{ngwcnc}) and (\ref{gwPBWnc}), involving $\Ho(p_{\dd',*}\ICS_{\Mst_{\dd'}}(\QQ))$ and $\Ho(p^{\zeta}_{\dd',*}\ICS_{\Mst_{\dd'}^{\zeta\sst}}(\mathbb{Q}))$ for $\dd'\in\Lambda_{\mu}^\zeta$ satisfying $\dd'\leq \dd$.  Then the isomorphisms are given by taking the Verdier duals of (\ref{ngwc}) and (\ref{gwPBW}) respectively, by Remarks \ref{finSD} and \ref{twDual}.
\end{proof}
Arguing as in the proof of Theorem \ref{weakPBW}, using properness of $q^{\zeta}$ and approximation of $p_{\mu}^{\zeta}$ by proper maps for the relevant base change isomorphisms, we deduce the corollary
\begin{corollary}
Let $\zeta$ be a Bridgeland stability condition for $Q$, not necessarily assumed to be generic.  Then there is an isomorphism of $\mathbb{N}^{Q_0}$-graded monodromic mixed Hodge structures
\[
\HO_c(\Mst^{\SP},\ICSt_{W})^{\vee}\cong\Boxtimes_{+, \infty\xrightarrow{\mu} -\infty}^{\tw}\HO_c\left(\Mst^{\zeta\sst,\SP}_{\mu},\ICSt_{W,\mu}^{\zeta}\right)^{\vee}
\]
if $\zeta$ is generic, we may combine this isomorphism with the absolute integrality isomorphisms for each slope $\mu$ (Theorem \ref{ThmA}) to obtain the isomorphism
\[
\HO_c(\Mst^{\SP},\ICSt_W)^{\vee}\cong\Boxtimes_{+, \infty\xrightarrow{\mu} -\infty}^{\tw}\Sym_{\boxtimes_+}\left(\HO(\BC)_{\vir}\otimes\DT^{\zeta,\SP}_{\mu}\right).
\]
\end{corollary}

\begin{remark}
In the case $W=0$, and $\SP=\mathbb{C}Q\mathrm{-mod}$ we obtain the isomorphism
\begin{equation}
\HO\left(\Mst,\ICS_{\Mst}(\mathbb{Q}) \right)\cong\Boxtimes_{+, \infty\xrightarrow{\mu} -\infty}^{\tw} \HO\left(\Mst_{\mu}^{\zeta\sst},\ICS_{\Mst_{\mu}^{\zeta\sst}}(\mathbb{Q})\right).
\end{equation}
The existence of such an isomorphism is proved by Franzen and Reineke in \cite{FrRe15} via a vanishing result for even cohomology, and before that for the case of a Dynkin quiver which is not an orientation of $E_8$ by Rimanyi in \cite{Ri13}, where this result is a corollary of the existence of a Poincar\'e--Birkhoff--Witt isomorphism for hypercohomology.
\end{remark}



\section{Cohomological Hall algebras}
\label{CoHAsec}
\subsection{The relative cohomological Hall algebra}
\label{relc}
In this section we define the relative cohomological Hall algebra $\Ho(\Coha_{W,\mu}^{\zeta})$.  For now we will not assume that the stability condition $\zeta\in\mathbb{H}_+^{Q_0}$ is generic.  The underlying cohomologically graded monodromic mixed Hodge module of $\Ho(\Coha_{W,\mu}^{\zeta})$ is $\Ho\left(p^{\zeta}_{\mu,*}\ICSt_{W,\mu}^{\zeta}\right)$ from Definition \ref{abbrev} and Theorem \ref{weakPBW}.  We will define morphisms  
\begin{equation}
\label{multint}
\Ho\left(\ms_{W,\dd',\dd''}^{\zeta}\right)\colon\Ho\left(p^{\zeta}_{\dd',*}\ICSt_{W,\dd'}^{\zeta}\right)\boxtimes^{\tw}_{\oplus} \Ho\left(p^{\zeta}_{\dd'',*}\ICSt_{W,\dd''}^{\zeta}\right)\rightarrow\\   \Ho\left(p^{\zeta}_{\dd,*} \ICSt_{W,\dd}^{\zeta}\right)
\end{equation}
for all $\dd=\dd'+\dd''$ with $\dd',\dd''\in\Lambda_{\mu}^{\zeta}$ satisfying the natural associativity condition for a monoid in the category $\Dulf(\MMHM(\Msp^{\zeta\sst}_{\mu}))$ with the twisted monoidal product $\boxtimes^{\tw}_{\oplus}$.  The result is a \textit{relative} version of the cohomological Hall algebra of Kontsevich and Soibelman \cite{KS2} in the sense explained in Section \ref{abvsrel}.
\smallbreak
We define (\ref{multint}) as the composition of two morphisms, the first of which is defined in terms of the commutative diagram
\[
\xymatrix{
&X^{\zeta\sst}_{\dd',\dd''}/\left(G_{\dd'}\times G_{\dd''}\right)\ar[dl]_{r_1}\ar[dr]^{r_2}\\
\left(X_{\dd'}^{\zeta\sst}\times X_{\dd''}^{\zeta\sst}\right)/\left(G_{\dd'}\times G_{\dd''}\right)\ar[ddr]^{p_{\dd'}^{\zeta}\times p_{\dd''}^{\zeta}}\ar[d]^{\cong} &&X_{\dd',\dd''}^{\zeta\sst}/G_{\dd',\dd''}\ar[d]^{\cong}\ar[ddl]_{p_{\dd',\dd''}^{\zeta}}\\
\Mst_{\dd'}^{\zeta\sst}\times\Mst_{\dd''}^{\zeta\sst}\ar[dr]&&\Mst_{\dd',\dd''}^{\zeta\sst}\ar[dl]\\
&\Msp^{\zeta\sst}_{\dd'}\times\Msp^{\zeta\sst}_{\dd''}\ar[r]^-{\oplus}&\Msp^{\zeta\sst}_{\dd}.
}
\]
We treat the vertical isomorphisms as identities in what follows.  We consider the following composition of isomorphisms
\begin{align*}
&\Ho\left(p^{\zeta}_{\dd',*}\phim{\WWW^{\zeta}_{\dd'}}\QQ_{\Mst_{\dd'}^{\zeta\sst}}\right)\boxtimes \Ho\left(p^{\zeta}_{\dd'',*}\phim{\WWW^{\zeta}_{\dd''}}\QQ_{\Mst^{\zeta\sst}_{\dd''}}\right)\cong^{\TS}\\
&\Ho\left((p^{\zeta}_{\dd'}\times p^{\zeta}_{\dd''})_*\phim{\WWW^{\zeta}_{\dd'}\boxplus\WWW^{\zeta}_{\dd''}}\QQ_{\Mst^{\zeta\sst}_{\dd'}\times\Mst^{\zeta\sst}_{\dd''}}\right)\cong\\
& \Ho\left(((p^{\zeta}_{\dd'}\times p^{\zeta}_{\dd''})\circ r_{1})_*\phim{\WWW^{\zeta}_{\dd'}\boxplus\WWW^{\zeta}_{\dd''}\circ r_1}\QQ_{X^{\zeta\sst}_{\dd',\dd''}/\left(G_{\dd'}\times G_{\dd''}\right)}\right)\cong\\
&\Ho\left(p^{\zeta}_{\dd',\dd'',*}\phim{\WWW^{\zeta}_{\dd',\dd''}}\QQ_{\Mst_{\dd',\dd''}^{\zeta\sst}}\right)
\end{align*}
to obtain the isomorphism
\begin{align*}
\alpha^{\zeta}_{\dd',\dd''}\colon&\oplus_*\left(\Ho\left(p^{\zeta}_{\dd',*}\phim{\WWW^{\zeta}_{\dd'}}\QQ_{\Mst_{\dd'}^{\zeta\sst}}\right)\boxtimes \Ho\left(p^{\zeta}_{\dd'',*}\phim{\WWW^{\zeta}_{\dd''}}\QQ_{\Mst^{\zeta\sst}_{\dd''}}\right)\right)\xrightarrow{\cong}\\& \oplus_*\left(\Ho\left(p^{\zeta}_{\dd',\dd'',*}\phim{\WWW^{\zeta}_{\dd',\dd''}}\QQ_{\Mst_{\dd',\dd''}^{\zeta\sst}}\right)\right).
\end{align*}
Next consider the commutative diagram of stacks
\begin{equation}
\label{smap}
\xymatrix{
\Mst^{\zeta\sst}_{\dd',\dd''}\ar[d]^{\cong}\ar[rr]^-{s^{\zeta}_{\dd',\dd''}}&&\Mst_{\dd}^{\zeta\sst}\ar[d]^{\cong}\\
X^{\zeta\sst}_{\dd',\dd''}/G_{\dd',\dd''}\ar[d]^{p^{\zeta}_{\dd',\dd''}}\ar[r]^{\iota^{\zeta}_{\dd',\dd''}}&X^{\zeta\sst}_{\dd}/G_{\dd',\dd''}\ar[r]^{c^{\zeta}_{\dd',\dd''}}&X^{\zeta\sst}_{\dd}/G_{\dd}\ar[d]^{p^{\zeta}_{\dd}}\\
\Msp^{\zeta\sst}_{\dd'}\times\Msp^{\zeta\sst}_{\dd''}\ar[rr]^{\oplus}&&\Msp^{\zeta\sst}_{\dd}.
}
\end{equation}
Note that $\iota_{\dd',\dd''}^\zeta$ is a closed inclusion, while $c_{\dd',\dd''}^\zeta$ is proper and representable, and so $s_{\dd',\dd''}^\zeta$ is a proper and representable morphism of stacks.

We define $V_N=\prod_{i\in Q_0}\Hom(\mathbb{C}^N,\mathbb{C}^{\dd_i})$, and $U_N=\prod_{i\in Q_0}\Hom^{\mathrm{surj}}(\mathbb{C}^N,\mathbb{C}^{\dd_i})$.  For $N\geq \max(\dd_i)$, the group $G_{\dd}$ acts freely on $U_N$, as does $G_{\dd',\dd''}\subset G_{\dd}$.  We define 
\begin{align}
\label{XNdef}&X^{\zeta\sst}_{\dd,N}=X_{\dd}^{\zeta\sst}\times_{G_{\dd}}U_N\\
\label{XXNdef}&X^{\zeta\sst}_{\dd',\dd'',N}=X_{\dd',\dd''}^{\zeta\sst}\times_{G_{\dd',\dd''}}U_{N},
\end{align}
and let $\W^{\zeta}_{\dd,N}$ and $\W^{\zeta}_{\dd',\dd'',N}$ denote the functions induced by $\W$ on $X^{\zeta\sst}_{\dd,N}$ and $X^{\zeta\sst}_{\dd',\dd'',N}$ respectively.
Passing to the limit of the composition of morphisms
\begin{align}\label{preLim}
\beta^{\zeta}_{\dd',\dd'',N}\colon &\oplus_*\Ho\left(p^{\zeta}_{\dd',\dd'',N,*}\phim{\W^{\zeta}_{\dd',\dd'',N}}\QQ_{X_{\dd',\dd'',N}^{\zeta\sst}}\right)\xrightarrow{\cong }
\Ho\left((p^{\zeta}_{\dd,N}\circ s^{\zeta}_{\dd',\dd'',N})_*\phim{\W^{\zeta}_{\dd',\dd'',N}}\QQ_{X_{\dd',\dd'',N}^{\zeta\sst}}\right)\xrightarrow{\cong}\\ \nonumber
&\rightarrow \Ho\left(p^{\zeta}_{\dd,N,*}\phim{\W^{\zeta}_{\dd,N}}s^{\zeta}_{\dd',\dd'',N,*}\QQ_{X^{\zeta\sst}_{\dd',\dd'',N}}\right)\rightarrow \Ho\left(p^{\zeta}_{\dd,N,*}\phim{\W^{\zeta}_{\dd,N}}\LL^{(\dd',\dd'')}_{X_{\dd,N}^{\zeta\sst}}\right)
\end{align}
we obtain
\begin{align}\label{betadef}
\beta^{\zeta}_{\dd',\dd''}\colon &\oplus_*\Ho\left(p^{\zeta}_{\dd',\dd'',*}\phim{\WWW^{\zeta}_{\dd',\dd''}}\QQ_{\Mst_{\dd',\dd''}^{\zeta\sst}}\right)\rightarrow \LL^{(\dd',\dd'')}\otimes \Ho\left(p^{\zeta}_{\dd,*}\phim{\WWW^{\zeta}_{\dd}}\QQ_{\Mst_{\dd}^{\zeta\sst}}\right).
\end{align}
Here we have used that $\oplus_*\Ho\left(p^{\zeta}_{\dd',\dd'',N,*}\phim{\W^{\zeta}_{\dd',\dd'',N}}\QQ_{X_{\dd',\dd'',N}^{\zeta\sst}}(\mathbb{Q})\right)$ is naturally isomorphic to $\Ho\left(\oplus_*p^{\zeta}_{\dd',\dd'',N,*}\phim{\W^{\zeta}_{\dd',\dd'',N}}\QQ_{X_{\dd',\dd'',N}^{\zeta\sst}}(\mathbb{Q})\right)$ since $\oplus_*$ is exact, as $\oplus$ is finite.  We define the morphism $\Ho\left(\ms_{W,\dd',\dd''}^{\zeta}\right)$ to be the composition\smallbreak
\begin{align*}
&\LL^{\langle\dd'',\dd'\rangle/2}\otimes\LL^{\chi(\dd',\dd')/2}\otimes\Ho\left(p^{\zeta}_{\dd',*}\phim{\WWW^{\zeta}_{\dd'}}\QQ_{\Mst_{\dd'}^{\zeta\sst}}\right)\boxtimes \left(\LL^{\chi(\dd'',\dd'')/2}\otimes\Ho\left(p^{\zeta}_{\dd'',*}\phim{\WWW^{\zeta}_{\dd''}}\QQ_{\Mst^{\zeta\sst}_{\dd''}}\right)\right)\xrightarrow{\cong}\\
&\LL^{\langle\dd'',\dd'\rangle/2}\otimes\LL^{\chi(\dd',\dd')/2}\otimes\LL^{\chi(\dd'',\dd'')/2}\otimes\Ho\left(p^{\zeta}_{\dd',*}\phim{\WWW^{\zeta}_{\dd'}}\QQ_{\Mst_{\dd'}^{\zeta\sst}}\right)\boxtimes \Ho\left(p^{\zeta}_{\dd'',*}\phim{\WWW^{\zeta}_{\dd''}}\QQ_{\Mst^{\zeta\sst}_{\dd''}}\right)\xrightarrow{=}\\
&\LL^{\langle\dd'',\dd'\rangle/2+\chi(\dd',\dd')/2+\chi(\dd'',\dd'')/2}\otimes\Ho\left(p^{\zeta}_{\dd',*}\phim{\WWW^{\zeta}_{\dd'}}\QQ_{\Mst_{\dd'}^{\zeta\sst}}\right)\boxtimes \Ho\left(p^{\zeta}_{\dd'',*}\phim{\WWW^{\zeta}_{\dd''}}\QQ_{\Mst^{\zeta\sst}_{\dd''}}\right)\xrightarrow{\id\otimes (\beta_{\dd',\dd''}\circ \alpha_{\dd',\dd})}\\
&\LL^{\langle\dd'',\dd'\rangle/2+\chi(\dd',\dd')/2+\chi(\dd'',\dd'')/2}\otimes\LL^{\chi(\dd',\dd'')}\otimes\Ho\left(p^{\zeta}_{\dd'+\dd'',*}\phim{\WWW^{\zeta}_{\dd'+\dd''}}\QQ_{\Mst_{\dd'+\dd''}^{\zeta\sst}}\right)\xrightarrow{=}\\
&\LL^{\chi(\dd'+\dd'',\dd'+\dd'')/2}\otimes \Ho\left(p^{\zeta}_{\dd'+\dd'',*}\phim{\WWW^{\zeta}_{\dd'+\dd''}}\QQ_{\Mst_{\dd'+\dd''}^{\zeta\sst}}\right).
\end{align*}
We define
\[
\Ho\left(\ms^{\zeta}_{W,\mu}\right)=\bigoplus_{\dd',\dd''\in\Lambda_{\mu}^{\zeta}}\Ho\left(\ms_{W,\dd',\dd''}^{\zeta}\right).
\]
Let $\dd^{(h)}$ for $h=1,2,3$ be a triple of dimension vectors, and set 
\[
\JJ_{\ee}\colonequals \Ho\left(p^{\zeta}_{\ee,*}\phim{\WWW^{\zeta}_{\ee}}\QQ_{\Mst_{\ee}^{\zeta\sst}}\right).
\]
By the standard argument (see \cite{KS2}) the diagram
\begin{align}
\label{Icom}
\xymatrix{
\JJ_{\dd^{(1)}}\boxtimes_{\oplus}\JJ_{\dd^{(2)}}\boxtimes_{\oplus}\JJ_{\dd^{(3)}}\ar[d]\ar[dr]^{\kappa}\ar[r]&\LL^{\chi(\dd^{(2)},\dd^{(3)})}\otimes\JJ_{\dd^{(1)}}\boxtimes \JJ_{\dd^{(2)}+\dd^{(3)}}\ar[d]
\\
\LL^{\chi(\dd^{(1)},\dd^{(2)})}\otimes \JJ_{\dd^{(1)}+\dd^{(2)}}\boxtimes_{\oplus}\JJ_{\dd^{(3)}}\ar[r]&\LL^{\chi(\dd^{(1)},\dd^{(2)})+\chi(\dd^{(1)},\dd^{(3)})+\chi(\dd^{(2)},\dd^{(3)})}\otimes \JJ_{\dd^{(1)}+\dd^{(2)}+\dd^{(3)}}
}
\end{align}
commutes.  It follows that we can describe both the maps
\begin{align}
\label{mapone}
&\left(\Ho\left(p^{\zeta}_{\dd^{(1)},*}\ICSt_{W,\dd^{(1)}}^{\zeta}\right)\boxtimes_{\oplus}^{\tw}\Ho\left(p^{\zeta}_{\dd^{(2)},*}\ICSt_{W,\dd^{(2)}}^{\zeta}\right)\right)\boxtimes_{\oplus}^{\tw}\Ho\left(p^{\zeta}_{\dd^{(3)},*}\ICSt_{W,\dd^{(3)}}^{\zeta}\right)\rightarrow \\&\Ho\left(p^{\zeta}_{\dd^{(1)}+\dd^{(2)}+\dd^{(3)},*}\ICSt_{W,\dd^{(1)}+\dd^{(2)}+\dd^{(3)}}^{\zeta}\right)\nonumber
\end{align}
and
\begin{align}
\label{maptwo}
&\Ho\left(p^{\zeta}_{\dd^{(1)},*}\ICSt_{W,\dd^{(1)}}^{\zeta}\right)\boxtimes_{\oplus}^{\tw}\left(\Ho\left(p^{\zeta}_{\dd^{(2)},*}\ICSt_{W,\dd^{(2)}}^{\zeta}\right)\boxtimes_{\oplus}^{\tw}\Ho\left(p^{\zeta}_{\dd^{(3)},*}\ICSt_{W,\dd^{(3)}}^{\zeta}\right)\right)\rightarrow \\&\Ho\left(p^{\zeta}_{\dd^{(1)}+\dd^{(2)}+\dd^{(3)},*}\ICSt_{W,\dd^{(1)}+\dd^{(2)}+\dd^{(3)}}^{\zeta}\right)
\end{align}
via the recipe: move all of the half Tate twists to the front, and then apply the map $\kappa$ to the final three tensor factors.  On the other hand, the natural map between the domains of (\ref{mapone}) and (\ref{maptwo}) involve commuting the $\LL^{\langle \dd^{(3)},\dd^{(2)}/2\rangle}$ twist in the second monoidal product past the $\LL^{\chi(\dd^{(1)},\dd^{(1)})/2}$ twist in the definition of $\Ho\left(p^{\zeta}_{\dd^{(1)},*}\ICSt_{W,\dd^{(1)}}^{\zeta}\right)$, introducing the sign $(-1)^{\langle \dd^{(3)},\dd^{(2)}\rangle\chi(\dd^{(1)},\dd^{(1)})}$, as in the definition (\ref{AssocSign}) of the associator natural isomorphism for the monoidal structure $\boxtimes_{\oplus}^{\tw}$.  We deduce
\begin{proposition}
The pair 
\[
\left(\Ho\left(p^{\zeta}_{\mu,*}\ICSt_{W,\mu}^{\zeta}\right),\Ho\left(\ms^{\zeta}_{W,\mu}\right)\right)
\]
is a monoid in $\Dulf(\MMHM(\Msp^{\zeta}_{\mu}))$.
\end{proposition}

\begin{definition}[Relative cohomological Hall algebra]
We denote by $\Ho(\Coha_{W,\mu}^{\zeta})$ the monoid $\left(\Ho\left(p_{\mu,*}^{\zeta}\ICSt_{W,\mu}^{\zeta}\right),\Ho(\ms^{\zeta}_{W,\mu}),\ICS_{\Msp_0^{\zeta\sst}}(\mathbb{Q})\right)$ in $\left(\Dulf(\MMHM(\Msp_{\mu}^{\zeta\sst})),\boxtimes^{\tw}_{\oplus}\right)$.
\end{definition}

Recall from Proposition \ref{cvs} the isomorphisms $\nu_{\mu}\colon\phim{\WW^{\zeta}_{\mu}}\Ho\left(p^{\zeta}_{\mu,*}\ICSt_{0,\mu}^{\zeta}\right)\cong \Ho\left(p^{\zeta}_{\mu,*}\ICSt_{W,\mu}^{\zeta}\right)$.  We will use the following technical lemma in the proof of Theorem \ref{strongPBW}.
\begin{lemma}
\label{cohacom}
The following diagram commutes:
\[
\xymatrix{
\phim{\WW_{\mu}^{\zeta}}\left(\Ho\left(p_{\mu,*}^{\zeta}\ICSt_{0,\mu}^{\zeta}\right)\boxtimes_{\oplus}^{\tw} \Ho\left(p_{\mu,*}^{\zeta}\ICSt_{0,\mu}^{\zeta}\right)\right)\ar[rr]^-{\phim{\WW^{\zeta}_{\mu}}\Ho(\ms^{\zeta}_{0,\mu})}\ar[d]^{(\nu_{\mu}\boxtimes^{\tw}_{\oplus} \nu_{\mu})\circ \TS^{-1}} &&\phim{\WW^{\zeta}_{\mu}}\Ho\left(p^{\zeta}_{\mu,*}\ICSt_{0,\mu}^{\zeta}\right)\ar[d]^{\nu_{\mu}}\\
\Ho\left(p^{\zeta}_{\mu,*}\ICSt_{W,\mu}^{\zeta}\right)\boxtimes^{\tw}_{\oplus} \Ho\left(p^{\zeta}_{\mu,*}\ICSt_{W,\mu}^{\zeta}\right)\ar[rr]^-{\Ho(\ms^{\zeta}_{W,\mu})}&& \Ho\left(p^{\zeta}_{\mu,*}\ICSt_{W,\mu}^{\zeta}\right).
}
\]
\end{lemma}
\begin{proof}
Here $\TS$ is the Thom--Sebastiani isomorphism.  We break the two horizontal arrows into their constituent parts, given by the constituent morphisms of the composition $\Ho(\ms^{\zeta}_{W,\mu})$.  Then the problem reduces to several trivial commutativity statements regarding smaller squares.
\end{proof}
\begin{remark}
\label{absRem}
Applying the functor $\dim_* \tilde{\omega}_{\mu}^{\zeta,!}$ to 
\[
\Ho\left(p^{\zeta}_{\mu,*}\ICSt_{W,\mu}^{\zeta}\right)\in\Dulf(\MMHM(\Msp^{\zeta\sst}_{\mu}))
\]
we obtain an element 
\[
\HO_c\left(\Msp^{\zeta\sst,\SP}_{\mu},\Ho\left(p^{\zeta}_{\mu,!}\ICSt_{W,\mu}^{\zeta}\right)\right)^{\vee}\in\Dulf(\MMHM(\mathbb{N}^{Q_0}))
\]
that is noncanonically isomorphic to the underlying $\mathbb{N}^{Q_0}$-graded monodromic mixed Hodge module of the absolute cohomological Hall algebra $\HO(\Coha^{\zeta,\SP}_{W,\mu})$, a monoid in the monoidal category $(\Dulf(\MMHM(\mathbb{N}^{Q_0})),\boxtimes^{\tw}_+)$, defined by Kontsevich and Soibelman (we recall the definition below).  I.e. we obtain the ``ambient perverse associated graded'' monoid 
\begin{equation}
\label{premdef}
\Gr_{\Pf}(\HO(\Coha_{W,\mu}^{\zeta,\SP}))=\Ho\left(\dim_*\tilde{\omega}^{\zeta,!}_{\mu}\Ho\left(\Coha_{W,\mu}^{\zeta,\SP}\right)\right).
\end{equation}
See Proposition \ref {PCoha} for a more precise statement.  The notation in the left hand side of (\ref{premdef}) will be justified after we introduce the (ambient) perverse filtration in Section \ref{PervSec}.
\end{remark}
\subsection{The absolute cohomological Hall algebra}
We start to complete Remark \ref{absRem} by recalling the definition of the cohomological algebra structure on $\HO(\mathcal{A}^{\zeta,\SP}_{W,\mu})$ from \cite{KS2}.  Extending (\ref{XNdef}, \ref{XXNdef}) we set
\begin{align*}
X^{\zeta\sst}_{\dd'\blacksquare\dd'',N} \colonequals &X_{\dd',\dd''}^{\zeta\sst}\times_{G_{\dd'}\times G_{\dd''}}U_N\\
X^{\zeta\sst}_{\dd'\square\dd'',N} \colonequals &(X^{\zeta\sst}_{\dd'}\times X^{\zeta\sst}_{\dd''})\times _{G_{\dd'}\times G_{\dd''}}U_N.
\end{align*}
We denote by
\begin{align*}
\omega_{\blacksquare,N}\colon &X^{\zeta\sst,\SP}_{\dd'\blacksquare\dd'',N}\rightarrow X^{\zeta\sst}_{\dd'\blacksquare\dd'',N}\\
\omega_{\square,N}\colon &X^{\zeta\sst,\SP}_{\dd'\square\dd'',N}\rightarrow X^{\zeta\sst}_{\dd'\square\dd'',N}\\
\omega_{\circ,N}\colon &X^{\zeta\sst,\SP}_{\dd',\dd'',N}\rightarrow X^{\zeta\sst}_{\dd',\dd'',N}
\end{align*}
the inclusions, and we denote by $\phim{\blacksquare,\square,\circ}$ the vanishing cycle functors induced by $\Tr(W)$ on the targets of these three maps.  Then for $n\in\ZZ$ and $N\in\NN$ there is an isomorphism
\[
\Phi\colon \HO_c^n\left(r_{1,N,!}\omega_{\blacksquare,N}^*\phim{\blacksquare}\QQ_{X^{\zeta\sst}_{\dd'\blacksquare\dd'',N}}\rightarrow\omega_{\square,N}^*r_{1,N,!}\phim{\blacksquare}\QQ_{X^{\zeta\sst}_{\dd'\blacksquare\dd'',N}}\rightarrow \omega_{\square,N}^*\phim{\square}\LL^{v}_{X^{\zeta\sst}_{\dd'\square\dd'',N}}\right)^{\vee}
\]
where $v=\sum_{a\in Q_1}\dd''_{s(a)}\dd'_{t(a)}$, provided by base change and the Verdier dual of the morphism
\[
\QQ_{X^{\zeta\sst}_{\dd'\square\dd'',N}}\rightarrow r_{1,N,*}\QQ_{X^{\zeta\sst}_{\dd'\blacksquare\dd'',N}}.
\]
Similarly, we define the isomorphism
\[
\Psi\colon \HO_c^n\left(r_{2,N,!}\omega_{\blacksquare,N}^*\phim{\blacksquare}\QQ_{X^{\zeta\sst}_{\dd'\blacksquare\dd'',N}}\rightarrow\omega_{\circ,N}^*r_{2,N,!}\phim{\blacksquare}\QQ_{X^{\zeta\sst}_{\dd'\blacksquare\dd'',N}}\rightarrow \omega_{\circ,N}^*\phim{\circ}\LL^{u}_{X^{\zeta\sst}_{\dd',\dd'',N}}\right)^{\vee}
\]
where $u=\sum_{i\in Q_0}\dd'_i\dd''_i$.  Taking $N$ sufficiently large (for each cohomological degree), and composing the appropriate twist of $\Phi$ with $\Psi^{-1}$, we define the isomorphism
\begin{align*}
&\overline{\alpha}^{\zeta}_{\dd',\dd''}\colon \HO_c\left(\Mst_{\dd'}^{\zeta\sst,\SP},\phim{\WWW^{\zeta}_{\dd'}}\ICS_{\Mst^{\zeta\sst}_{\dd'}}(\mathbb{Q})\right)^{\vee}\boxtimes_+\HO_c\left(\Mst_{\dd''}^{\zeta\sst,\SP},\phim{\WWW^{\zeta}_{\dd''}}\ICS_{\Mst^{\zeta\sst}_{\dd''}}(\mathbb{Q})\right)^{\vee}\rightarrow\\&\LL^{-(\dd'',\dd')/2}\otimes\HO_c\left(\Mst_{\dd',\dd''}^{\zeta\sst,\SP},\phim{\WWW^{\zeta}_{\dd',\dd'}}\ICS_{\Mst^{\zeta\sst}_{\dd',\dd''}}(\mathbb{Q})\right)^{\vee}.
\end{align*}
Similarly, with notation as in (\ref{preLim}), applying $\HO_c\left(\omega_N^*\phim{\W_{\dd,N}^{\zeta}}\right)^{\vee}$ to 
\[
\QQ_{X_{\dd,N}^{\zeta\sst}}\rightarrow s^{\zeta}_{\dd',\dd'',N,!}\QQ_{X_{\dd',\dd'',N}^{\zeta\sst}}
\]
and passing to the limit, we obtain the morphism
\[
\overline{\beta}^{\zeta}_{\dd',\dd''}\colon\HO_c\left(\Mst_{\dd',\dd''}^{\zeta\sst,\SP},\phim{\WWW^{\zeta}_{\dd',\dd'}}\ICS_{\Mst^{\zeta\sst}_{\dd',\dd''}}(\mathbb{Q})\right)^{\vee}\rightarrow \LL^{(\dd',\dd'')/2}\otimes \HO_c\left(\Mst_{\dd}^{\zeta\sst,\SP},\phim{\WWW^{\zeta}_{\dd}}\ICS_{\Mst^{\zeta\sst}_{\dd}}(\mathbb{Q})\right)^{\vee}.
\]
We set
\[
\HO(\Coha_{W,\mu}^{\zeta,\SP})\colonequals \HO_c(\Mst^{\zeta\sst,\SP}_{\mu},\ICSt^{\zeta}_{W,\mu})^{\vee}
\]
and 
\[
\HO(\ast_{W,\mu}^{\zeta,\SP})\colonequals \bigoplus_{\dd',\dd''\in\NN^{Q_0}} (\LL^{-(\dd',\dd'')/2}\otimes\overline{\beta}^{\zeta}_{\dd',\dd''})\circ(\LL^{\langle \dd'',\dd'\rangle/2}\otimes\overline{\alpha}^{\zeta}_{\dd',\dd''})
\]
to obtain the ``absolute cohomological Hall algebra'' in the sense of Section \ref{abvsrel}.  This is precisely the CoHA defined and studied by Kontsevich and Soibelman in \cite[Sec.7]{KS2}.

\subsection{The ambient perverse filtration}
\label{PervSec}
For $\ff,\dd\in \NN^{Q_0}$ set 
\[
\gamma(\ff,\dd)=(\dd,\dd)/2-\ff\cdot\dd.
\]
Since $\Msp_{\ff,\dd }^{\zeta}$ is smooth, and $\pi^{\zeta}_{\ff,\dd }\colon \Msp_{\ff,\dd }^{\zeta}\rightarrow \Msp_{\dd}^{\zeta\sst}$ is proper, by the decomposition theorem there is an isomorphism
\[
\pi ^{\zeta}_{\ff,\dd ,!}\LL^{\gamma(\ff,\dd)}_{\Msp_{\ff,\dd}^{\zeta}}\cong\Ho\left(\pi ^{\zeta}_{\ff,\dd ,!}\LL^{\gamma(\ff,\dd)}_{\Msp_{\ff,\dd }^{\zeta}}\right).
\]
This decomposition is not canonical, but the mere existence of such a decomposition implies that the natural map 
\[
\Gamma\colon \pi ^{\zeta}_{\ff,\dd ,!}\LL^{\gamma(\ff,\dd)}_{\Msp_{\ff,\dd }^{\zeta}}\rightarrow \tau_{\geq p}\pi ^{\zeta}_{\ff,\dd ,!}\LL^{\gamma(\ff,\dd)}_{\Msp_{\ff,\dd }^{\zeta}}
\]
admits a right inverse, where $\tau_{\geq p}$ is the truncation functor with respect to the natural t-structure on the category of monodromic mixed Hodge modules.  We deduce that the morphism $\Gamma'$ given by the composition of natural maps
\[
\xymatrix{
\pi ^{\zeta}_{\ff,\dd ,!}\phim{\WW_{\ff,\dd}^{\zeta}}\LL^{\gamma(\ff,\dd)}_{\Msp_{\ff,\dd }^{\zeta}}\ar[r]^-{\cong}&\phim{\WW_{\dd}^{\zeta}}\pi ^{\zeta}_{\ff,\dd ,!}\LL^{\gamma(\ff,\dd)}_{\Msp_{\ff,\dd }^{\zeta}}\ar[dl]_>>>>>{\phim{\WW_{\dd}^{\zeta}}\Gamma}\\
\phim{\WW_{\dd}^{\zeta}}\tau_{\geq p}\pi ^{\zeta}_{\ff,\dd ,!}\LL^{\gamma(\ff,\dd)}_{\Msp_{\ff,\dd }^{\zeta}}\ar[r]^-{\cong}&\tau_{\geq p}\pi ^{\zeta}_{\ff,\dd ,!}\phim{\WW_{\ff,\dd}^{\zeta}}\LL^{\gamma(\ff,\dd)}_{\Msp_{\ff,\dd }^{\zeta}}
}
\]
admits a right inverse $\Sigma$.  With 
\begin{align*}
\tilde{\omega}^{\zeta}_{\dd}\colon &\Msp^{\zeta\sst,\SP}_{\dd}\rightarrow \Msp^{\zeta\sst}_{\dd}\\
\tilde{\omega}^{\zeta}_{\ff,\dd}\colon &\Msp^{\zeta,\SP}_{\ff,\dd}\rightarrow \Msp^{\zeta}_{\ff,\dd}
\end{align*}
the natural inclusions, there is a natural isomorphism given by base change
\[
\Omega\colon \pi^{\zeta,\SP}_{\ff,\dd,!}\tilde{\omega}^{\zeta,*}_{\ff,\dd}\phim{\WW_{\ff,\dd}^{\zeta}}\LL^{\gamma(\ff,\dd)}_{\Msp_{\ff,\dd }^{\zeta}}\rightarrow  \tilde{\omega}^{\zeta,*}_{\dd}\pi ^{\zeta}_{\ff,\dd ,!}\phim{\WW_{\ff,\dd}^{\zeta}}\LL^{\gamma(\ff,\dd)}_{\Msp_{\ff,\dd }^{\zeta}}.
\]
Setting 
\begin{align*}
\Pf^p\left(\HO_c\left(\Msp^{\zeta,\SP}_{\ff,\dd},\phim{\WW^{\zeta}_{\ff,\dd}}\LL^{\gamma(\ff,\dd)}\right)^{\vee}\right)\colonequals &\HO_c\left(\Msp^{\zeta,\SP}_{\ff,\dd},\tilde{\omega}^{\zeta,*}_{\dd}\tau_{\geq {}-p}\pi_{\ff,\dd,!}^{\zeta}\phim{\WW^{\zeta}_{\ff,\dd}}\LL^{\gamma(\ff,\dd)}\right)^{\vee}\\
\cong&\HO\left(\Msp^{\zeta,\SP}_{\ff,\dd},\tilde{\omega}^{\zeta,!}_{\dd}\tau_{\leq p}\pi_{\ff,\dd,*}^{\zeta}\phim{\WW^{\zeta}_{\ff,\dd}}\LL^{-(\dd,\dd)/2}\right)
\end{align*}
we consider the morphism $\iota_p\colonequals\HO_c\left(\tilde{\omega}_{\ff,\dd}^{\zeta,*}\Gamma'\circ \Omega\right)^{\vee}$.  This is a split injection, since it is the dual of a split surjection with right inverse $\HO_c(\Omega^{-1}\circ\tilde{\omega}_{\ff,\dd}^{\zeta,*}\Sigma)$.  We thus arrive at a canonical filtration, indexed by $p\in\mathbb{Z}$, which we call the ambient perverse filtration:
\[
\Pf^p\left(\HO_c\left(\Msp^{\zeta,\SP}_{\ff,\dd},\phim{\WW^{\zeta}_{\ff,\dd}}\LL^{\gamma(\ff,\dd)}\right)^{\vee}\right)\subset \HO_c\left(\Msp^{\zeta,\SP}_{\ff,\dd},\phim{\WW_{\ff,\dd}^{\zeta}}\LL^{\gamma(\ff,\dd)}\right)^{\vee}.
\]
\begin{warning}
\label{PervWarning}
Throughout the paper, we always consider the perverse truncation functors for $\Msp_{\dd}^{\zeta}$, and never the truncation functors for monodromic mixed Hodge modules on $\Msp_{\dd}^{\zeta,\SP}$.  As a result, the associated graded of 
\[
\HO_c\left(\Msp^{\zeta,\SP}_{\ff,\dd},\phim{\WW_{\ff,\dd}^{\zeta}}\LL^{\gamma(\ff,\dd)}\right)^{\vee}\cong \HO\left(\Msp^{\zeta,\SP}_{\dd},\tilde{\omega}^{\zeta,!}_{\dd}\pi_{\ff,\dd,*}^{\zeta}\phim{\WW^{\zeta}_{\ff,\dd}}\LL^{{}-(\dd,\dd)/2}\right)
\]
is in general \textit{not} isomorphic as a graded object to the hypercohomology of the total perverse cohomology of $\tilde{\omega}^{\zeta,!}_{\dd}\pi_{\ff,\dd,*}^{\zeta}\phim{\WW^{\zeta}_{\ff,\dd}}\LL^{{}-(\dd,\dd)/2}$, for which the grading comes from the perverse truncation functors for $\Msp_{\dd}^{\zeta,\SP}$.  For example, let $Q$ be the one loop quiver, let $W=0$, let $\Mst^{\SP}\subset \Mst(Q)$ be the reduced substack, the closed points of which correspond to nilpotent $\mathbb{C}Q$-representations, let $\dd=\ff=1$, and let $\zeta$ be any stability condition on $Q$.   Then there are isomorphisms
\begin{align*}
\Msp^{\zeta}_{\ff,\dd}\cong\Msp^{\zeta\sst}_{\dd}\cong\mathbb{A}^1\\
\Msp^{\zeta,\SP}_{\ff,\dd}\cong\Msp^{\zeta\sst,\SP}_{\dd}\cong \{0\}
\end{align*}
and we may identify the (reduced) nilpotent locus with $0\in\AA^1$, and the map $\pi^{\zeta}_{\ff,\dd}$ with the identity map.  We have $(\dd,\dd)=0$, and so the perverse cohomology of $\pi^{\zeta}_{\ff,\dd}\LL_{\Msp_{\ff,\dd}^{\zeta}}^{-(\dd,\dd)}\cong\QQ_{\AA^1}$ is concentrated in degree 1, and so
\[
\Gr_{\Pf}\left(\HO\left(\Msp^{\zeta,\SP}_{\ff,\dd},\tilde{\omega}^{\zeta,!}_{\dd}\pi_{\ff,\dd,*}^{\zeta}\phim{\WW^{\zeta}_{\ff,\dd}}\LL^{{}-(\dd,\dd)/2}\right)\right)
\]
is concentrated in degree 1.  On the other hand $\tilde{\omega}^{\zeta,!}_{\dd}\mathbb{Q}_{\AA^1}\cong \LL_{\{0\}}$ is concentrated in perverse degree 2.

Furthermore, since $\tilde{\omega}^{\zeta,!}_{\dd}$ is in general not exact, there is no obvious way to use the decomposition theorem to deduce that $\tilde{\omega}^{\zeta,!}_{\dd}\pi_{\ff,\dd,*}^{\zeta}\phim{\WW^{\zeta}_{\ff,\dd}}\LL^{-(\dd,\dd)/2}$ is isomorphic to its total perverse cohomology, for the perverse truncation functors associated to $\Msp_{\dd}^{\zeta,\SP}$.  In fact it is not hard to produce examples in which this statement is false, for instance by modifying the above example by setting $\Mst^{\SP}\subset \Mst(Q)$ to be the open substack of $Q$-representations for which the loop of $Q$ acts invertibly.
\end{warning}
\begin{remark}
The distinction between different perverse filtrations exemplified by Warning \ref{PervWarning} is the reason for our being explicit about picking the \textit{ambient} perverse filtration.  Since this is the only type of filtration we consider, we will omit the word ambient from now on.
\end{remark}
The perverse filtration is well-defined in the limit $\ff\gg0$, as for fixed $\dd\in\mathbb{N}^{Q_0}$, fixed cohomological degree $n$, and for $\ff'>\ff\gg 0$, in the commutative diagram
\[
\xymatrix{
\HO_c^n\left(\Msp_{\ff',\dd}^{\zeta,\SP},\pi ^{\zeta}_{\ff',\dd,!}\LL^{\gamma(\ff',\dd)}_{\Msp_{\ff',\dd}^{\zeta}}\right)\ar[d] &\HO_c^n\left(\Msp_{\ff,\dd }^{\zeta,\SP},\pi ^{\zeta}_{\ff,\dd ,!}\LL^{\gamma(\ff,\dd)}_{\Msp_{\ff,\dd }^{\zeta}}\right)\ar[l]\ar[d]\\
\HO_c^n\left(\Msp_{\ff',\dd}^{\zeta,\SP},\tau_{\geq p}\left(\pi ^{\zeta}_{\ff',\dd,!}\LL^{\gamma(\ff',\dd)}_{\Msp_{\ff',\dd}^{\zeta}}\right)\right)& \HO_c^n\left(\Msp_{\ff,\dd }^{\zeta,\SP},\tau_{\geq p}\left(\pi ^{\zeta}_{\ff,\dd ,!}\LL^{\gamma(\ff,\dd)}_{\Msp_{\ff,\dd }^{\zeta}}\right)\right)\ar[l],
}
\]
the horizontal maps are isomorphisms by the argument of Lemma \ref{appLem}.  

Since the constituent morphisms of $\HO(\ast_{W,\mu}^{\zeta,\SP})$ lift to morphisms of monodromic mixed Hodge modules on $\Msp^{\zeta\sst}_{\mu}$, it follows that $\HO(\ast_{W,\mu}^{\zeta,\SP})$ respects the perverse filtration, and from the definitions we obtain the first statement of the following proposition, while the second is Proposition \ref{bcprop}.
\begin{proposition}
\label{PCoha}
There is a natural isomorphism of monoids in $\Dulf(\MMHM(\mathbb{N}^{Q_0}))$
\begin{equation}\label{AGA}
\Gr_{\Pf}\left(\HO(\Coha_{W,\mu}^{\zeta,\SP}),\HO(\ast_{W,\mu}^{\zeta,\SP}),\HO(\Coha_{W,0}^{\zeta})\right)\cong\Ho\left(\dim_*\tilde{\omega}^{\zeta,!}_{\mu} \left(\Ho(\mathcal{A}^{\zeta}_{W,\mu}),\Ho(\ms^{\zeta}_{W,\mu}),\ICS_{\Msp_0}^{\zeta\sst}(\mathbb{Q})\right)\right).
\end{equation}
Furthermore, forgetting the algebra structure, there is a noncanonical isomorphism of underlying monodromic mixed Hodge modules $\Gr_{\Pf}(\HO(\Coha_{W,\mu}^{\zeta,\SP}))\cong\HO(\Coha_{W,\mu}^{\zeta,\SP})$.
\end{proposition}
The isomorphism of algebras is what justifies the notation of (\ref{premdef}).  
The following technical lemma is what will enable us to use the localised coproduct on $\HO(\Coha^{\zeta,\SP}_{W,\mu})$ to induce a Hopf algebra structure on the associated graded algebra (\ref{AGA}).  It is only a very slight variation of \cite[Prop.1.4.4]{deCat12}, but we include the proof for completeness.
\begin{lemma}
\label{HCM}
Let $V$ be a vector bundle on $\Mst_{\dd}$, and let $\eu(V)\in\HO(\Mst_{\dd},\mathbb{Q})$ be the corresponding equivariant Euler class.  Then 
\[
\eu(V)\cdot \Pf^p(\HO_c(\Mst^{\SP,\zeta\sst}_{\dd},\ICSt_{W,\dd}^{\zeta})^{\vee})\subset \Pf^{p+2\dim(V)}(\HO_c(\Mst^{\SP,\zeta\sst}_{\dd},\ICSt_{W,\dd}^{\zeta})^{\vee}).
\]
\end{lemma}
\begin{proof}
Let $\pr\colon T(V)\rightarrow \Mst^{\zeta\sst}_{\dd}$ be the projection from the total space of $V$ restricted to $\Mst^{\zeta\sst}_{\dd}$, and let $i\colon \Mst^{\zeta\sst}_{\dd}\rightarrow T(V)$ be the inclusion of the zero section.  Let $T_\ff(V)$ be the total space of the bundle $V$ pulled back along the map $\Msp_{\ff,\dd }^{\zeta}\rightarrow\Mst_{\dd}^{\zeta\sst}$ defined by forgetting the framing, and let $\pr_{\ff}$ and $i_{\ff}$ denote the corresponding projections and inclusions.  Let $\WW_{V,\ff}$ be the function induced by $\Tr(W)$ on $T_{\ff}(V)$ and let $\mathcal{N}=\HO_c(\Msp_{\ff,\dd }^{\zeta,\SP},\phim{\WW_{\ff,\dd}^{\zeta}}\QQ_{\Msp_{\ff,\dd}^{\zeta}})^{\vee}$.  Then we have the equality of morphisms
\begin{equation}
\label{euup}
\cdot \eu(V)|_{\mathcal{N}}=\HO\left(\tilde{\omega}_{\ff,\dd}^{\zeta,!}\pr_{\ff,*}\phim{\WW_{V,\ff}}\left(i_{\ff,*}\mathbb{Q}_{\Msp^{\zeta}_{\ff,\dd }}\rightarrow \LL^{-\dim(V)}_{T_{\ff}(V)}\rightarrow  i_{\ff,*}\LL^{-\dim(V)}_{\Msp^{\zeta}_{\ff,\dd }}\right)\right)
\end{equation}
and the action of $\cdot\eu(V)$ on $\HO_c(\Mst^{\zeta\sst,\SP}_{\dd},\ICSt_{W,\dd}^{\zeta})^{\vee}$ is given by letting $\ff\gg 0$ in (\ref{euup}) and twisting by $\LL^{(\dd,\dd)/2}$.  This morphism respects the perverse filtration on $\HO(\Msp^{\zeta,\SP}_{\ff,\dd },\phim{\WW_{\ff,\dd}^{\zeta}}\LL^{(\dd,\dd)/2}_{\Msp^{\zeta}_{\ff,\dd }})$ (with the shift by $2\dim_{\mathbb{C}}(V)$) since 
\[
\pi^{\zeta} _{\ff,\dd ,*}\pr_{\ff,*}\phim{\WW_{V,\ff}}\left(i_{\ff,*}\LL^{(\dd,\dd)/2}_{\Msp^{\zeta}_{\ff,\dd }}\rightarrow i_{\ff,*}\LL^{-\dim(V)+(\dd,\dd)/2}_{\Msp^{\zeta}_{\ff,\dd }}\right)
\]
is a map of complexes of mixed Hodge modules on $\Msp^{\zeta\sst}_{\dd}$.  The result then follows from the definition of the perverse filtration on $\HO_c(\Mst^{\zeta\sst,\SP}_{\dd},\ICSt_{W,\dd}^{\zeta})^{\vee}$.
\end{proof}
We finish this section with the ``absolute'' version of Corollary \ref{inccor}.
\begin{corollary}
\label{Inccor}
There is a canonical inclusion 
\begin{equation}
\label{caninc2}
\Upsilon_{W,\mu}^{\zeta,\SP,\ab}\colon\HO(\BC)_{\vir}\otimes \DT^{\zeta,\SP}_{W,\mu}\rightarrow \HO\left(\mathcal{A}^{\zeta,\SP}_{W,\mu}\right)
\end{equation}
\end{corollary}
\begin{proof}
There is a natural inclusion $\LL^{1/2}\otimes\DT^{\zeta,\SP}_{W,\mu}\rightarrow \HO\left(p_{\mu,*}^{\zeta}\ICSt_{W,\mu}^{\zeta}\right)$ since $\DT^{\zeta,\SP}_{W,\mu}$ is precisely the first piece of the perverse filtration of the target.  Via the $\HO(\BC)$-action on $\HO\left(\mathcal{A}^{\zeta,\SP}_{W,\mu}\right)$, this gives rise to a map $\Upsilon_{W,\mu}^{\zeta,\SP,\ab}$, which respects the perverse filtration by Lemma \ref{HCM}.  After passing to the associated graded of the perverse filtration, it is an injection, since it is obtained by restricting the (split) injection of Corollary \ref{inccor} to the locus of points in $\SP$.  Therefore $\Upsilon_{W,\mu}^{\zeta,\SP,\ab}$ is an injection.
\end{proof}

\section{Structural results for cohomological Hall algebras}
\label{PBWsec}
\subsection{Symmetric monoidal structures, revisited}
\label{lc_signs}
Let $\zeta\in \mathbb{H}_+^{Q_0}$ be a stability condition, and let $\mu\in(-\infty,\infty)$ be a slope.  We assume that $\zeta$ is $\mu$-generic.  We define $X^{\zeta\sst}_{\dd,N}$ and $X^{\zeta\sst}_{\dd',\dd'',N}$ as in (\ref{XNdef}) and (\ref{XXNdef}) respectively.  As throughout the paper, we assume that we are given a Serre subcategory of the category of finite-dimensional $\mathbb{C}Q$-modules, and denote by
\begin{align*}
\omega_{\dd,N}\colon &X^{\zeta\sst,\SP}_{\dd,N}\rightarrow X_{\dd,N}^{\zeta\sst}&
\omega_{\dd',\dd'',N}\colon & X^{\zeta\sst,\SP}_{\dd',\dd'',N}\rightarrow X_{\dd',\dd'',N}^{\zeta\sst}
\end{align*}
the inclusions of the reduced loci corresponding to $\mathbb{C}Q$-modules in $\SP$.  Consider again the algebra from the end of Section \ref{relc}
\[
\HO\left(\mathcal{A}_{W,\mu}^{\zeta,\SP}\right)=\HO_c\left(\Mst_{\mu}^{\zeta\sst,\SP},\ICSt_{W,\mu}^{\zeta}\right)^{\vee}\cong\bigoplus_{\dd\in\Lambda_{\mu}^{\zeta}}\lim_{N\mapsto\infty}\Ho\left(\Dim_*\omega^!_{\dd,N}\phim{\W_{\dd,N}^{\zeta}}\LL^{(\dd,\dd)/2}_{X_{\dd,N}^{\zeta\sst}}\right).
\]
By \cite[Thm.5.11]{Da13} this algebra carries a localised bialgebra structure, in the sense that for all decompositions $\dd=\dd'+\dd''$, with $\dd',\dd''\in\Lambda_{\mu}^{\zeta}$, there are maps
\begin{align}
\label{coprodint}
\HO\left(\mathcal{A}_{W,\dd}^{\zeta,\SP}\right)\rightarrow \HO\left(\mathcal{A}_{W,\dd'}^{\zeta,\SP}\right)\tilde{\boxtimes}_+^{\tw} \HO\left(\mathcal{A}_{W,\dd''}^{\zeta,\SP}\right)
\end{align}
where 
\begin{align*}
\HO\left(\mathcal{A}_{W,\dd'}^{\zeta,\SP}\right)\tilde{\boxtimes}_+^{\tw} \HO\left(\mathcal{A}_{W,\dd''}^{\zeta,\SP}\right)\colonequals &\left(\HO\left(\mathcal{A}_{W,\dd'}^{\zeta,\SP}\right)\boxtimes_+^{\tw} \HO\left(\mathcal{A}_{W,\dd''}^{\zeta,\SP}\right)\right)_{(\dd,\dd'')}\\
\colonequals &\left(\HO\left(\mathcal{A}_{W,\dd'}^{\zeta,\SP}\right)\boxtimes_+^{\tw} \HO\left(\mathcal{A}_{W,\dd''}^{\zeta,\SP}\right)\right)[\mathfrak{L}_{\dd',\dd''}^{-1}]\\
\mathfrak{L}_{\dd',\dd''}\colonequals &\prod_{i,j\in Q_0}\prod_{1\leq t'\leq \dd'_i}\prod_{1\leq t''\leq \dd''_j}(y_{j,t''}-x_{i,t'})\\
\in &\HO_{G_{\dd'}}(\pt)\otimes \HO_{G_{\dd''}}(\pt)\\
=&\mathbb{C}[x_{1,1},\ldots,x_{1,\dd'_1},\ldots,x_{n,1},\ldots,x_{n,\dd'_n}]^{\Sym_{\dd'}}\\&\otimes \mathbb{C}[y_{1,1},\ldots,y_{1,\dd''_1},\ldots,y_{n,1},\ldots,y_{n,\dd''_n}]^{\Sym_{\dd''}}
\end{align*}
with $Q_0=\{1,\ldots,n\}$, satisfying some natural compatibility conditions with the multiplication, see \cite[Def.5.3]{Da13}.  We next describe the most important of these conditions, which is a localised variant of the condition that the multiplication and comultiplication together endow $\HO\left(\mathcal{A}_{W,\mu}^{\zeta,\SP}\right)$ with a bialgebra structure.

Let $\dd^1,\dd^2,\dd^3,\dd^4\in\Lambda_{\mu}^{\zeta}$.  Set 
\begin{align*}
\dd=&\dd^1+\dd^2+\dd^3+\dd^4\\
\cc^1=&\dd^1+\dd^2\\
\cc^2=&\dd^3+\dd^4\\
\ee^1=&\dd^1+\dd^3\\
\ee^2=&\dd^2+\dd^4.
\end{align*}
Let $\mathbb{I}_{\bullet}=\HO\left(\Mst_{\bullet}^{\zeta\sst,\SP},\phim{\WWW^{\zeta}_{\bullet}}\mathbb{Q}_{\Mst_{\bullet}^{\zeta\sst,\SP}}\right)$.  I.e. for now we do not impose our usual system of (half) Tate twists.  By the main result of \cite{Da13} the following diagram commutes
\[
\xymatrix{
\mathbb{I}_{\cc^1}\boxtimes_+ \mathbb{I}_{\cc^2}\ar[d]^-{b}\ar[r]^-a&\left(\mathbb{I}_{\dd^1}\boxtimes_+  \mathbb{I}_{\dd^2}\boxtimes_+  \mathbb{I}_{\dd^3}\boxtimes_+  \mathbb{I}_{\dd^4}\right)_{(\dd^1,\dd^2),(\dd^3,\dd^4)}\ar[d]^-c\\
\mathbb{I}_{\dd}\ar[r]^-d& \left(\mathbb{I}_{\ee^1}\boxtimes_+ \mathbb{I}_{\ee^2}\right)_{(\ee^1,\ee^2)}
}
\]
where we have left out the (whole) Tate twists, and the subscripts indicate which Euler class we localise with respect to.  Here, $a$ and $d$ are given by the comultiplication, while $b$ is given by multiplication.  The morphism $c$ is given by first further localising, then swapping the second and third entries (observing the Koszul sign rule), multiplying by the scalar $(-1)^{\chi(\dd^2,\dd^3)}$, then the CoHA multiplication.  The reader is encouraged to consult \cite{Da13} for the full details; our description in this section is mainly directed towards the limited aim of explaining required sign conventions, as we next explain.

We would like to have, instead, commutativity of the diagram
\begin{align}
\label{desComm}
\xymatrix{
\HO\left(\mathcal{A}_{W,\cc^1}^{\zeta,\SP}\right)\boxtimes_+\HO\left(\mathcal{A}_{W,\cc^2}^{\zeta,\SP}\right)\ar[r]^-{a'}\ar[d]^{b'}&\left(\HO\left(\mathcal{A}_{W,\dd^1}^{\zeta,\SP}\right)\boxtimes_+ \HO\left(\mathcal{A}_{W,\dd^2}^{\zeta,\SP}\right)\boxtimes_+ \HO\left(\mathcal{A}_{W,\dd^3}^{\zeta,\SP}\right)\boxtimes_+ \HO\left(\mathcal{A}_{W,\dd^4}^{\zeta,\SP}\right)\right)_{\heartsuit}\ar[d]^{c'}\\
\HO\left(\mathcal{A}_{W,\dd}^{\zeta,\SP}\right)\ar[r]^-{d'}&\left(\HO\left(\mathcal{A}_{W,\ee^1}^{\zeta,\SP}\right)\boxtimes \HO\left(\mathcal{A}_{W,\ee^2}^{\zeta,\SP}\right)\right)_{\spadesuit}
}
\end{align}
where now the map $c'$ is defined to be the map swapping the second and third terms via the symmetric monoidal structure in $\Dulf(\MMHM(\Lambda_{\mu}^{\zeta}))$, and then applying the multiplication.  Here,
\begin{align*}
\heartsuit=&(\dd^1,\dd^2),(\dd^3,\dd^4)\\
\spadesuit=&(\ee^1,\ee^2)
\end{align*}
as before.  The relation between the two diagrams is as follows.  The objects in diagram (\ref{desComm}) are tensor products of terms $\mathbb{I}_{\bullet}$ tensored on the left with half Tate twists, which we may commute to the left.  Then tensoring the morphisms $a,b,c,d$ on the left by an overall half Tate twist and commuting the required half Tate twists into the tensor product, we obtain the maps $a',b',c',d'$ respectively.  In order to arrange commutativity of (\ref{desComm}), then, we define the sign occurring in the symmetrizing morphism for $\Dulf(\MMHM(\Lambda_{\mu}^{\zeta}))$ to be given by
\[
{}^{\tau}\mathbf{sw}(\alpha\otimes\beta)=(-1)^{\chi(\dd^2,\dd^2)\chi(\dd^3,\dd^3)+\chi(\dd^2,\dd^3)}\mathbf{sw}(\alpha\otimes\beta)
\]
where $\mathbf{sw}$ is the standard symmetrizing morphism.  We pick the superscript $\tau$ since the bilinear form (\ref{half_tau}) defines the deviation from the usual sign convention.  The $\chi(\dd^2,\dd^3)$ term is as in the definition of the morphism $c$.  We denote by ${}^{\tau}\Dulf(\MMHM(\Lambda_{\mu}^{\zeta}))$ the symmetric monoidal category obtained via this modified sign convention.
\begin{proposition}
Let $(\mathcal{B},\ast)$ be a free supercommutative algebra in the category $\Dub(\MMHM(\Lambda_{\mu}^{\zeta}))$, given the above symmetric monoidal structure.  Then if we define ${}^{\psi}\!\mathcal{B}$ by modifying the product as in (\ref{psitwdef}), it is a free supercommutative algebra in $\Dub(\MMHM(\Lambda_{\mu}^{\zeta}))$ with its usual symmetric monoidal structure.
\end{proposition}

Explicitly, we define ${}^{\tau}\mathbf{sw}$ via the chain of isomorphisms
\begin{align}
\label{funct_tw}
\bigoplus_{\dd'\in\Lambda_{\mu}^{\zeta}}\mathcal{F}'_{\dd'}\boxtimes_+ \bigoplus_{\dd''\in\Lambda_{\mu}^{\zeta}}\mathcal{F}''_{\dd''}\cong&\bigoplus_{\dd',\dd''\in\Lambda_{\mu}^{\zeta}} \LL^{\left(\chi(\dd',\dd')+\chi(\dd'',\dd'')\right)/2}\otimes \left(\LL^{-\chi(\dd',\dd')/2}\otimes\mathcal{F}'_{\dd'}\right)\boxtimes_+\left(\LL^{-\chi(\dd'',\dd'')/2}\otimes\mathcal{F}''_{\dd''}\right)
\\ \cong&\bigoplus_{\dd',\dd''\in\Lambda_{\mu}^{\zeta}} \LL^{\left(\chi(\dd',\dd')+\chi(\dd'',\dd'')\right)/2}\otimes \left(\LL^{-\chi(\dd'',\dd'')/2}\otimes\mathcal{F}''_{\dd''}\right)\boxtimes_+\left(\LL^{-\chi(\dd',\dd')/2}\otimes\mathcal{F}'_{\dd'}\right)\nonumber 
\\ \cong &\bigoplus_{\dd''\in\Lambda_{\mu}^{\zeta}}\mathcal{F}''_{\dd''}\boxtimes_+ \bigoplus_{\dd'\in\Lambda_{\mu}^{\zeta}}\mathcal{F}'_{\dd'}\nonumber
\end{align}
where we have applied the usual symmetrizing natural isomorphism for the tensor product of the objects $\LL^{-\chi(\dd',\dd')/2}\otimes\mathcal{F}'_{\dd'}$ and $\LL^{-\chi(\dd'',\dd'')/2}\otimes\mathcal{F}''_{\dd''}$.  We define the twisted symmetric monoidal structure on $\Dulf(\MMHM(\Msp_{\mu}^{\zeta\sst}))$ via the same recipe as (\ref{funct_tw}), with $\boxtimes_+$ replaced with $\boxtimes_{\oplus}$.  It then follows as before from \cite{MSS11} that 
\[
\dim_*\colon{}^{\tau}\Dulf(\MMHM(\Msp_{\mu}^{\zeta\sst}))\rightarrow {}^{\tau}\Dulf(\MMHM(\Lambda_{\mu}^{\zeta}))
\]
is a symmetric monoidal functor.

The algebras $\HO\left(\mathcal{A}_{W}^{\zeta,\SP}\right)$ and $\HO(^{\psi}\!\!\mathcal{A}_{W,\mu}^{\zeta,\SP})$ differ by the above type of sign twist.  In particular, their underlying graded monodromic mixed Hodge structures are the same, and Theorem \ref{qea} implies that there is an isomorphism in $\Dulf(\MMHM(\Lambda_{\mu}^{\zeta}))$:
\[
{}^{\tau}\!\Sym_{\boxtimes_+}\left(\HO(\BC)_{\vir}\otimes \DT_{W,\mu}^{\zeta,\SP}\right)\cong{}\Sym_{\boxtimes_+}\left(\HO(\BC)_{\vir}\otimes \DT_{W,\mu}^{\zeta,\SP}\right).
\]
The next proposition shows that this is a general phenomenon.
\begin{proposition}
\label{underlying_same}
Let $\mathcal{F}\in\Dulf(\MMHM(\Msp_{\mu}^{\zeta\sst}\setminus\Msp_0^{\zeta\sst}))$.  Then there is an isomorphism
\[
{}^{\tau}\Sym_{\boxtimes_+}(\mathcal{F})\cong\Sym_{\boxtimes_+}(\mathcal{F}).
\]
\end{proposition}
\begin{proof}
Let $\mathcal{F}$ be represented by the complex $(\mathcal{G},d)$ of mixed Hodge modules on $\Msp_{\mu}^{\zeta\sst}\times\mathbb{A}^1$.  Let $\psi$ be as in (\ref{psitwdef}).  Let 
\begin{align*}
T(\mathcal{G},d)\colonequals &\left(\bigoplus_{n\geq 0}T_n(\mathcal{G}),d\right)\\
T_n(\mathcal{G})\colonequals &\bigoplus_{n\geq 0}\underbrace{\mathcal{G}\boxtimes\ldots\boxtimes \mathcal{G}}_{n \;\textrm{times}}
\end{align*}
be the algebra generated by $\mathcal{G}$ in the category 
\[
\Dulf\left(\MHM\left(\coprod_{n\geq 0}\left((\Msp_{\mu}^{\zeta\sst}\setminus\Msp_0^{\zeta\sst})\times\mathbb{A}^1\right)^{\times n}\right)\right)
\]
along with its natural differential.  Via the results of \cite{MSS11}, as recalled in Section \ref{prodss}, $T_n(\mathcal{G},d)$ carries an action of $\SSym_n$, determined by morphisms for each $\pi\in\SSym_n$ and $a_1,\ldots,a_n\in\Lambda_{\mu}^{\zeta}$
\[
\mathcal{G}_{a_1}\boxtimes\ldots\boxtimes\mathcal{G}_{a_n}\xrightarrow{\pi^{\#}_{a_1,\ldots,a_n}}\pi_*\left(\mathcal{G}_{\pi(a_1)}\boxtimes\ldots\boxtimes\mathcal{G}_{\pi(a_n)}\right).
\]
Setting
\[
\pi^{\#,\tau}_{a_1,\ldots,a_n}\colonequals \sum_{1\leq i<j\leq n|\pi(i)>\pi(j)}(-1)^{\tau(a_i,a_j)}\pi^{\#}_{a_1,\ldots,a_n}
\]
we obtain the $\tau$-twisted action.  Then the morphisms
\[
\mathcal{G}_{a_1}\boxtimes\ldots\boxtimes\mathcal{G}_{a_n}\xrightarrow{\cdot\sum_{1\leq i<j\leq n|\pi(i)>\pi(j)}(-1)^{\psi(a_i,a_j)}}\mathcal{G}_{a_1}\boxtimes\ldots\boxtimes\mathcal{G}_{a_n}
\]
commute with the differential, and intertwine the two $\SSym_n$-actions.  Passing to the invariant subcomplexes, and taking the direct image along the iterated monoid morphism
\[
\coprod_{n\geq 0}\left((\Msp_{\mu}^{\zeta\sst}\setminus\Msp_0^{\zeta\sst})\times\mathbb{A}^1\right)^{\times n}\rightarrow \Msp_{\mu}^{\zeta\sst}\times\mathbb{A}^1
\]
we deduce the result.
\end{proof}

\subsection{The perverse associated graded Hopf algebra}
Next we describe the localised comultiplication in a little more detail, for further details see \cite{Da13}.  Consider the proper map $s_{\dd',\dd''}\colon \Mst_{\dd',\dd''}^{\zeta\sst}\rightarrow\Mst_{\dd}^{\zeta\sst}$ defined as in (\ref{smap}).  We form the map 
\begin{equation}
\label{gammadef}
\gamma\colon \HO_c(\Mst^{\zeta\sst,\SP}_{\dd},\ICSt_{W,\dd}^{\zeta})^{\vee}\rightarrow \LL^{(\dd',\dd'')/2}\otimes\HO_c(\Mst^{\zeta\sst,\SP}_{\dd',\dd''},\phim{\WWW_{\dd',\dd''}^{\zeta}}\ICS_{\Mst_{\dd',\dd''}^{\zeta\sst}}(\mathbb{Q}))^{\vee}
\end{equation}
given, via Verdier duality and Proposition \ref{basicfacts}(2), as the tensor product with $\LL^{(\dd,\dd)/2}$ of the limit as $N\mapsto\infty$ of the composition of maps
\begin{align}
\label{sagain}
&\Ho\left(\Dim_*\omega^!_{\dd,N}\phim{\W_{\dd,N}^{\zeta}}\QQ_{X_{\dd,N}^{\zeta\sst}}\right)\rightarrow \Ho\left(\Dim_*\omega_{\dd,N}^!\phim{\W^{\zeta}_{\dd,N}}s_{\dd',\dd'',N,*}\QQ_{X^{\zeta\sst}_{\dd',\dd'',N}}\right)\xrightarrow{\cong}\\&\rightarrow \Ho\left(\Dim_*\omega_{\dd,N}^!s_{\dd',\dd'',N,*}\phim{\W^{\zeta}_{\dd',\dd'',N}}\QQ_{X^{\zeta\sst}_{\dd',\dd'',N}}\right)\xrightarrow{\cong}\Ho\left(\Dim_*s_{\dd',\dd'',N,*}\omega_{\dd',\dd'',N}^!\phim{\W^{\zeta}_{\dd',\dd'',N}}\QQ_{X^{\zeta\sst}_{\dd',\dd'',N}}\right).\nonumber
\end{align}
Since the morphism $\Dim$ factors through a map to the base $\Msp^{\zeta\sst}_{\dd}$, we deduce as in Section \ref{PervSec} that $\gamma$ preserves the perverse filtration, i.e.
\[
\gamma\left(\Pf_s\left(\HO_c(\Mst_{\dd}^{\zeta\sst,\SP},\ICSt_{W,\dd}^{\zeta})^{\vee}\right)\right)\subset \Pf_{s}\left(\LL^{(\dd',\dd'')/2}\otimes\HO_c\left(\Mst^{\zeta\sst,\SP}_{\dd',\dd''},\phim{\WWW^{\zeta}_{\dd',\dd''}}\ICS_{\Mst_{\dd',\dd''}^{\zeta\sst}}(\mathbb{Q})\right)^{\vee}\right).
\]
Composing with the isomorphism
\begin{align*}
\Gamma\colon &\LL^{(\dd',\dd'')/2}\otimes \HO_c\left(\Mst^{\zeta\sst,\SP}_{\dd',\dd''},\phim{\WWW^{\zeta}_{\dd',\dd''}}\ICS_{\Mst_{\dd',\dd''}^{\zeta\sst,\SP}}(\mathbb{Q})\right)^{\vee}\rightarrow \\&\LL^{(\dd',\dd'')/2+(\dd'',\dd')/2}\otimes \HO_c(\Mst^{\zeta\sst}_{\dd'},\ICSt_{W,\dd'}^{\zeta})^{\vee}\otimes\HO_c(\Mst^{\zeta\sst,\SP}_{\dd''},\ICSt_{W,\dd''}^{\zeta})^{\vee}
\end{align*}
we likewise deduce
\begin{align}
\nonumber
&\Gamma\circ\gamma\left(\Pf_s\left(\HO_c(\Mst_{\dd}^{\zeta\sst,\SP},\ICSt_{W,\dd}^{\zeta})^{\vee}\right)\right)\subset
\\\nonumber& \Pf_s\left(\LL^{(\dd',\dd'')/2+(\dd'',\dd')/2}\otimes\HO(\Mst^{\zeta\sst,\SP}_{\dd'},\ICSt_{W,\dd'}^{\zeta})\otimes\HO_c(\Mst^{\zeta\sst,\SP}_{\dd''},\ICSt_{W,\dd''}^{\zeta})^{\vee}\right)\\\nonumber
& =\Pf_s\left(\LL^{(\dd',\dd'')}\otimes \HO_c(\Mst^{\zeta\sst,\SP}_{\dd'},\ICSt_{W,\dd'}^{\zeta})^{\vee}\boxtimes_+^{\tw}\HO_c(\Mst^{\zeta\sst,\SP}_{\dd''},\ICSt_{W,\dd''}^{\zeta})^{\vee}\right)\\\label{downshift}
& =\Pf_{s-2(\dd',\dd'')}\left(\HO_c(\Mst^{\zeta\sst,\SP}_{\dd'},\ICSt_{W,\dd'}^{\zeta})^{\vee}\boxtimes_+^{\tw}\HO_c(\Mst^{\zeta\sst,\SP}_{\dd''},\ICSt_{W,\dd''}^{\zeta})^{\vee}\right),
\end{align}
since $\Gamma$ is obtained by taking the hypercohomology of a morphism of complexes of monodromic mixed Hodge modules on $\Msp_{\mu}^{\zeta\sst}$.  The final equality (\ref{downshift}) does not preserve the cohomological degree, or the mixed Hodge structure.  We define 
\[
\mathfrak{E}_{1,\dd',\dd''}\colonequals \prod_{a\in Q_1}\prod_{1\leq l'\leq \dd'_{s(a)}}\prod_{1\leq l''\leq \dd''_{t(a)}}(x_{s(a),l'}-y_{t(a),l''})
\]
and
\[
\mathfrak{E}_{0,\dd',\dd''}\colonequals \prod_{i\in Q_0}\prod_{1\leq l'\leq \dd'_i}\prod_{1\leq l''\leq \dd''_i}(x_{i,l'}-y_{i,l''}).
\]

Multiplication by $\mathfrak{E}^{-1}_{1,\dd',\dd''}\mathfrak{E}_{0,\dd',\dd''}$ defines a map
\begin{align}
\beta^*\colon &\HO(\mathcal{A}^{\zeta,\SP}_{W,\dd'})\boxtimes_+^{\tw}\HO(\mathcal{A}^{\zeta,\SP}_{W,\dd''})\rightarrow   \left(\HO(\mathcal{A}^{\zeta,\SP}_{W,\dd'})\boxtimes_+^{\tw}\HO(\mathcal{A}^{\zeta,\SP}_{W,\dd''})\right)[\mathfrak{L}^{-1}_{\dd',\dd''}].
\end{align}
The comultiplication is defined \cite[Cor.5.9]{Da13} by
\begin{align*} \Delta_{W,\dd',\dd''}^{\zeta,\SP}\colon &\HO(\mathcal{A}_{W,\dd}^{\zeta,\SP})\rightarrow \left(\HO(\mathcal{A}_{W,\dd'}^{\zeta,\SP})\boxtimes^{\tw}_+ \HO(\mathcal{A}_{W,\dd''}^{\zeta,\SP})\right)[\mathfrak{L}^{-1}_{\dd',\dd''}]
\\ 
\Delta_{W,\dd',\dd''}^{\zeta,\SP}\colonequals &(\cdot\mathfrak{E}^{-1}_{1,\dd',\dd''}\mathfrak{E}_{0,\dd',\dd''})\circ \Gamma\circ\gamma.  
\end{align*}
This comultiplication then respects cohomological degree, and lifts to a morphism in the category of cohomologically graded monodromic mixed Hodge structures.  For a dimension vector $\ee$, we define $\lvert \ee\rvert=\sum_{i\in Q_0}\ee_i$.  So, for example, $\mathfrak{L}_{\dd',\dd''}$ is of cohomological degree $2\lvert\dd'\rvert \lvert \dd''\rvert$.  We extend the perverse filtration to the right hand side of (\ref{coprodint}) by setting 
\begin{align}\label{pervExt}
&\Pf_{s}\left(\left(\HO(\mathcal{A}_{W,\dd'}^{\zeta,\SP})\boxtimes^{\tw}_+ \HO(\mathcal{A}_{W,\dd''}^{\zeta,\SP})\right)[\mathfrak{L}^{-1}_{\dd',\dd''}]\right)\colonequals\\&\sum_{n\geq 0}\Pf_{s+2\lvert\dd''\rvert\lvert\dd'\rvert n}\left(\left(\HO(\mathcal{A}_{W,\dd'}^{\zeta,\SP})\boxtimes^{\tw}_+ \HO(\mathcal{A}_{W,\dd''}^{\zeta,\SP})\right)\right)\cdot \mathfrak{L}_{\dd',\dd''}^{-n}.\nonumber
\end{align}
By \cite[Prop.4.1]{Da13}, the localisation map 
\begin{align}\label{fDef}
\left(\HO(\mathcal{A}_{W,\dd'}^{\zeta,\SP})\boxtimes^{\tw}_+ \HO(\mathcal{A}_{W,\dd''}^{\zeta,\SP})\right)\xrightarrow{f_{\dd',\dd''}} \left(\HO(\mathcal{A}_{W,\dd'}^{\zeta,\SP})\boxtimes^{\tw}_+ \HO(\mathcal{A}_{W,\dd''}^{\zeta,\SP})\right)[\mathfrak{L}^{-1}_{\dd',\dd''}]
\end{align}
is an inclusion, since the operation $\cdot\mathfrak{L}_{\dd',\dd''}$ is injective on the domain.  By definition (\ref{pervExt}), $f_{\dd',\dd''}$ preserves the perverse filtration.  By Lemma \ref{HCM}\, multiplication by $\mathfrak{E}_{1,\dd',\dd''}^{-1}\cdot\mathfrak{E}_{0,\dd',\dd''}$ induces a map
\begin{align}
\nonumber&\Pf_s\left(\left(\HO(\mathcal{A}_{W,\dd'}^{\zeta,\SP})\boxtimes^{\tw}_+ \HO(\mathcal{A}_{W,\dd''}^{\zeta,\SP})\right)[\mathfrak{L}^{-1}_{\dd',\dd''}]\right)\rightarrow \\ \label{upshift}
&\Pf_{s+2(\dd',\dd'')}\left(\left(\HO(\mathcal{A}_{W,\dd'}^{\zeta,\SP})\boxtimes^{\tw}_+ \HO(\mathcal{A}_{W,\dd''}^{\zeta,\SP})\right)[\mathfrak{L}^{-1}_{\dd',\dd''}]\right).
\end{align}
Combining the shifts in perverse degree in (\ref{upshift}) and (\ref{downshift}), we deduce that
\[
\Delta^{\zeta}_{W,\dd',\dd''}\left(\Pf_s \left(\HO\left(\mathcal{A}^{\zeta,\SP}_{W,\dd}\right)\right)\right)\subset \Pf_s\left(\left(\HO(\mathcal{A}_{W,\dd'}^{\zeta,\SP})\boxtimes^{\tw}_+ \HO(\mathcal{A}_{W,\dd''}^{\zeta,\SP})\right)[\mathfrak{L}^{-1}_{\dd',\dd''}]\right).
\]
The proof of the following is essentially the original proof of the Atiyah--Bott lemma of \cite{AtBo83}, adapted to deal with vanishing cycles and passage to the associated graded of the perverse filtration.
\begin{lemma}
\label{PAB}
The map $\Gr_{\Pf}(f_{\dd',\dd''})$ induced by the map (\ref{fDef}) is injective.
\end{lemma}
\begin{proof}
We abbreviate $\Pf_i=\Pf_i\left(\HO(\Mst^{\zeta\sst,\SP}_{\dd'},\ICSt_{W,\dd'}^{\zeta})\boxtimes_+^{\tw} \HO(\Mst^{\zeta\sst,\SP}_{\dd''},\ICSt_{W,\dd''}^{\zeta})\right)$.  It is sufficient to show that the degree $2\lvert\dd'\rvert \lvert\dd''\rvert$ map $\Gr_{\Pf}(\cdot \mathfrak{L}_{\dd',\dd''})$ is injective.  For then if 
\[
f_{\dd',\dd''}(a)\in \left(\sum_{r\geq 0}\Pf_{i+2\lvert\dd'\rvert \lvert\dd''\rvert r}\left(\HO(\Mst^{\zeta\sst,\SP}_{\dd'},\ICSt_{W,\dd'}^{\zeta})\boxtimes_+^{\tw} \HO(\Mst^{\zeta\sst,\SP}_{\dd''},\ICSt_{W,\dd''}^{\zeta})\right)\cdot \mathfrak{L}_{\dd',\dd''}^{-r}\right)
\]
we can write $a=\sum_{r\leq n} b_r\cdot\mathfrak{L}^{-r}_{\dd',\dd''}$ for $b_r\in \Pf_{i+2\lvert\dd'\rvert \lvert\dd''\rvert r}$.  Then $a\cdot \mathfrak{L}^n_{\dd',\dd''}\in\Pf_{i+2\lvert\dd'\rvert \lvert\dd''\rvert n}$ by Lemma \ref{HCM}, and so $a\in \Pf_{i}$, by injectivity of $\Gr_{\Pf}(\cdot \mathfrak{L}_{\dd',\dd''})$.

Consider the subgroup $\GG_m\cong T\subset G_{\dd'}\times G_{\dd''}$ given by the embedding 
\[
z\mapsto \left( (z^{\dd''_i}\id_{\mathbb{C}^{\dd'_i}})_{i\in Q_0},(z^{-\dd'_i}\id_{\mathbb{C}^{\dd''_i}})_{i\in Q_0}\right)
\]
and let $P_{\dd',\dd''}\colonequals (G_{\dd'}\times G_{\dd''})/T$.  Note that $T$ acts trivially on $X^{\zeta\sst}_{\dd'}\times X^{\zeta\sst}_{\dd''}$ and furthermore $T$ acts trivially on the linearization (\ref{charDef}), so that the linearization of the $G_{\dd'}\times G_{\dd''}$-action lifts to a linearization of the $P_{\dd',\dd''}$ action.  Let 
\[
p'_{\dd',\dd''}\colon (X^{\zeta\sst}_{\dd'}\times X^{\zeta\sst}_{\dd''})/P_{\dd',\dd''}\rightarrow \Msp_{\dd'}^{\zeta\sst}\times \Msp_{\dd''}^{\zeta\sst}
\]
be the map to the GIT quotient.  Via the Cartesian square
\[
\xymatrix{
\Mst_{\dd'}^{\zeta\sst}\times \Mst_{\dd''}^{\zeta\sst}\ar[r]\ar[d]&\pt/(G_{\dd'}\times G_{\dd''})\ar[d]\\
(X^{\zeta\sst}_{\dd'}\times X^{\zeta\sst}_{\dd''})/P_{\dd',\dd''}\ar[r]&\pt/P_{\dd',\dd''}
}
\]
we obtain the isomorphism 
\begin{align*}
&\Ho\left(p^{\zeta}_{\dd',*}\ICSt_{W,\dd'}^{\zeta}\boxtimes p^{\zeta}_{\dd'',*}\ICSt_{W,\dd''}^{\zeta}\right)\cong\\& \Ho\left(p'_{\dd',\dd'',*}\phim{f}\ICS_{X^{\zeta\sst}_{\dd'}\times X^{\zeta\sst}_{\dd''}/P_{\dd',\dd''}}(\mathbb{Q})\right)\otimes_{\HO_{P_{\dd',\dd''}}(\pt)}\HO_{G_{\dd'}\times G_{\dd''}}(\pt)\otimes \LL^{1/2}.
\end{align*}
where $f$ is the function induced by $\W_{\dd'}\boxplus \W_{\dd''}$ on $(X^{\zeta\sst}_{\dd'}\times X^{\zeta\sst}_{\dd''})/P_{\dd',\dd''}$.  After picking a splitting
\[
\lambda\colon \HO_{G_{\dd'}\times G_{\dd''}}(\pt)\cong \HO_{P_{\dd',\dd''}}(\pt)\otimes \HO_{T}(\pt)
\]
we obtain an isomorphism
\begin{align*}
&\Ho\left(p^{\zeta}_{\dd',*}\ICSt_{W,\dd'}^{\zeta}\boxtimes p^{\zeta}_{\dd'',*}\ICSt_{W,\dd''}^{\zeta}\right)\cong \Ho\left(p'_{\dd',\dd'',*}\phim{f}\ICS_{X^{\zeta\sst}_{\dd'}\times X^{\zeta\sst}_{\dd''}/P_{\dd',\dd''}}(\mathbb{Q})\right)\otimes \HO(\ICS_{\pt/T}(\mathbb{Q})).
\end{align*}
We define the filtration
\begin{align*}
D^i=&\HO\left(\Msp_{\dd'}^{\zeta\sst,\SP}\times \Msp_{\dd''}^{\zeta\sst,\SP},\tilde{\omega}^!\Ho^{\geq i}\left(p'_{\dd',\dd'',*}\phim{f}\ICS_{X^{\zeta\sst}_{\dd'}\times X^{\zeta\sst}_{\dd''}/P_{\dd',\dd''}}(\mathbb{Q})\right)\otimes \HO(\ICS_{\pt/T}(\mathbb{Q})\right)\\&
\subset \HO\left(\Msp_{\dd'}^{\zeta\sst,\SP}\times \Msp_{\dd''}^{\zeta\sst,\SP}, \tilde{\omega}^!\Ho\left(p^{\zeta}_{\dd',*}\ICSt_{W,\dd'}^{\zeta}\boxtimes p^{\zeta}_{\dd'',*}\ICSt_{W,\dd''}^{\zeta}\right)\right)\\&
=\Gr_{\Pf}\left(\HO(\mathcal{A}_{W,\dd'}^{\zeta,\SP})\boxtimes^{\tw}_+ \HO(\mathcal{A}_{W,\dd''}^{\zeta,\SP})\right)=:\mathcal{B}
\end{align*}
where we have made the abbreviation 
\[
\tilde{\omega}=\tilde{\omega}^{\zeta}_{\dd'}\times\tilde{\omega}^{\zeta}_{\dd''}\colon \Msp^{\zeta\sst,\SP}_{\dd'}\times\Msp^{\zeta\sst,\SP}_{\dd''}\hookrightarrow \Msp^{\zeta\sst}_{\dd'}\times\Msp^{\zeta\sst}_{\dd''}.
\]
Consider the associated graded object $\Gr_D\mathcal{B}$.  Multiplication by $\mathfrak{L}_{\dd',\dd''}$ preserves $D$ degree, and so induces a map 
\[
\Gr_{D}(\Gr_{\Pf}(\cdot \mathfrak{L}_{\dd',\dd''}))\colon \Gr_{D}\mathcal{B}\rightarrow \Gr_{D}\mathcal{B}.
\]
Specifically, the map is given by multiplication by the $T$-equivariant Euler characteristic of $X_{\dd',\dd''}\rightarrow X_{\dd'}\times X_{\dd''}$, which is nonzero since the vector bundle has no sub-bundle with trivial $T$-weight.  It follows that $\Gr_{D}(\Gr_{\Pf}(\cdot \mathfrak{L}_{\dd',\dd''}))$ is injective, so that $\Gr_{\Pf}(\cdot \mathfrak{L}_{\dd',\dd''})$ is injective too.

\end{proof}
We deduce the following proposition.
\begin{proposition}
The triple $\left(\Gr_{\Pf}(\HO(\Coha_{W,\mu}^{\zeta,\SP})),\Gr_{\Pf}(\HO(\ms^{\zeta,\SP}_{W,\mu})),\Gr_{\Pf}(\Delta_{W,\mu}^{\zeta,\SP})\right)$ defines a localised bialgebra in the category ${}^{\tau}\Dulf(\MMHM(\Lambda_{\mu}^{\zeta}))$ in the sense of \cite[Def.5.3]{Da13}.
\end{proposition}

\subsection{Proof of Theorem \ref{qea} and construction of the BPS Lie algebra}
\label{ThmCproof}

We start with a couple of technical lemmas for the proofs in this section.
\begin{lemma}
\label{Tech1}
Let $\mathcal{F},\mathcal{G}\in\MMHM(X)$.  Then $f\colon \mathcal{F}\rightarrow \mathcal{G}$ is an isomorphism if and only if $f_x$ is an isomorphism for every closed point $x\in X$.
\end{lemma}
\begin{proof}
One implication is trivial, so we assume that $f_x$ is an isomorphism for every $x$ and prove that $f$ is an isomorphism.  Since the functor 
\[
\form{X}\colon \MMHM(X)\rightarrow \perv(X)
\]
is faithful, it suffices to prove the lemma with the category $\MMHM(X)$ replaced by $\perv(X)$.  Write $C=\textrm{cone}(f)$.  Assume for a contradiction that $C\neq 0$.  Let $i\in\mathbb{Z}$ be the lowest number such that the constructible cohomology sheaf $\Ho_{\textrm{con}}^i(C)$ is nonzero, and set 
\[
\mathcal{R}\colonequals \Ho_{\textrm{con}}^i(C)[-i].  
\]
Denote by $\iota_x\colon x\hookrightarrow X$ the inclusion.  The functor $\iota^*_x$ is exact for the natural t-structure on the derived category of constructible sheaves, and there is an isomorphism $\iota_x^*C\cong \textrm{cone}(f_x)$, so the first term in the exact sequence
\[
\rightarrow \Ho^{i-1}_{\textrm{con}}(\iota_x^*\tau_{>i}C)\rightarrow \Ho^i_{\textrm{con}}(\iota_x^*\mathcal{R})\rightarrow \Ho^i_{\textrm{con}}(\iota^*_xC)\rightarrow.
\]
is zero, as well as the last, by supposition.  Then $\mathcal{R}[i]$ is a nonzero constructible sheaf satisfying $\iota^*_x\mathcal{R}[i]=0$ for all $x$, a contradiction.
\end{proof}
\begin{lemma}
\label{Tech2}
Let $\mathcal{F},\mathcal{G}\in\MMHM(X)$, with $\form{X}(\mathcal{F})$ and $\form{X}(\mathcal{G})$ semisimple perverse sheaves, and let $f\colon \mathcal{F}\rightarrow\mathcal{G}$ be a morphism between them.  Then $f=0$ if and only if $f_x=0$ for every closed point $\iota_x\colon x\hookrightarrow X$.
\end{lemma}
\begin{proof}
Again, one direction is trivial, and we may assume that we are working in $\perv(X)$ throughout, by faithfulness of $\form{X}$.  So assume that $f_x=0$ for all $x$.  Let $\{S_t\}_{t\in T}$ be the set of all simple perverse sheaves appearing in a direct sum decomposition of $\mathcal{F}$ and $\mathcal{G}$.  Then there is a commutative diagram
\[
\xymatrix{
\bigoplus_{t\in T} S_t\otimes_{\mathbb{C}}\Hom(S_t,\mathcal{F})\ar[r]^-{\cong}\ar[d]_{\bigoplus_t \id_t\otimes f_t}&\mathcal{F}\ar[d]_f\\
\bigoplus_{t\in T} S_t\otimes_{\mathbb{C}}\Hom(S_t,\mathcal{G})\ar[r]^-{\cong}& \mathcal{G}
}
\]
where $\id_t\colon S_t\rightarrow S_t$ is the identity map, and $f_t\colon \Hom(S_t,\mathcal{F})\rightarrow \Hom(S_t,\mathcal{G})$ is obtained from composition with $f$.  Writing
\[
S_t=\ICSn_{Y}(\mathcal{L})[\dim(Y)]
\]
for $Y\subset X$ a closed irreducible subvariety, and $\mathcal{L}$ a local system defined on $Y'\subset Y$ a regular open subvariety, and picking $x\in Y'$, we deduce that $\iota_x^*\id_t\neq 0$, and so $f_t=0$ for all $t$.
\end{proof}

With these lemmas in hand, we move back to consideration of cohomological Hall algebras.  Throughout this section we assume that $\zeta$ is $\mu$-generic.  By Proposition \ref{inccor} there is a canonical embedding 
\[
\iota=\Ho(\dim_*\tilde{\omega}^{\zeta,!}_{\mu}\Upsilon_{W,\mu}^{\zeta})
\]
of $\HO(\BC )_{\vir}\otimes \DT_{W,\mu}^{\zeta,\SP}$ inside $\Gr_{\Pf}(\HO(\Coha^{\zeta,\SP}_{W,\mu}))$.
\begin{proposition}
\label{primitivity}
The subspace $\iota(\HO(\BC )_{\vir}\otimes \DT_{W,\mu}^{\zeta,\SP})$ is a primitive subspace in $\Gr_{\Pf}(\HO(\Coha_{W,\mu}^{\zeta,\SP}))$, i.e. the following composition of maps is zero:
\begin{align*}
&\bigoplus_{\dd\in \Lambda_{\mu}^{\zeta}\setminus\{0\}}\HO(\BC )_{\vir}\otimes\DT_{W,\dd}^{\zeta,\SP}\xrightarrow{\iota}\Gr_{\Pf}(\HO(\Coha^{\zeta,\SP}_{W,\mu}))\xrightarrow{\Gr_{\Pf}(\Delta^{\zeta,\SP}_{W,\mu})}\\&\bigoplus_{\dd',\dd''\in\Lambda_{\mu}^{\zeta}}\Gr_{\Pf}\left((\HO(\Coha^{\zeta,\SP}_{W,\dd'})\boxtimes^{\tw}_+\HO(\Coha^{\zeta,\SP}_{W,\dd''}))[\mathfrak{L}_{\dd',\dd''}^{-1}]\right)
\twoheadrightarrow \bigoplus_{\dd',\dd''\in\Lambda_{\mu}^{\zeta}\setminus\{0\}}\Gr_{\Pf}\left((\Coha^{\zeta,\SP}_{W,\dd'}\boxtimes^{\tw}_+\Coha^{\zeta,\SP}_{W,\dd''})[\mathfrak{L}_{\dd',\dd''}^{-1}]\right).
\end{align*}
\end{proposition}
\begin{proof}
We need to show that for decompositions $\dd=\dd'+\dd''$, with $\dd'\neq 0\neq \dd''$, the image of $\iota(\HO(\BC )_{\vir}\otimes\DT_{W,\dd}^{\zeta,\SP})$ under the map
\begin{align}
\label{int4}
&\Gr_{\Pf}(\Delta^{\zeta,\SP}_{\dd',\dd''})\colon \HO\left(\Msp_{\dd}^{\zeta\sst,\SP},\tilde{\omega}_{\dd}^!\Ho\left(p^{\zeta}_{\dd,*}\phim{\WWW^{\zeta}_{\dd}}\ICS_{\Mst_{\dd}^{\zeta\sst}}\right)\right)\rightarrow \\ \nonumber&\left(\HO\left(\Msp_{\dd'}^{\zeta\sst,\SP},\Ho\left(p^{\zeta}_{\dd',*}\phim{\WWW^{\zeta}_{\dd'}}\ICS_{\Mst_{\dd'}^{\zeta\sst}}\right)\right)\boxtimes_{+}^{\tw} \HO\left(\Msp^{\zeta\sst,\SP}_{\dd''},\Ho\left(p^{\zeta}_{\dd'',*}\phim{\WWW^{\zeta}_{\dd''}}\ICS_{\Mst_{\dd''}^{\zeta\sst}}\right)\right)\right)[\mathfrak{L}_{\dd',\dd''}^{-1}]
\end{align}
is zero.  We assume that $\Msp^{\zeta\st}_{\dd}\neq \emptyset$, or the statement is trivial, for then by definition $\DT_{W,\dd}^{\zeta,\SP}=0$.  Denote by $R\Gamma$ the derived global sections functor on $\Msp_{\dd}^{\zeta\sst,\SP}$.  In terms of the commutative diagram
\begin{equation}
\label{int5}
\xymatrix{
R\Gamma\left(\tilde{\omega}^{\zeta,!}_{\dd}\Ho\left(p^{\zeta}_{\mu,*}\phim{\WWW_{\dd}^{\zeta}}\ICS_{\Mst_{\dd}^{\zeta\sst}}(\QQ)\right)\right)\ar[r]^-{\Gr_{\Pf}(\gamma)} &R\Gamma\left(\tilde{\omega}^{\zeta,!}_{\dd}\Ho\left(p^{\zeta}_{\mu,*}s_{\dd',\dd'',*}\phim{\WWW^{\zeta}_{\dd',\dd''}}\LL^{(\dd,\dd)/2}_{\Mst^{\zeta\sst}_{\dd',\dd''}}\right)\right)\\
R\Gamma\left(\tilde{\omega}^{\zeta,!}_{\dd}\phim{\WW_{\dd}^{\zeta}}\Ho\left(p^{\zeta}_{\mu,*}\ICS_{\Mst_{\dd}^{\zeta\sst}}(\QQ)\right)\right)\ar[r]^-{\gamma'} \ar[u]^{\LL^{(\dd,\dd)/2}\otimes R\Gamma\left( \tilde{\omega}^{\zeta,!}_{\dd}\nu_{\dd}\right)}& R\Gamma\left(\tilde{\omega}^{\zeta,!}_{\dd}\phim{\WW_{\dd}^{\zeta}}\Ho\left(p^{\zeta}_{\mu,*}s_{\dd',\dd'',*}\LL^{(\dd,\dd)/2}_{\Mst^{\zeta\sst}_{\dd',\dd''}}\right)\right)\ar[u]^{\LL^{(\dd,\dd)/2}\otimes R\Gamma\left(\tilde{\omega}^{\zeta,!}_{\dd}\nu_{\dd',\dd''}\right)}\\
R\Gamma\left( \tilde{\omega}^{\zeta,!}_{\dd}\phim{\WW^{\zeta}_{\dd}}\ICS_{\Msp^{\zeta\sst}_{\dd}}(\QQ)\otimes\HO(\BC )_{\vir}\right)\ar[u]^{\iota}
}
\end{equation}
with $\gamma$ as in (\ref{gammadef}), it is sufficient to show that 
\[
\Gr_{\Pf}(\gamma)\circ(\LL^{(\dd,\dd)/2}\otimes R\Gamma(\tilde{\omega}^{\zeta,!}_{\dd}\nu_{\dd}))\circ\iota=0.
\]
The maps $\nu_{\dd}$ and $\nu_{\dd',\dd''}$ are as in Proposition \ref{cvs}.  On the other hand, the map $\gamma'\iota$ is obtained by applying $R\Gamma\left(\tilde{\omega}^{\zeta,!}_{\dd}\phim{\WW_{\dd}^{\zeta}}\right)$ to a map
\begin{align}
\label{nearmiss}
\HO(\BC )_{\vir}\otimes\ICS_{\Msp^{\zeta\sst}_{\dd}}(\QQ)\rightarrow \Ho\left(p^{\zeta}_{\mu,*}s_{\dd',\dd'',*}\LL^{(\dd,\dd)}_{\Mst^{\zeta\sst}_{\dd',\dd''}}\right)
\end{align}
which preserves perverse degree, and so is necessarily zero, since for every cohomological degree, the domain and the target of (\ref{nearmiss}) are pure complexes of mixed Hodge modules with distinct strict support.
\end{proof}
Let $\mathcal{P}^{\zeta,\SP}_{W,\mu}\subset \Gr_{\Pf}(\HO(\Coha^{\zeta,\SP}_{W,\mu}))$ be the subalgebra generated by 
\[ 
\bigoplus_{\dd\in\Lambda_{\mu}^{\zeta}}\HO(\BC )_{\vir}\otimes\DT_{W,\dd}^{\zeta,\SP}.  
\]
Since $\mathcal{P}^{\zeta,\SP}_{W,\mu}$ is generated by primitive elements of the localised Hopf algebra, it follows from Lemma \ref{PAB} that the localised bialgebra structure on $\mathcal{P}^{\zeta,\SP}_{W,\mu}$ lifts to an honest cocommutative bialgebra structure.  Here, cocommutativity is with respect to the symmetrizing morphism ${}^{\tau}\mathbf{sw}$ defined in Section \ref{lc_signs}.  We denote by $\overline{\Delta}_{W,\mu}^{\zeta,\SP}\colon \mathcal{P}^{\zeta,\SP}_{W,\mu}\rightarrow \mathcal{P}^{\zeta,\SP}_{W,\mu}\boxtimes_+ \mathcal{P}^{\zeta,\SP}_{W,\mu}$ the induced coproduct.
\begin{corollary}
\label{trueHopf}
Assume that $\zeta$ is $\mu$-generic.  The triple 
\[
\left(\mathcal{P}_{W,\mu}^{\zeta,\SP},\Ho\left(\dim_*\Ho(\ms_{W,\mu}^{\zeta,\SP})\right), \overline{\Delta}^{\zeta,\SP}_{W,\mu}\right).
\]
extends to a Hopf algebra.
\end{corollary}
\begin{proof}
All that is missing is a compatible antipode, but existence and uniqueness of an antipode is a formal consequence of connectedness of the algebra.
\end{proof}

We define $\Ho(\Gamma_{W,\mu}^{\zeta})$ to be the composition of maps in the following diagram
\[
\xymatrix{
{}^{\tau}\!\Sym_{\boxtimes_{\oplus}}(\HO(\BC )_{\vir}\otimes\DTS_{W,\mu}^{\zeta})\ar@{^(->}[r]& \Free_{\boxtimes_{\oplus}}(\HO(\BC )_{\vir}\otimes\DTS_{W,\mu}^{\zeta})\ar[dl]^-{\;\;\;\Free_{\boxtimes_{\oplus}}(\Upsilon_{W,\mu}^{\zeta})}\\
 \Free_{\boxtimes_{\oplus}}\left(\Ho(\mathcal{A}_{W,\mu}^{\zeta})\right)\ar[r]&\Ho\left(\mathcal{A}_{W,\mu}^{\zeta}\right),
}
\]
where the upper horizontal arrow comes from the natural inclusion 
\[
{}^{\tau}\!\Sym^n_{\boxtimes_{\oplus}}(\HO(\BC )_{\vir}\otimes \DTS_{W,\mu}^{\zeta})\subset \underbrace{\left(\HO(\BC )_{\vir}\otimes\DTS_{W,\mu}^{\zeta}\right)\boxtimes_{\oplus}\ldots\boxtimes_{\oplus}\left(\HO(\BC )_{\vir}\otimes\DTS_{W,\mu}^{\zeta}\right)}_{n \;\textrm{times}},
\]
the lower horizontal arrow is given by iterated relative CoHA multiplication, and $\Upsilon_{W,\mu}^{\zeta}$ is as in (\ref{caninc}).  The relative version of Theorem \ref{qea} amounts to the following theorem.
\begin{theorem}
\label{zPBW}
The map $\Ho(\Gamma_{W,\mu}^{\zeta})$ is an isomorphism.
\end{theorem}
\begin{proof}
It is enough to show that $\Xi=\DD^{\mon}_{\Msp_{\mu}^{\zeta\sst}}\Ho(\Gamma^{\zeta}_{W,\mu})$ is an isomorphism, and for this it is enough to show that $\Xi|_{x}$ is an isomorphism for each $x$ a point in $\Msp_{\dd}^{\zeta\sst}$, with $\dd\in\Lambda^{\zeta}_{\mu}$, by Lemma \ref{Tech1}.  Let $x$ represent the polystable $\dd$-dimensional $\mathbb{C}Q$-representation
\[
\bigoplus_i V_i\otimes_{\mathbb{C}} \rho_i
\]
where the $\rho_i$ are pairwise distinct $\zeta$-stable $\mathbb{C}Q$-representations of slope $\mu$, and $V_i$ are finite-dimensional complex vector spaces.  Let $\Mst^{x}\subset\Mst^{\zeta\sst}_{\mu}$ be the reduced substack, of which the closed points represent $\mathbb{C}Q$-representations that are objects of the Serre subcategory for which the objects are all of the iterated extensions of the representations $\rho_i$.  Let $\Msp^{x}\subset \Msp^{\zeta\sst}_{\mu}$ be the discrete subscheme for which the points represent finite direct sums of the $\rho_i$.  Then 
\[
\Mst^{x}\cap \Mst^{\zeta\sst}_{\dd}\supset (p^{\zeta}_{\mu})^{-1}(x)
\]
and $\Msp_{\dd}^x\subset \Msp_{\dd}^{\zeta\sst}$ is a discrete set of points containing $x$.  So $\Xi\lvert_x$ is an isomorphism if and only if $\Xi\lvert_{{\Msp}^x_{\dd}}$ is.  By Theorem \ref{ThmA} and Proposition \ref{underlying_same}, there is an isomorphism of $\mathbb{N}^{Q_0}\oplus\ZZ$-graded vector spaces (the extra $\ZZ$-factor is the cohomological grading)
\begin{equation}
\label{fout}
{}^{\tau}\!\Sym_{\boxtimes_+}\left(\HO(\BC )_{\vir}\otimes \DT^{\zeta,x}_{W,\mu}\right)\cong \HO_c(\Mst^x,\ICSt_{W,\mu}^{\zeta})^{\vee}.
\end{equation}
On the other hand, by Proposition \ref{primitivity} the subspaces 
\[
\HO(\BC )_{\vir}\otimes \DT^{\zeta,x}_{W,\mu}\subset \Gr_{\Pf}\left(\HO_c(\Mst^x,\ICSt_{W,\mu}^{\zeta})^{\vee}\right)
\]
are primitive for the induced coproduct, so they form a Lie algebra under the commutator Lie bracket.  Letting $\mathfrak{g}$ be the space of primitives for $\Gr_{\Pf}\left(\HO_c(\Mst^x,\ICSt_{W,\mu}^{\zeta})^{\vee}\right)$, and letting $U(\mathfrak{g})$ be the universal enveloping algebra of $\mathfrak{g}$,then by the Milnor--Moore theorem, there is an isomorphism
\[
U(\mathfrak{g})\cong \Gr_{\Pf}\left(\HO_c(\Mst^x,\ICSt_{W,\mu}^{\zeta})^{\vee}\right).
\]
Setting $\upsilon=\Ho(\dim_*\tilde{\omega}_{\mu}^!\Ho(\Gamma_{W,\mu}))$ there is a commutative diagram
\[
\xymatrix{
{}^{\tau}\!\Sym_{\boxtimes_+}\left(\HO(\BC )_{\vir}\otimes \DT^{\zeta,x}_{W,\mu}\right)\ar@{^{(}->}[d]\ar[r]^-{\cong}\ar[dr]^-{\upsilon} &U\left(\HO(\BC )_{\vir}\otimes \DT^{\zeta,x}_{W,\mu}\right)\ar[d]
\\{}^{\tau}\!\Sym_{\boxtimes_+}\left(\mathfrak{g}\right)\ar[r]^-{\cong}&\Gr_{\Pf}\left(\HO_c(\Mst^x,\ICSt_{W,\mu}^{\zeta})^{\vee}\right)
}
\]
where the top and bottom horizontal arrows are the PBW isomorphism and the Milnor--Moore isomorphism, respectively.  It follows that $\upsilon$ is injective, and hence an isomorphism, since the domain and the codomain have the same graded dimensions, by (\ref{fout}).  The dual of the restriction of this map to the degree $\dd$ piece is $\Xi|_{{\Msp}^x_{\dd}}$, which is thus an isomorphism, as required.
\end{proof}
We construct the map
\[
\HO(\Gamma_{W,\mu}^{\zeta,\SP})\colon {}^{\tau}\!\Sym_{\boxtimes_+}\left(\HO(\BC )_{\vir}\otimes\DT^{\zeta,\SP}_{W,\mu}\right)\rightarrow\HO\left(\mathcal{A}^{\zeta,\SP}_{W,\mu}\right)
\]
in the same way as $\Ho(\Gamma_{W,\mu}^{\zeta})$, using Corollary \ref{Inccor} instead of Corollary \ref{inccor}.  Then this map preserves the perverse filtration, and
\[
\Gr_{\Pf}\left(\HO(\Gamma_{W,\mu}^{\zeta,\SP})\right)=\Ho(\dim_*\tilde{\omega}^{\zeta,!}_{\mu}\Ho(\Gamma_{W,\mu}^{\zeta}))
\]
is an isomorphism, so that $\HO(\Gamma_{W,\mu}^{\zeta,\SP})$ is at least an injection.  The domain and target of this map are graded vector spaces of the same dimension (finite in every graded degree) by Theorem \ref{ThmA}.  It follows that $\HO(\Gamma_{W,\mu}^{\zeta,\SP})$ is an isomorphism.  This completes all but the free commutativity statements of Theorem \ref{qea}, which we now deduce as a corollary.
\begin{corollary}
\label{twSC}
The (twisted) relative cohomological Hall algebra $\Ho(\mathcal{A}^{\zeta,\SP}_{W,\mu})$ is a free commutative algebra (for the symmetrizing morphism ${}^{\tau}\!\textbf{sw}$), as is the algebra $\Gr_{\Pf}\left(\HO(\mathcal{A}^{\zeta,\SP}_{W,\mu})\right)$, and so the twisted algebra $\Gr_{\Pf}\left(\HO(^{\psi}\!\mathcal{A}^{\zeta,\SP}_{W,\mu})\right)$ is a free supercommutative algebra (for the usual symmetrizing morphism $\textbf{sw}$).
\end{corollary}
\begin{proof}
The result regarding the relative cohomological Hall algebra follows from Theorem \ref{zPBW} and the claim that $\Ho(\mathcal{A}^{\zeta}_{W,\mu})$ is commutative.  This claim in turn can be checked at fibres as in the proof of Theorem \ref{zPBW}, this time using Lemma \ref{Tech2}, and so it is sufficient to prove commutativity for $\Gr_{\Pf}\left(\HO(\mathcal{A}^{\zeta,\SP}_{W,\mu})\right)$ (as in the same proof).  Let $\mathcal{P}\subset \mathcal{K}=\Gr_{\Pf}\left(\HO(\mathcal{A}^{\zeta,\SP}_{W,\mu})\right)$ be the subspace of primitives.  Then the algebra generated by $\mathcal{P}$ is a cocommutative sub bialgebra of $\mathcal{K}$, and so there is a chain of inclusions 
\[
{}^{\tau}\!\Sym\left(\HO(\BC)_{\vir}\otimes \DT^{\zeta,\SP}_{W,\mu}\right)\subset {}^{\tau}\!\Sym(\mathcal{P})\subset \mathcal{K}
\]
where the first inclusion is given by Proposition \ref{primitivity} and the second by the Milnor--Moore theorem.  Since by Theorem \ref{ThmA} the first and last graded vector spaces have the same graded dimensions, all of the inclusions are equalities, and in particular $\mathcal{P}=\DT^{\zeta,\SP}_{W,\mu}\otimes\HO(\BC)_{\vir}$, and $\mathcal{K}$ is the universal enveloping algebra of $\mathcal{P}$, considered as a Lie sub-algebra of $\mathcal{K}$.  On the other hand, the Lie bracket on $\mathcal{P}$ is zero, since it respects the perverse grading, and $\mathcal{P}$ is concentrated entirely in odd perverse degree.  The statement regarding the $\psi$-twisted algebra then follows.
\end{proof}

Finally, we are in a position to define the BPS Lie algebra:
\begin{corollary}
Let $\psi$ be as in (\ref{pcocdef}).  The subspace 
\[
\LL^{1/2}\otimes \DT^{\zeta,\SP}_{W,\mu}=\Pf_{\leq 1}\left(\HO(\mathcal{A}^{\zeta,\SP}_{W,\mu})\right)\subset \HO(\mathcal{A}^{\zeta,\SP}_{W,\mu})
\]
is closed under the commutator Lie bracket in $\HO(^{\psi}\!\mathcal{A}^{\zeta,\SP}_{W,\mu})$, and so carries the structure of a Lie algebra, called the \textbf{BPS Lie algebra}.
\end{corollary}
The proof is an immediate consequence of Corollary \ref{twSC}, as explained in Section \ref{BPSLA}.
\subsection{Proof of Theorem \ref{strongPBW}}
\label{CWCF}
We finish by proving Theorem \ref{strongPBW}.  Fix a generic stability condition $\zeta$.  By Corollary \ref{inccor} there are canonical split inclusions 
\[
\HO(\BC )_{\vir}\otimes \DTS^{\zeta}_{W,\mu}\rightarrow \Ho\left(p_{\mu,*}^{\zeta}\ICSt_{W,\mu}^{\zeta}\right)
\]
giving rise to canonical split inclusions $\HO(\BC )_{\vir}\otimes q^{\zeta}_{\mu,*}\DTS^{\zeta}_{W,\mu}\rightarrow q_{\mu,*}^{\zeta}\Ho\left(p^{\zeta}_{\mu,*}\ICSt_{W,\mu}^{\zeta}\right)$ and thus split inclusions
\[
\iota'_{\mu}\colon \HO(\BC )_{\vir}\otimes q^{\zeta}_{\mu,*}\DTS_{W,\mu}^{\zeta}\hookrightarrow \Ho\left(q^{\zeta}_{\mu,*}p^{\zeta}_{\mu,*}\ICSt^{\zeta}_{W,\mu}\right)
\]
by the decomposition theorem and properness of $q^{\zeta}_{\mu}$.  Via the proof of Theorem \ref{gcDT} we can pick a right inverse $\eta_{\mu,*}$ to the restriction map in cohomology 
\[
\eta_{\mu}^*\colon\Ho\left(p_{\mu,*}\ICSt_{W,\mu}\right)\rightarrow\Ho\left(q^{\zeta}_{\mu,*}p^{\zeta}_{\mu,*}\ICSt^{\zeta}_{W,\mu}\right)
\]
for each $\mu$.  We set 
\[
\iota_{\mu}=\eta_{\mu,*}\iota'_{\mu}\colon \HO(\BC )_{\vir}\otimes q^{\zeta}_{\mu,*}\DTS_{W,\mu}^{\zeta} \rightarrow \Ho\left(p_{\mu,*}\ICSt_{W,\mu}\right).
\]
\begin{theorem}
\label{sPBW}
Let $\zeta$ be a generic stability condition, and consider the morphism
\begin{equation}
\label{finalPBW}
\boxtimes_{\oplus,\infty \xrightarrow{\mu}-\infty}^{\tw} \left({}^{(\tau)}\!\Sym_{\boxtimes_{\oplus}}\left(\HO(\BC )_{\vir}\otimes q_{\mu,*}^{\zeta}\DTS_{W,\mu}^{\zeta}\right)\right)\xrightarrow{^{(\neg\psi)}\! \Ho(\ms_W)} \Ho\left(p_*\ICSt_W\right)
\end{equation}
induced by the inclusions $\iota_{\mu}$ and the Hall algebra structure on $\Ho({}^{(\psi)}\!\!\mathcal{A}_W)$, where the bracketed superscripts indicate that if we are not taking the symmetrizing morphism ${}^{\tau}\mathbf{tw}$ to define the functor ${}^{\tau}\!\Sym_{\boxtimes_{\oplus}}$, then we take the $\psi$-twisted multiplication, for $\psi$ a bilinear form as in (\ref{pcocdef}).  Then (\ref{finalPBW}) is an isomorphism.
\end{theorem}

\begin{proof}

We consider dual compactly supported cohomology, to match Theorem \ref{gcDT}.  In addition, we will assume that $W=0$ to ease the notation a little ---  applying the vanishing cycles functor to all of our arguments yields the general case by Lemma \ref{cohacom}.

Let $\overline{\dd}\in\HN_{\dd}$ be a Harder--Narasimhan type.  Extending the notation $\Mst_{\dd',\dd''}$, we let $\Mst_{\overline{\dd}}$ be the stack of all $\dd$-dimensional modules equipped with a filtration by submodules with the dimensions of successive subquotients given by $\overline{\dd}$.  We define
\[
s_{\overline{\dd}}\colon\Mst_{\overline{\dd}}\rightarrow \Mst_{\dd}
\]
to be the map of stacks given by forgetting the filtration.  In common with the map $s_{\dd',\dd''}$ of diagram (\ref{smap}), which corresponds to the special case $\overline{\dd}=(\dd',\dd'')$, the map $s_{\overline{\dd}}$ factors as a proper surjection and a closed inclusion, and is therefore proper.  

The Harder--Narasimhan type of the underlying $\dd$-dimensional module of a filtered module parametrised by a point of $\Mst_{\overline{\dd}}$ need not be $\overline{\dd}$.  On the other hand, by the proof of \cite[Prop.3.7]{Reineke_HN}, the proper map $s_{\overline{\dd}}$ factors through the closed inclusion of stacks
\[
i_{\leq \overline{\dd}}\colon\Mst^{\zeta}_{\leq\overline{\dd}}\hookrightarrow \Mst_{\dd}.
\]

By abuse of notation, for each dimension vector $\dd$ we continue to denote by $\eta_{\dd,*}$ the right inverse to the map
\[
\DD^{\mon}_{\Msp}\Ho\left(p_{\dd,!}\ICS_{\Mst_{\dd}}(\QQ)\right)\rightarrow \DD^{\mon}_{\Msp}\Ho\left(q_{\dd,!}^{\zeta}p_{\dd,!}^{\zeta}\ICS_{\Mst^{\zeta\sst}_{\dd}}(\QQ)\right)
\]
defined by the map already denoted $\eta_{\dd,*}$.  Given a stack $\mathfrak{N}\xrightarrow{r} \Msp_{\dd}$ over $\Msp_{\dd}$, we denote by $\IS(\mathfrak{N})=\DD^{\mon}_{\Msp_{\dd}}\Ho(r_!\QQ_{\mathfrak{N}})$ the induced complex of mixed Hodge modules on $\Msp_{\dd}$.  Consider the commutative diagram
\[
\xymatrix{
&\Mst_{\leq \overline{\dd}}^{\zeta}\ar[dr]^{i_{\leq\overline{\dd}}}\\
\Mst_{\overline{\dd}}^{\zeta}\ar[d]_{r_{\overline{\dd}}^{\zeta}}\ar[ur]^{\beta}\ar[r]^{\gamma} & \Mst_{\overline{\dd}} \ar[u]^{\alpha}\ar[r]_{s_{\overline{\dd}}} \ar[d]_{r_{\overline{\dd}}}& \Mst_{\dd}\\
\Mst_{\dd^1}^{\zeta\sst}\times\ldots\times\Mst_{\dd^s}^{\zeta\sst}\ar[r]^{\delta} & \Mst_{\dd^1}\times\ldots\times\Mst_{\dd^s}
}
\]
where the map $\alpha$ exists due to the comments above.

We consider the diagram of complexes of mixed Hodge modules
\[
\xymatrix{
&\IS(\Mst_{\leq \overline{\dd}}^{\zeta})\ar[dl]^{\beta^*}\ar[dr]^{i_{\leq\overline{\dd},*}}\\
\IS(\Mst_{\overline{\dd}}^{\zeta})\ar@/^/@{.>}[ur]^{\beta_*} \ar[d]_{r^{\zeta}_{\overline{\dd},*}}^{\cong} \ar@/_/@{.>}[r]_{\gamma_*}&\ar[u]^{\alpha_*} \ar[l]_{\gamma^*}\IS(\Mst_{\overline{\dd}})\ar[r]_{s_{\overline{\dd},*}}\ar[d]_{r_{\overline{\dd},*}}^{\cong} & \IS(\Mst_{\dd})\\
\LL^{(\overline{\dd},\overline{\dd})}\otimes\IS(\Mst_{\dd^1}^{\zeta\sst}\times\ldots\times\Mst_{\dd^s}^{\zeta\sst}) \ar@/_/@{.>}[r]_-{\delta_*}& \LL^{(\overline{\dd},\overline{\dd})}\otimes\IS(\Mst_{\dd^1}\times\ldots\times\Mst_{\dd^s})\ar[l]_-{\delta^*}
}
\]
with $(\overline{\dd},\overline{\dd})$ defined as in (\ref{dddotdef}).  The map $\delta_*\colonequals \LL^{(\overline{\dd},\overline{\dd})}\otimes\left(\oplus_*(\eta_{\dd^1,*}\boxtimes\ldots\boxtimes\eta_{\dd^s,*})\right)$ is a right inverse to the natural restriction map $\delta^*$.  Then we define
\[
\gamma_*\colonequals r_{\overline{\dd},*}^{-1}\delta_*r^{\zeta}_{\overline{\dd},*},
\]
a right inverse to $\gamma^*$.  Then set $\beta_*\colonequals \alpha_*\gamma_*$.  It can easily be verified that $\gamma^*=\beta^*\alpha_*$, and so $\beta_*$ is a right inverse to $\beta^*$.

The map $\eta_{\dd^1,*}(\mathcal{H}(\mathcal{A}_{\dd^1}^{\zeta}))\boxtimes_{\oplus}^{\tw}\ldots \boxtimes_{\oplus}^{\tw}\eta_{\dd^s,*}(\mathcal{H}(\mathcal{A}_{\dd^1}^{\zeta}))\rightarrow \mathcal{H}(\mathcal{A}_{\dd})$ has image in the piece $\LL^{(\dd,\dd)/2}\otimes \IS(\Mst^{\zeta}_{\leq\overline{\dd}})$ of the filtration of $\mathcal{H}(p_*\ICSt_0)$ by Harder--Narasimhan types.  By the commutativity of the above diagram, the resulting map
\[
\gamma^*\left(\eta_{\dd^1,*}(\mathcal{H}(\mathcal{A}_{\dd^1}^{\zeta}))\boxtimes_{\oplus}^{\tw}\ldots \boxtimes_{\oplus}^{\tw}\eta_{\dd^s,*}(\mathcal{H}(\mathcal{A}_{\dd^1}^{\zeta}))\right)\rightarrow  \LL^{(\dd,\dd)/2}\otimes\beta^*\IS(\Mst^{\zeta}_{\leq\overline{\dd}})
\]
is an isomorphism.  Filtering the domain and the target of the map
\[
\Psi\colon \boxtimes_{\oplus,\infty \xrightarrow{\mu}-\infty}^{\tw}\left(\eta_{\mu,*}(\mathcal{H}(\mathcal{A}_{\mu}^{\zeta}))\right)\rightarrow \mathcal{H}(\mathcal{A})
\]
by the filtration indexed by the Harder--Narasimhan filtration, it follows that the associated graded morphism is an isomorphism between pure complexes of mixed Hodge modules, and hence $\Psi$ is an injection, and so an isomorphism, since by Theorem \ref{CWCT} the domain and image are isomorphic. 

Since $\eta_{\mu}^*$ is a morphism of monoids, the diagram 
\[
\xymatrix{
\eta_{\mu}^* {}^{(\tau)}\!\Sym_{\boxtimes_{\oplus}}\left(\eta_{\mu,*}\DTS_{\mu}^{\zeta}\otimes\HO(\BC )_{\vir}\right)\ar[d]^{\cong}\ar[r]& \eta_{\mu}^*\mathcal{H}({}^{(\neg\psi)}\!\!\mathcal{A}_{\mu})\ar[d]^{\cong}\\
{}^{(\tau)}\!\Sym_{\boxtimes_{\oplus}}\left(\DTS_{\mu}^{\zeta}\otimes\HO(\BC )_{\vir}\right)\ar[r]& \mathcal{H}({}^{(\neg\psi)}\!\!\mathcal{A}_{\mu}^{\zeta})
}
\]
commutes, with the horizontal maps defined via the relative multiplications on $\mathcal{H}({}^{(\psi)}\!\!\mathcal{A}_{\mu})$ and $\mathcal{H}({}^{(\psi)}\!\!\mathcal{A}_{\mu}^{\zeta})$ respectively.  We deduce as above that the theorem is true after passing to the associated graded with respect to the Harder--Narasimhan filtration, and so is injective, and hence an isomorphism since the domain and target are isomorphic pure monodromic mixed Hodge modules by Theorems \ref{ThmA} and \ref{CWCT}.
\end{proof}
\begin{corollary}
For a generic stability condition $\zeta$ there exist embeddings $\DT^{\zeta,\SP}_{W,\mu}\otimes\HO(\BC )_{\vir}\subset \HO({}^{(\psi)}\!\!\Coha_{W})$ such that the induced morphism
\[
\boxtimes_{\oplus,\infty \xrightarrow{\mu}-\infty}^{\tw}\left({}^{(\tau)}\!\Sym_{\boxtimes_+}\left(\DT^{\zeta,\SP}_{W,\mu}\otimes\HO(\BC )_{\vir}\right)\right)\xrightarrow{{}^{(\neg\psi)}\!\!\HO(\ms_W)}\HO(\Coha^{\SP}_W)
\]
is an isomorphism.
\end{corollary}
The meaning of the bracketed superscripts is as in Theorem \ref{sPBW}.
\begin{proof}
By Theorem \ref{sPBW} and Proposition \ref{PCoha}, the corollary is true after replacing $\HO(\Coha^{\SP}_W)$ with the perverse associated graded $\Gr_{\Pf}(\HO(\Coha^{\SP}_{W}))=\HO_c(\Msp^{\SP},\dim_!\Ho(p_*\ICSt_W))$.  Then for any lift of the embedding $\DT^{\zeta,\SP}_{W,\mu}\otimes\HO(\BC )_{\vir}\subset \Gr_{\Pf}(\HO(\Coha^{\SP}_{W}))$ to an embedding $\DT^{\zeta,\SP}_{W,\mu}\otimes\HO(\BC )_{\vir}\subset \HO(\Coha^{\SP}_{W})$, the result follows.
\end{proof}

\bibliographystyle{plain}
\bibliography{Literatur}

\begin{thebibliography}{10}

\bibitem{Stackproject}
{\em {The Stacks Project}}.
\newblock http://stacks.math.columbia.edu/.

\bibitem{AtBo83}
M.~Atiyah and R.~Bott.
\newblock {The Yang-Mills equations over Riemann surfaces}.
\newblock {\em Philos. Trans. Roy. Soc. London}, 308:523--615, 1983.

\bibitem{BBD}
A.~Beilinson, J.~Bernstein, and P.~Deligne.
\newblock Faisceaux pervers.
\newblock {\em Ast{\'e}risque}, 100, 1983.

\bibitem{BL94}
J.~Bernstein and V.~Lunts.
\newblock {\em {Equivariant Sheaves and Functors}}, volume 1578.
\newblock Springer, 1994.

\bibitem{Br12}
C.~Brav, V.~Bussi, D.~Dupont, D.~Joyce, and B.~Szendr{\H{o}}i.
\newblock Symmetries and stabilization for sheaves of vanishing cycles.
\newblock {\em J. Singul.}, 11:85--151, 2015.
\newblock With an appendix by J{\"o}rg Sch{\"u}rmann.

\bibitem{Bridgeland02}
T.~Bridgeland.
\newblock Stability conditions on triangulated categories.
\newblock {\em Annals of Math.}, 166:317--345, 2007.

\bibitem{Carlson}
James Carlson.
\newblock Extensions of mixed hodge structures.
\newblock {\em Journ{\'e}es de g{\'e}om{\'e}trie alg{\'e}brique d’Angers},
  pages 107--128, 1979.

\bibitem{Chev58}
C.~Chevalley.
\newblock {\em {Anneaux de Chow et applications}}.
\newblock Secr{\'e}tariat math{\'e}matique, 1958.

\bibitem{Da13b}
B.~Davison.
\newblock {Purity of critical cohomology and Kac's conjecture}.
\newblock 2013.
\newblock axXiv:1311.6989.

\bibitem{Da13}
B.~Davison.
\newblock {The critical CoHA of a quiver with potential}.
\newblock {\em The Quarterly Journal of Mathematics}, 68(2):635--703, 2017.

\bibitem{Da16a}
B.~Davison.
\newblock {Positivity for quantum cluster algebras}.
\newblock {\em Annals of Math.}, 187(1):157--219, 2018.

\bibitem{DMSS15}
B.~Davison, D.~Maulik, J.~Sch{\"u}rmann, and B.~Szendr{\H{o}}i.
\newblock Purity for graded potentials and quantum cluster positivity.
\newblock {\em Compos. Math.}, 151(10):1913--1944, 2015.

\bibitem{DaMe4}
B.~Davison and S.~Meinhardt.
\newblock {Donaldson-Thomas theory for categories of homological dimension one
  with potential}.
\newblock 2015.
\newblock arXiv:1512.08898.

\bibitem{deCat12}
MAA. de~Cataldo, T.~Hausel, and L.~Migliorini.
\newblock {Topology of Hitchin systems and Hodge theory of character varieties:
  the case $A_1$}.
\newblock {\em Annals of Math.}, 175(3):1329--1407, 2012.

\bibitem{EdGr98}
D.~Edidin and W.~Graham.
\newblock {Equivariant intersection theory (With an appendix by Angelo Vistoli:
  The Chow ring of M2)}.
\newblock {\em Invent. Math.}, 131(3):595--634, 1998.

\bibitem{Efimov}
A.~Efimov.
\newblock {Cohomological Hall algebra of a symmetric quiver}.
\newblock {\em Comp. Math.}, 148, no. 4:1133--1146, 2012.

\bibitem{FrRe15}
H.~Franzen and M.~Reineke.
\newblock {Semi-Stable Chow--Hall Algebras of Quivers and Quantized
  Donaldson--Thomas Invariants}.
\newblock {\em arXiv preprint arXiv:1512.03748}, 2015.

\bibitem{HM98}
J.~A. Harvey and G.~Moore.
\newblock {On the algebras of BPS states}.
\newblock 197:489--519, 10 1998.

\bibitem{HLRV13}
T.~Hausel, E.~Letellier, and F.~Rodriguez-Villegas.
\newblock {Positivity for Kac polynomials and DT-invariants of quivers}.
\newblock {\em Annals of Math.}, 177:1147--1168, 2013.

\bibitem{JoyceI}
D.~Joyce.
\newblock {Configurations in abelian categories. I. Basic properties and moduli
  stacks}.
\newblock {\em Adv. Math.}, 203:194--255, 2006.
\newblock math.AG/0312190.

\bibitem{JoyceII}
D.~Joyce.
\newblock {Configurations in abelian categories. II. Ringel--Hall algebras}.
\newblock {\em Adv. Math.}, 210:635--706, 2007.
\newblock math.AG/0503029.

\bibitem{JoyceIII}
D.~Joyce.
\newblock {Configurations in abelian categories. III. Stability conditions and
  identities}.
\newblock {\em Adv. Math.}, 215:153--219, 2007.
\newblock math.AG/0410267.

\bibitem{JoyceMF}
D.~Joyce.
\newblock {Motivic invariants of Artin stacks and `stack functions'}.
\newblock {\em Quarterly Journal of Mathematics}, 58, 2007.
\newblock math.AG/0509722.

\bibitem{JoyceIV}
D.~Joyce.
\newblock {Configurations in abelian categories. IV. Invariants and changing
  stability conditions}.
\newblock {\em Adv. Math.}, 217:125--204, 2008.
\newblock math.AG/0503029.

\bibitem{JoyceDT}
D.~Joyce and Y.~Song.
\newblock {A theory of generalized Donaldson--Thomas invariants}.
\newblock {\em Mem. Amer. Math. Soc.}, 217(1020), 2012.
\newblock math.AG/08105645.

\bibitem{KSsheaves}
M.~Kashiwara and P.~Schapira.
\newblock {\em {Sheaves on manifolds, volume 292 of Grundlehren der
  Mathematischen Wissenschaften [Fundamental Principles of Mathematical
  Sciences]}}.
\newblock Springer-Verlag, Berlin, 1990.

\bibitem{King}
A.~D. King.
\newblock {Moduli of representations of finite-dimensional algebras}.
\newblock {\em Quart. J. Math. Oxford Ser.}, (2) 45, no. 180, 1994.

\bibitem{KS1}
M.~Kontsevich and Y.~Soibelman.
\newblock {Stability structures, motive Donaldson--Thomas invariants and
  cluster transformations}.
\newblock 2008.
\newblock math.AG/08112435.

\bibitem{KS3}
M.~Kontsevich and Y.~Soibelman.
\newblock {Motivic Donaldson--Thomas invariants: summary of results}.
\newblock In {\em Mirror symmetry and tropical geometry}, volume 527 of {\em
  Contemp. Math.}, pages 55--89. Amer. Math. Soc., Providence, RI, 2010.

\bibitem{KS2}
M.~Kontsevich and Y.~Soibelman.
\newblock {Cohomological Hall algebra, exponential Hodge structures and motivic
  Donaldson--Thomas invariants}.
\newblock {\em Commun. Number Theory Phys.}, 5, 2011.
\newblock arXiv:1006.2706.

\bibitem{Kul98}
S.~Kulikov.
\newblock {\em {Mixed Hodge structures and singularities}}, volume 132.
\newblock Cambridge University Press, 1998.

\bibitem{Ma01}
D.~Massey.
\newblock {The Sebastiani--Thom isomorphism in the Derived Category}.
\newblock {\em Comp. Math.}, 125(3):353--362, 2001.

\bibitem{Ma09}
D.~Massey.
\newblock {Natural commuting of vanishing cycles and the Verdier dual}.
\newblock {\em Pac. J. Math.}, 284(2):431--437, 2016.

\bibitem{MSS11}
L.~Maxim, M.~Saito, and J.~Sch{\"u}rmann.
\newblock {Symmetric products of mixed Hodge modules}.
\newblock {\em Journal de math{\'e}matiques pures et appliqu{\'e}es},
  96(5):462--483, 2011.

\bibitem{Meinhardt14}
S.~Meinhardt and M.~Reineke.
\newblock {Donaldson--Thomas invariants versus intersection cohomology of
  quiver moduli}.
\newblock 2014.
\newblock arXiv:1411.4062.

\bibitem{MFK94}
D.~Mumford, J.~Fogarty, and F.~Kirwan.
\newblock {\em Geometric invariant theory}, volume~34.
\newblock Springer Science \& Business Media, 1994.

\bibitem{Reineke_HN}
M.~Reineke.
\newblock The {H}arder-{N}arasimhan system in quantum groups and cohomology of
  quiver moduli.
\newblock {\em Invent. Math.}, 152(2):349--368, 2003.

\bibitem{Ri13}
R.~Rim{\'a}nyi.
\newblock {On the cohomological Hall algebra of Dynkin quivers}.
\newblock {\em arXiv preprint arXiv:1303.3399}, 2013.

\bibitem{Saito10}
M.~Saito.
\newblock {Thom--Sebastiani Theorem for Hodge Modules}.
\newblock {\em preprint 2010}.

\bibitem{Sai88}
M.~Saito.
\newblock {Modules de Hodge polarisables}.
\newblock {\em {Publications of the Research Institute for Mathematical
  Sciences}}, 24(6):849--995, 1988.

\bibitem{Sai89duality}
M.~Saito.
\newblock Duality for vanishing cycle functors.
\newblock {\em Publications of the Research Institute for Mathematical
  Sciences}, 25(6):889--921, 1989.

\bibitem{Saito1}
M.~Saito.
\newblock {Introduction to mixed Hodge modules}.
\newblock {\em {Ast\'{e}risque}}, 179-180:145--162, 1989.

\bibitem{Saito89}
M.~Saito.
\newblock {Mixed Hodge modules and admissible variations}.
\newblock {\em {CR Acad. Sci. Paris}}, 309(6):351--356, 1989.

\bibitem{Saito90}
M.~Saito.
\newblock {Mixed Hodge modules}.
\newblock {\em {Publ. Res. Inst. Math.}}, 26:221--333, 1990.

\bibitem{Sch85}
J.~Scherk and J.~Steenbrink.
\newblock {On the mixed Hodge structure on the cohomology of the Milnor fibre}.
\newblock {\em Mathematische Annalen}, 271(4):641--665, 1985.

\bibitem{StZu85}
J.~Steenbrink and S.~Zucker.
\newblock {Variation of mixed Hodge structure. I}.
\newblock {\em Invent. Math.}, 80(3):489--542, 1985.

\bibitem{Thomas1}
R.P. Thomas.
\newblock A holomorphic casson invariant for {Calabi--Yau} 3-folds, and bundles
  on {K3} fibrations.
\newblock {\em J. Diff. Geom.}, 54:367--438, 2000.
\newblock math.AG/9806111.

\bibitem{Toda17-2}
Y.~Toda.
\newblock {Moduli stacks of semistable sheaves and representations of
  Ext-quivers}.
\newblock 2017.
\newblock arXiv:1710.01841.

\bibitem{To99}
B.~Totaro.
\newblock The chow ring of a classifying space.
\newblock In {\em Proceedings of Symposia in Pure Mathematics}, volume~67,
  pages 249--284. Providence, RI; American Mathematical Society; 1998, 1999.

\end{thebibliography}

\vfill

\textsc{\small B. Davison: School of Mathematics, University of Edinburgh, James Clerk Maxwell Building, Peter Guthrie Tait Road, King's Buildings, Edinburgh EH9 3FD, United Kingdom}\\
\textit{\small E-mail address:} \texttt{\small ben.davison@ed.ac.uk}\\
\\

\textsc{\small S. Meinhardt: School of Mathematics and Statistics, Hicks Building, Hounsfield Road, Sheffield S3 7RH, United Kingdom}\\
\textit{\small E-mail address:} \texttt{\small s.meinhardt@shef.ac.uk}\\

\end{document}